\documentclass[a4paper,11pt]{amsart}
\usepackage{amsthm,amsfonts,amsbsy,amssymb,amsmath,amscd}
\usepackage[all]{xy}
\usepackage[colorlinks=true]{hyperref}
\usepackage{pgf, tikz-cd}
\usepackage{tikz}
\usetikzlibrary{decorations.pathreplacing}

\newcommand{\rounddown}[1]{\lfloor{#1}\rfloor}

\newcommand{\E}{\ensuremath{\mathcal{E}}}

\newcommand{\M}{\ensuremath{\mathcal{M}}}

\newcommand{\CC}{\ensuremath{\mathbb{C}}}
\newcommand{\FF}{\ensuremath{\mathbb{F}}}

\newcommand{\NN}{\ensuremath{\mathbb{N}}}
\newcommand{\PP}{\ensuremath{\mathbb{P}}}
\newcommand{\QQ}{\ensuremath{\mathbb{Q}}}
\newcommand{\RR}{\ensuremath{\mathbb{R}}}
\newcommand{\ZZ}{\ensuremath{\mathbb{Z}}}

\newcommand{\s}{\ensuremath{\mathfrak{s}}}

\def\hol{{\mathcal{O}}}

\DeclareMathOperator{\Aut}{Aut}

\DeclareMathOperator{\Ext}{Ext}

\DeclareMathOperator{\Hom}{Hom}

\DeclareMathOperator{\Sym}{Sym}

\newtheorem{thm}{Theorem}[section]
\newtheorem{lem}[thm]{Lemma}

\newtheorem{cor}[thm]{Corollary}
\newtheorem{prop}[thm]{Proposition}
\newtheorem{conj}[thm]{Conjecture}

\theoremstyle{definition}
\newtheorem{df}[thm]{Definition}

\theoremstyle{remark}
\newtheorem{remark}[thm]{Remark}
\newtheorem{example}[thm]{Example}

\numberwithin{equation}{section}

\begin{document}

\title[Moduli spaces of threefolds on the Noether line]{Moduli spaces of threefolds \\ on the Noether line}

\author{Stephen Coughlan}
\author{Yong Hu}
\author{Roberto Pignatelli}
\author{Tong Zhang}
\date{\today}
	
\address[S.C.]{Department of Mathematics and Computer Studies, Mary Immaculate College, South Circular Road, Limerick (Ireland)}
\email{stephen.coughlan@mic.ul.ie}

\address[Y.H.]{School of Mathematical Sciences, Shanghai Jiao Tong University, 800 Dongchuan Road, Shanghai (China)}
\email{yonghu@sjtu.edu.cn}

\address[R.P.]{Dipartimento di Matematica dell'Universit\`a di Trento, via Sommarive 14, I-38123, Trento (Italy)}
\email{Roberto.Pignatelli@unitn.it}

\address[T.Z.]{School of Mathematical Sciences, Ministry of Education Key Laboratory of Mathematics and Engineering Applications \& Shanghai Key Laboratory of PMMP, East China Normal University, Shanghai (China)}
\email{tzhang@math.ecnu.edu.cn, mathtzhang@gmail.com}

\begin{abstract}
    In this paper, we study the moduli spaces of canonical threefolds with any prescribed geometric genus $p_g \ge 5$ which have the smallest possible canonical volume. This minimal volume is equal to the smallest half-integer that is larger than or equal to $\frac43 p_g -\frac{10}3$, and the threefolds in question are said to lie on the (refined) Noether line. For every such moduli space, we establish an explicit stratification, compute the dimension of all strata, and estimate the number of its irreducible components. Thus it yields a complete classification of threefolds on the (refined) Noether line. A new and unexpected phenomenon is that the number of irreducible components of the moduli space grows linearly with $p_g$, while the moduli space of canonical surfaces on the Noether line with any prescribed geometric genus has at most two irreducible components.
    
    The key idea in the proof is to relate these canonical threefolds $X$ to simple fibrations in $(1, 2)$-surfaces. In turn, this depends on the observation that a general member in $|K_X|$ is a canonical surface on the Noether line.
\end{abstract}

	
\maketitle

\setcounter{tocdepth}{1}
\tableofcontents

\section{Introduction} 

\subsection{Background} 

One of the most fundamental problems in algebraic geometry is to classify algebraic varieties, with probably the ultimate goal to understand the moduli space of varieties with prescribed discrete numerical invariants. As a typical example, the moduli spaces $\M_g$ of smooth curves of genus $g \ge 2$ have been extensively studied since the seminal work of Mumford. In the moduli theory for higher dimensional varieties of general type, the main objects are varieties with ample canonical class and canonical singularities \cite[\S 1.2]{Kollar}. Geometric invariant theory (GIT) can be applied to construct a quasi-projective coarse moduli space of such varieties \cite{Viehweg} (see also \cite{Gieseker} for surfaces). An alternative construction using the minimal model program (MMP) was outlined for surfaces in \cite{KSB} (see also \cite{Alexeev}), and it gives a projective moduli space by adding stable varieties (see \cite{Kollar} for details including the higher dimensional case). 
However, the geometry of these moduli spaces is still far from being understood, even without considering the locus parameterizing strictly stable varieties. The basic questions include, for example:
\begin{itemize}
	\item the non-emptiness of the moduli space of varieties of general type with prescribed birational invariants;
	
	\item the dimension and the number of irreducible/connected components of the moduli space, if it is non-empty.
\end{itemize}

In this paper, we describe the explicit geometry of moduli spaces of a class of threefolds with ample canonical class, which are of special importance from the viewpoint of the geography of algebraic varieties. To motivate our result, in the following, we assume that $X$ is a variety of general type of dimension $n \ge 2$ with at worst canonical singularities. If the canonical class $K_X$ is ample, then $X$ is called {\em canonical}. Let 
$$
\mathrm{Vol}(X) := \limsup\limits_{m \to \infty} \frac{h^0(X, mK_X)}{m^n/n!}
$$
denote the \emph{canonical volume} of $X$, and let 
$$
p_g(X) := h^0(X, K_X)
$$ 
denote its \emph{geometric genus}. These two numerical invariants are fundamental in the study of the birational geometry of $X$. Note that if $K_X$ is nef, then $\mathrm{Vol}(X) = K_X^n$. 

When $n = 2$, the famous inequality due to M. Noether \cite{Noether} states that
$$
\mathrm{Vol}(X) \ge 2p_g(X) - 4
$$
for every surface $X$ of general type. Surfaces satisfying the above equality are usually said to be on the Noether line, and the study of such surfaces dates back to the work of Enriques \cite{Enriques}. They are also known as Horikawa surfaces since in his celebrated paper \cite{Horikawa1}, Horikawa completely described for each possible $p_g \ge 3$ the moduli space parameterizing canonical surfaces on the Noether line. More precisely, he showed that the moduli space is either irreducible and unirational, or it has two unirational irreducible components of the same dimension that do not intersect. Horikawa computed the dimension of each component as well.

When $n = 3$, the corresponding Noether inequality, conjectured around the end of the last century, is now proved. More precisely, Chen et al. proved in \cite{CCJ,CCJAdd,CHJ} that the inequality 
\begin{equation} \label{eq: Noether1}
    \mathrm{Vol}(X) \ge \frac{4}{3} p_g(X) - \frac{10}{3}
\end{equation}
holds for every threefold $X$ of general type. The inequality is optimal due to known examples found by Kobayashi \cite{Kobayashi} for infinitely many but not all $p_g$. However, combining with results in \cite{HZ}, it can be refined as
\begin{equation} \label{eq: Noether}
    \mathrm{Vol}(X) \ge \frac{1}{2} \left \lceil {\frac{8p_g(X) - 20}{3}} \right \rceil \in \frac12 {\mathbb N},
\end{equation}
with the term on the right-hand side being the smallest half-integer larger than or equal to $\frac43 p_g(X) - \frac{10}3$ (see Theorem \ref{thm: Noether}). This refined inequality is optimal for every $p_g \ge 3$ due to infinitely many examples constructed in \cite{CP,HZ}, and it naturally splits into three distinct Noether inequalities, subject to the residue of $p_g$ modulo $3$.

We say that a threefold $X$ is \emph{on the refined Noether line} (see Definition \ref{def: Noether line}) if it satisfies the equality in the above \eqref{eq: Noether}. In other words, threefolds on the refined Noether line with prescribed $p_g \ge 3$ have the smallest possible canonical volume. Recently, more examples of threefolds on the (refined) Noether line have been constructed in \cite{CH,CP,CJL}, but the question of whether there is a classification of all threefolds on the (refined) Noether line (see \cite[Question 1.5]{CCJ}) has remained open until now.

\subsection{Main theorem}

The main result in this paper is an explicit description of the moduli spaces of canonical threefolds on the refined Noether line with geometric genus $p_g \ge 5$. 
It can be seen as a three dimensional version of Horikawa's result \cite{Horikawa1} and provides a complete answer to the above question when $p_g \ge 5$. We summarize it as the following two theorems.

\begin{thm} \label{thm: main1}
	For an integer $p_g \ge 13$, let $\M_{K^3,p_g}$ be the coarse moduli space parameterizing all canonical threefolds on the refined Noether line with geometric genus $p_g$. Let $N \in \{0, 1, 2\}$ such that $N \equiv p_g + 2$ $(\mathrm{mod} \, 3)$. Then 
	\begin{itemize}
		\item [(1)] $\M_{K^3,p_g}$ is a union of $\alpha_{p_g}$ unirational strata, where
        $$
        \alpha_{p_g} = 
        \begin{cases}
            \left\lfloor \frac{p_g+6}{4} \right\rfloor, & \text{ if } N = 0, 2;\\[5pt]
            \left\lfloor \frac{p_g + 8}{4} \right\rfloor, & \text{ if } N = 1.
        \end{cases}
        $$
		
		\item [(2)] The number $\nu_{p_g}$ of irreducible components is at most $\alpha_{p_g}$ and at least $\alpha_{p_g} - \beta_{p_g}$, where the value of $\beta_{p_g}$ is given in the following table.
        \begin{table}[htbp]
            \centering
            \renewcommand{\arraystretch}{1.4}
            \begin{tabular}{cccccc}
            & $N = 0$ & $N = 1$
            & $N = 2$
            \\
            \hline
            $\beta_{p_g}$ & $\left\lfloor \frac{p_g + 8}{78} \right\rfloor$ & $\left\lfloor \frac{p_g + 61}{78} \right\rfloor$ & $\left\lfloor \frac{p_g + 36}{78} \right\rfloor$
            \end{tabular}
        \end{table}
        \noindent In particular, $\nu_{p_g}$ grows linearly with $p_g$, as $p_g/4$.
		
		\item [(3)] $\M_{K^3, p_g}$ is not equidimensional, and its irreducible component of maximal dimension has dimension
		$$
		\dim \M_{K^3, p_g} = \frac{169}{3}p_g - 56 \left \lceil \frac{p_g + 2 + 2N}{12} \right \rceil + \frac{386 - 10N}{3}.
		$$
	\end{itemize}
\end{thm}

In contrast with aforementioned Horikawa's results (and rather surprisingly for us), Theorem \ref{thm: main1} (2) shows that the number of irreducible components is unbounded as $p_g$ tends to infinity. Moreover, we obtain not only the dimension of $\M_{K^3, p_g}$ as in Theorem \ref{thm: main1} (3) but also dimensions of all strata of those in Theorem \ref{thm: main1} (1) (see Propositions \ref{prop: dimension moduli}, \ref{prop: dimension moduli N=1 d=2}, \ref{prop: dimension moduli N=2} for details). 

If $5 \le p_g \le 12$, the following theorem gives a more concrete description of the corresponding moduli space of threefolds on the refined Noether line.

\begin{thm} \label{thm: main2}
    For an integer $5 \le p_g \le 12$, let $\M_{K^3,p_g}$ be the coarse moduli space parameterizing all canonical threefolds on the refined Noether line with geometric genus $p_g$. Then $\M_{K^3,p_g}$ consists of $\nu_{p_g}$ unirational irreducible components, where $\nu_{p_g}$ and the dimensions of each irreducible component are given in the following table.
    \begin{table}[htbp]
            \centering
            \begin{tabular}{ccc}
            $p_g$ & $\nu_{p_g}$ & \emph{dimensions} \\
            \hline
            $5$ & $2$ & $305,309$ \\
            $6$ & $2$ & $341,357$ \\
            $7$ & $2$ & $391,417$ \\
            $8$ & $3$ & $427,430,468$
            \end{tabular}
            \quad\quad
            \begin{tabular}{ccc}
            $p_g$ & $\nu_{p_g}$ & \emph{dimensions} \\
            \hline
            $9$ & $3$ & $463,476,520$ \\
            $10$ & $3$ & $513,536,582$ \\
            $11$ & $4$ & $549,551,585,634$ \\
            $12$ & $4$ & $585,596,636,687$
            \end{tabular}
    \end{table}
\end{thm}

The moduli spaces $\M_{K^3,p_g}$ of the canonical threefolds on the refined Noether line with $p_g = 3$, $4$ have been investigated in \cite{CHJ}. In both cases, the corresponding moduli spaces are irreducible, and a general member in the moduli has only one terminal singularity of type $\frac{1}{2}(1,1,1)$ when $p_g = 3$, and is even smooth when $p_g = 4$. However, Theorem \ref{thm: main1} and \ref{thm: main2} reveal new phenomena when $p_g \ge 5$. More precisely, the moduli space $\M_{K^3,p_g}$ of canonical threefolds on the refined Noether line with $p_g \ge 5$ is never equidimensional (thus always reducible). Consider the (unique) irreducible component of $\M_{K^3,p_g}$ with the maximal dimension. Then a general member in it has non-isolated canonical singularities of type $cE_8$ when $p_g \ge 6$ and type $cA_1$ when $p_g = 5$ (see the tables in \S \ref{section: moduli} for details). In the case $N=0$, this gives a lot of examples of non-smoothable canonical threefolds whose singularities are locally smoothable. This also differs dramatically from the surface case \cite{Horikawa1}, where a general canonical surface on the Noether line is always smooth.

\subsection{Idea of the proof}

The proof of the main theorems begins with investigating the following birational version of a conjecture stated in \cite[Introduction]{CP}.

\begin{conj} \label{conj}
    There exists an $\varepsilon > 0$ such that every canonical threefold $X$ with $K_X^3 < \frac{4}{3} p_g(X) - \frac{10}{3} + \varepsilon$ and $p_g(X) \gg 1$ birationally admits a simple fibration in $(1, 2)$-surfaces over $\PP^1$.
\end{conj}

Here and throughout this paper, a \emph{$(1, 2)$-surface} is a surface $S$ with at worst canonical singularities, $\mathrm{Vol}(S) = 1$ and $p_g(S) = 2$. A key feature of a $(1, 2)$-surface is that its canonical ring is generated by four elements of respective degree $1$, $1$, $2$ and $5$ and related by a single equation of degree $10$. Simple fibrations in $(1,2)$-surfaces were introduced and studied in \cite{CP} (see Definition \ref{def: simple fibration} for a precise definition). They are fibrations $f \colon X \to B$ from a threefold $X$ with canonical singularities to a smooth curve $B$ with $K_X$ being $f$-ample such that the canonical ring of each fibre is ``algebraically" like that of a $(1, 2)$-surface. An enlightening result proved in \cite{CP} is that every Gorenstein minimal threefold $X$ admitting a simple fibration in $(1, 2)$-surfaces over $\PP^1$ satisfies $K_X^3 = \frac{4}{3}p_g(X) - \frac{10}{3}$. Thus Conjecture \ref{conj} is a generalization of the converse of the above result. 

The first step in the proof of the main theorems is to confirm the above conjecture in an effective way.

\begin{thm} [See Corollary \ref{cor: simple fibration}] \label{thm: main3}
    Up to a crepant birational morphism, every canonical threefold on the refined Noether line with $p_g \ge 5$ admits a simple fibration in $(1, 2)$-surfaces over $\PP^1$.
\end{thm}

Combining this theorem with the results in \cite{HZ} and \cite{CHJ}, it follows that Conjecture \ref{conj} holds for $\varepsilon = \frac{1}{2}$ when $p_g \ge 11$, and for $\varepsilon = \frac{1}{30}$ when $p_g \ge 5$. Moreover, it does not hold for $p_g = 4$ since a general canonical threefold on the Noether line with $p_g = 4$ has no pencil of $(1, 2)$-surfaces, see Remark \ref{rem: pg=4}.

The original biregular version of the conjecture, the one in \cite[Introduction]{CP}, claimed the existence of a simple fibration in $(1,2)$-surfaces directly on the canonical model. This also follows by Corollary \ref{cor: simple fibration}, but for bigger $p_g$. That is, for $\varepsilon = \frac{1}{2}$ we need $p_g\ge 23$. In \S \ref{subsection: index 3}, we construct canonical threefolds of index three showing that the (birational) conjecture does not hold for any $\varepsilon > \frac23$. 

Now we explain the strategy of the proof of Theorem \ref{thm: main3}. Let $X$ be a canonical threefold on the refined Noether line with $p_g(X) \ge 5$. If $p_g(X) \ge 11$, then it has been proved in \cite{HZ} that $X$ has a minimal model $X'$ which is fibred by $(1, 2)$-surfaces over $\PP^1$. Thanks to \cite{CHJ}, such a result can be extended to the case when $p_g(X) \ge 5$ (see Theorem \ref{thm: existence of (1,2)-surface fibration}). Let $X_0$ be the relative canonical model of $X'$ over $\PP^1$. The main technical difficulty is to prove that the fibration $f_0\colon X_0 \to \PP^1$ is a simple fibration. That is, to determine the canonical ring of every fibre. To overcome this, our main discovery is that the Cartier index of $X_0$ is at most two and that a general member of $|K_{X_0}|$ is a canonical surface on the Noether line (see Theorem \ref{thm: Horikawa surface}). By Horikawa's work on the classification of fibrations by curves of genus two \cite{Horikawagenus2}, we deduce that a general member of $|K_{F_p}|$ for any fibre $F_p$ of $f_0$ is a Gorenstein integral curve of arithmetic genus two. With such a nice canonical curve, the canonical ring of $F_p$ can be computed via the method in \cite{CFPR,FPRb}.

Given Theorem \ref{thm: main3}, in the second step of the whole proof, we focus on threefolds $X$ admitting simple fibrations in $(1, 2)$-surfaces over $\PP^1$. To such a threefold $X$ we associate a triple of integers $(d, N, d_0)$. Here $N = 6K_X^3 - 8p_g(X) + 20 \ge 0$. The novelty here is to show that

\begin{thm} [See Theorem \ref{thm: X is hypersurface}] \label{thm: main4}
    If $N \le 4$, then $X$ is isomorphic to a hypersurface in a toric fourfold  uniquely determined by the triple $(d, N, d_0)$ with an explicit defining equation.
\end{thm}

Note that $X$ lying on the refined Noether line implies $N \le 2$. Thus by Theorem \ref{thm: main3}, every threefold $X$ on the refined Noether line with $p_g(X) \ge 5$ is isomorphic to a divisor in a toric fourfold, and in this case, we have $p_g(X) = 3d - 2 + N$. Moreover, as a general hypersurface, $X$ has at worst canonical singularities if and only if $\frac{1}{4}(d + N) \le d_0 \le \frac{1}{2}(3d + N)$, and $d_0$ determines the singularities on $X$ (see Proposition \ref{prop: sing-X-appendix}). Roughly speaking, the smaller $d_0$ is, the more singular $X$ is.

We note that the assumption $N \le 4$ in Theorem \ref{thm: main4} is optimal, as we give in \S \ref{subsection: N=5} an example of a simple fibration in $(1, 2)$-surfaces over $\PP^1$ with $N = 5$ that is not isomorphic to a hypersurface in any toric fourfold of those considered in Theorem \ref{thm: X is hypersurface}.

In the final step of the entire proof, we study the modular family $\M^N_d(d_0)$ of hypersurfaces $X(d, N; d_0)$ in $\FF(d, N; d_0)$ with the desired degree for $N \le 2$. By Theorem \ref{thm: main4}, every $\M^N_d(d_0)$ maps (finite-to-one) to $\M_{K^3, p_g}$, the moduli space of canonical threefolds on the refined Noether line with $p_g = 3d - 2 + N \ge 5$, and the images $V_d^N(d_0)$ of $\M^N_d(d_0)$ give rise to a stratification of $\M_{K^3, p_g}$. Based on the explicit equation of $X(d, N; d_0)$, we manage to compute all dimensions of $\M^N_d(d_0)$, thus $V_d^N(d_0)$. Together with some deformation technique in \cite{Pignatelli}, we are able to show that every $V_d^N(d_0)$ is contained in the closure of $V^N_d\left( \left\lfloor \frac32 d + \frac N2 \right\rfloor \right)$ when $N = 0$ and $d_0 \ge d$ or when $N > 0$ and $d_0 \ge d + 1$. This gives one irreducible component of $\M_{K^3, p_g}$. On the other hand, for most $d_0 \le d$, the closure of $V_d^N(d_0)$ forms an irreducible component of $\M_{K^3, p_g}$ (see Theorem \ref{thm: moduli N=0}, \ref{thm: moduli N=1} and \ref{thm: moduli N=2} for details). Thus the number of irreducible components of $\M_{K^3, p_g}$ grows as $d$ (thus $p_g$) grows.

We summarize the geometric consequences of our classification. Suppose that $X$ is a canonical threefold on the refined Noether line, general in its stratum of the moduli space. We assume that $X$ is not one of the finite and small number of cases with $p_g \le 22$, for which we usually need a crepant blowup to realize the simple fibration. Then $X$ has the following properties:

\begin{enumerate}
	\item $X$ admits a simple fibration $f\colon X\to \PP^1$ in $(1,2)$-surfaces, and the $f$-relative canonical model of $X$ is isomorphic to $X$ itself.
	
	\item $X$ has $N$ singularities of type $\frac12(1,1,1)$ and possibly Gorenstein canonical singularities along a section of $f$. Thus $X$ is Gorenstein if $N=0$, and $2$-Gorenstein otherwise.
	
	\item The canonical map of $\varphi_{K}\colon X\dashrightarrow \Sigma$ is a rational map whose image is a Hirzebruch surface. The simple fibration is induced by the composition of $\varphi_K$ with the natural projection to $\PP^1$. The indeterminacy locus of $\varphi_K$ is a section $\sigma$ of $f$. For $p$ in $\PP^1$, the corresponding point $\sigma(p)$ is the basepoint of $|K_{F_p}|$, where $F_p := f^*p$.
	
	\item The bicanonical map $\varphi_{2K}\colon X\to\mathcal{Q}$ is a 2-to-1 morphism to a (toric) $\PP(1,1,2)$-bundle $\mathcal{Q}$ over $\PP^1$, branched along a surface of relative degree $10$ and the section of vertices $\sigma_{(2)}$. That is, $\sigma_{(2)}(p)$ is the point $(0:0:1)$ in $\mathcal{Q}_p \cong\PP(1,1,2)$. The branch surface intersects $\sigma_{(2)}$ in $N$ points.
	
	\item The general canonical surface section $S$ in $|K_X|$ is a Horikawa surface with canonical singularities. If $N=0$, then $K_S$ is $2$-divisible as a line bundle and $S$ is an even Horikawa surface.
\end{enumerate}

\subsection{Structure of the paper}
The paper is structured as follows.

In Section \ref{section: Noether line}, we recall the Noether and the refined Noether inequality for threefolds of general type obtained in \cite{CCJ,CCJAdd,CHJ,HZ}. The key result here is that every canonical threefold on the refined Noether line birationally admits a fibration in $(1, 2)$-surfaces over $\PP^1$, whose proof is in Appendix \ref{section: A1}.

Section \ref{section: fibration} is devoted to the study of threefolds fibred by $(1, 2)$-surfaces. The main result in this section is Theorem \ref{thm: Horikawa surface}, showing that a general canonical divisor is in fact a canonical surface on the Noether line.

In Section \ref{section: simple fibration}, we apply Theorem \ref{thm: Horikawa surface} to prove Theorem \ref{thm: main3}, verifying the aforementioned Conjecture \ref{conj}.

In Section \ref{section: toric 4fold}, we prove Theorem \ref{thm: main4}. Moreover, in Proposition \ref{prop: sing-X} (cf.~Appendix \ref{appendix: sing}) we also give an explicit description of singularities on threefolds $X$ admitting simple fibrations in $(1, 2)$-surfaces. 

In Section \ref{section: moduli}, we apply the results in Section \ref{section: toric 4fold} to study the stratification of the moduli space of canonical threefolds on the refined Noether line, obtaining  Theorems \ref{thm: main1} and \ref{thm: main2}. 

In Section \ref{section: example}, we provide more examples of fibrations in $(1, 2)$-surfaces that complement the main results. \S \ref{subsection: K non-nef} contains the classification of the simple fibrations in $(1, 2)$-surfaces over $\PP^1$ whose canonical class is not nef. This gives sporadic interesting examples of canonical threefolds with small volume and small genus, that are not in the moduli spaces described by Theorems \ref{thm: main1} and \ref{thm: main2}. In \S \ref{subsection: N=5}, we construct a simple fibration in $(1, 2)$-surfaces over $\PP^1$ which is not a hypersurface in a toric fourfold as in Theorem \ref{thm: main4}. Similar constructions are known for each $N \ge 5$. In \S \ref{subsection: index 3}, we produce canonical threefolds with arbitrarily high genus $p_g$ and canonical volume $\frac43 p_g -\frac83$, that have no simple fibration in $(1, 2)$-surfaces. They all have index three.

Finally, the appendices. Appendix \ref{section: A1} contains the proof that every canonical threefold on the refined Noether line birationally admits a fibration in $(1, 2)$-surfaces over $\PP^1$. 
Appendix \ref{appendix: sing} classifies the singularities occurring on simple fibrations in $(1,2)$-surfaces.

\subsection{Notation}
Throughout this paper, we work over the complex number field $\CC$, and all varieties are projective. 
\begin{itemize}
	\item A variety $X$ is {\em minimal} if it has at worst $\QQ$-factorial terminal singularities and $K_X$ is nef.
	
	\item For a normal variety $X$, if $p_g(X) \ge 2$, then the global sections of the canonical class induce a rational map, called the {\em canonical map}, from $X$ to $\PP^{p_g(X)-1}$. The closure of the image of $X$ under its canonical map is called the {\em canonical image} of $X$, whose dimension is called the \emph{canonical dimension} of $X$.

    \item Given two variables $t_0, t_1$, we denote by $S^n(t_0,t_1)$ the set of homogeneous polynomials of degree $n$ in the variables $t_0, t_1$. 
\end{itemize}

\subsection*{Acknowledgements} We would like to thank Jungkai Chen and Meng Chen for their interest in this problem.

Y.H. was supported by National Key Research and Development Program of China \# 2023YFA1010600 and the National Natural Science Foundation of China (Grant No. 12571044). R.P. was partially supported by the ``National Group for Algebraic and Geometric Structures, and their Applications" (GNSAGA - INdAM) and by the European Union-Next Generation EU, Mission 4 Component 2 - CUP E53D23005400001. T.Z. was partially supported by the National Natural Science Foundation of China (Grant No. 12071139), the Science and Technology Commission of Shanghai Municipality (No. 22JC1400700, No. 22DZ2229014) and the Fundamental Research Funds for the Central Universities.


\section{Noether inequality and the refined Noether line} \label{section: Noether line}
In this section, we collect some known results about threefolds with small volume. 

We are interested in the moduli space of canonical threefolds. Some of the following results we use are stated in the original papers for minimal threefolds of general type, but these results extend to canonical threefolds by the obvious use of a terminalisation. Indeed, for a canonical threefold $X$, there exists a crepant birational morphism $\tau \colon \tilde{X} \to X$ such that $\tilde{X}$ is minimal by \cite{Kaw88} or \cite[Theorem 6.25]{KM}. Therefore, we reformulate these results directly here for canonical threefolds. 

We start from the Noether inequality for threefolds of general type, which is an accumulation of \cite[Theorem 1.1]{CCJ}, \cite[Theorem 1]{CCJAdd} and \cite[Theorem 1.1]{CHJ}.
\begin{thm}[Noether inequality] \label{thm: Noether1}
	Let $X$ be a canonical threefold. Then the inequality \eqref{eq: Noether1}
	\[
	K_X^3 \ge \frac{4}{3}p_g(X) - \frac{10}{3}
	\]
    holds.
\end{thm}
The inequality \eqref{eq: Noether1} is indeed optimal for infinitely many $p_g$ (see \cite{Kobayashi,CH,CP,CJL,HZ} for examples for which the inequality becomes an equality). However, it is shown in \cite[Theorem 1.2]{HZ} that if the equality in \eqref{eq: Noether1} holds, then $p_g \equiv 1$ $(\mathrm{mod}\, 3)$. It turns out that, combining with results in \cite{HZ}, we actually have the following refined Noether inequality.

\begin{thm}[Refined Noether inequality] \label{thm: Noether}
	Let $X$ be a canonical threefold. Then the inequality \eqref{eq: Noether}
    \[
	K_X^3 \ge \frac{1}{2} \left \lceil {\frac{8p_g(X) - 20}{3}} \right \rceil
    \]
    holds.
\end{thm}

\begin{proof}
    To prove this inequality, we may assume that $p_g(X) \ge 3$. When $p_g(X) \le 4$, the inequality follows from \cite[Theorem 1.5]{Chen}. When $p_g(X) \ge 5$, by \cite[Theorem 2.4]{Kobayashi}, \cite[Theorem 4.4 and 4.5]{CCJ} and \cite[Theorem 4.6]{CHJ}, we only need to treat the case when the canonical image $\Sigma$ of $X$ is a surface. In this case, by  Lemma \ref{lem: p_g5} and \cite[Proposition 2.1]{HZ}, we may further assume that $X$ admits a fibration over $\PP^1$ with general fibre a $(1,2)$-surface. Then we are under the setting of \cite[\S 3]{HZ}, and the inequality follows from \cite[Proposition 3.5(2)]{HZ} (note that we have $d \geq \deg\Sigma \geq p_g(X)-2$).
\end{proof}

Equivalently, as is stated in Theorem \ref{thm: Noether}, suppose that $X$ is a canonical threefold.
\begin{itemize}
    \item [(1)] If $p_g \equiv 1$ $(\mathrm{mod}\, 3)$, then $K_X^3 \ge \frac{4}{3}p_g(X) - \frac{10}{3}$;

    \item [(2)] If $p_g \equiv 2$ $(\mathrm{mod}\, 3)$, then $K_X^3 \ge \frac{4}{3}p_g(X) - \frac{19}{6}$;

    \item [(3)] If $p_g \equiv 0$ $(\mathrm{mod}\, 3)$, then $K_X^3 \ge \frac{4}{3}p_g(X) - 3$.
\end{itemize}
The key difference from \eqref{eq: Noether1} is that, by the examples constructed in \cite{CP,HZ}, the refined Noether inequality \eqref{eq: Noether} is optimal for all $p_g \ge 3$. 

\begin{df} \label{def: Noether line}
	For a canonical threefold $X$ with $p_g(X) \ge 3$, we say that it is \emph{on the refined Noether line}, if 
    $$
    K_X^3 = \frac{1}{2} \left \lceil {\frac{8p_g(X) - 20}{3}} \right \rceil.
    $$
\end{df}

Clearly, the above equality means three distinct equalities subject to the residue of $p_g$ modulo $3$, which in turn give rise to three distinct Noether lines (they are called the first, second and third Noether lines in \cite{HZ}). However, in the current paper, we will use the above equality to unify the three lines as one ``refined line", just because it works for all $p_g \ge 3$ and involves less notation.

As is discovered in \cite{HZ} for $p_g \ge 11$ as well as in \cite{CHJ} for $5 \le p_g \le 10$, canonical threefolds on the refined Noether line with $p_g \ge 5$ satisfy the following geometric property.

\begin{thm} \label{thm: fibration exsits}
    Let $X$ be a canonical threefold on the refined Noether line with $p_g(X) \ge 5$. Then the canonical dimension of $X$ is two. Moreover, it has a birational minimal model $X_1$ such that $X_1$ admits a fibration over $\PP^1$ whose general fibre is a smooth $(1, 2)$-surface.
\end{thm}

\begin{proof}
    See Theorem \ref{thm: existence of (1,2)-surface fibration} for the proof.
\end{proof}

As we will see in the sequel, the structure of the fibration in $(1, 2)$-surfaces completely determines the geometry of the threefolds on the refined Noether line.

We remark that the assumption that $p_g \ge 5$ in Theorem \ref{thm: fibration exsits} is also optimal. In fact, by \cite[Theorem 1.5]{CHJ}, a general canonical threefold on the Noether line with $p_g=4$ is a double cover over $\PP^3$. In particular, it has canonical dimension three and has no pencils of $(1, 2)$-surfaces. Meanwhile, by \cite[Example 3.1]{CH2} and \cite[Theorem 1.6]{CHJ}, a general canonical threefold on the refined Noether line with $p_g = 3$ does not have pencils of $(1, 2)$-surfaces, either.


\section{Threefolds fibred by \texorpdfstring{$(1,2)$}{(1, 2)}-surfaces with small volume} \label{section: fibration}

In this section, we always assume that $X$ is a minimal threefold of general type with $p_g(X) \ge 5$ such that 
\begin{itemize}
	\item [(1)] the canonical dimension is two;
	\item [(2)] $X$ admits a fibration $f\colon X \to \PP^1$ with general fibre $F$ a $(1, 2)$-surface.
\end{itemize}

\subsection{General setting}\label{setting}
In this subsection, we study the canonical map of $X$. We first recall some results in \cite[\S 3]{HZ} and refer the interested reader to loc.~cit.~for more details.

Let $\phi_{K_X}\colon X \dashrightarrow \PP^{p_g(X) - 1}$ be the canonical map of $X$ whose image is a surface $\Sigma$. As in \cite[\S 3.1]{HZ}, we may take a birational modification $\pi\colon X' \to X$ such that $\pi$ is an isomorphism over the smooth locus of $X$ and that $|M| = \mathrm{Mov}|\pi^*K_X|$ is base point free. Write
$$
\pi^*K_X = M + Z,
$$
where $Z \ge 0$ is a $\mathbb{Q}$-divisor. Then we have the following commutative diagram
$$
\xymatrix{
	& & X' \ar[d]_{\pi} \ar[lld]_{f'} \ar[rr]^{\psi} \ar[drr]^{\phi_{M}} & &  \Sigma' \ar[d]^{\tau}  \\
	\PP^1 & & X  \ar@{-->}[rr]_{\phi_{K_X}} \ar[ll]^f  & & \Sigma
}
$$
where $\phi_M$ is the morphism induced by $|M|$, $X' \stackrel {\psi}\rightarrow \Sigma' \stackrel{\tau} \rightarrow \Sigma$ is the Stein factorization of $\phi_M$, and $f' = f \circ \pi$ is the induced fibration. Denote by $F'$ a general fibre of $f'$. Furthermore, since $X$ has at worst terminal singularities, we may write
$$
K_{X'} = \pi^*K_X + E_\pi,
$$
where $E_\pi \ge 0$ is a $\pi$-exceptional $\mathbb{Q}$-divisor.

Take a general member $S\in |M|$. By Bertini's theorem, $S$ is a smooth surface of general type.  Let $C$ be a general fibre of $\psi$. By \cite[Lemma 3.1]{HZ}, $C$ is a smooth curve of genus $2$. We have
$$
M|_S\equiv dC,
$$
where $d = (\deg\tau) \cdot (\deg\Sigma) \ge p_g(X) - 2$. As in \cite[(3.3)]{HZ}, we may write
\begin{equation} \label{eq: E|_SZ|_S}
	E_\pi|_S = \Gamma_S + E_V, \quad Z|_S = \Gamma_S + Z_V,
\end{equation}
where $\Gamma_S$ is a section of the fibration $\psi|_S\colon S\to \PP^1$, $E_V$ and $Z_V$ are effective divisors which are vertical with respect to $\psi|_S$. By the adjunction formula, we have
\begin{align}\label{eq: adjunction}
	K_S = (K_{X'} + S)|_S = (2M + E_\pi + Z)|_S \equiv 2dC + 2\Gamma_S + E_V + Z_V.
\end{align}

Denote by $\sigma\colon S \to S_0$ the contraction onto the minimal model of $S$. By the proof in \cite[Proposition 3.5]{HZ}, the fibration $\psi|_S \colon S\to \PP^1$ descends to a fibration $S_0\to \PP^1$. Let $C_0=\sigma_*(C)$ and $\Gamma_{S_0}=\sigma_*(\Gamma_S)$. Then $g(C_0) = 2$ and $\Gamma_{S_0}$ is a section of the fibration $S_0 \to \PP^1$. By \eqref{eq: adjunction}, we have
\begin{align}\label{eq: K_S0}
	K_{S_0}\equiv 2dC_0+2\Gamma_{S_0}+\sigma_*(E_V+Z_V).
\end{align}
As in \cite[(3.5)]{HZ}, we may write
\begin{align}\label{eq: extension}
	(\pi^*K_X)|_S \sim_{\mathbb{Q}} \frac{1}{2}\sigma^*K_{S_0} + H,
\end{align}
where $H \ge 0$ is a $\QQ$-divisor. 

The following proposition follows from the proof of \cite[Proposition 3.5]{HZ}.
\begin{prop}\label{prop: estimate on S}
	The following (in)equalities hold:
	\begin{enumerate}
		\item[(1)]  $(K_{S_0} \cdot \Gamma_{S_0}) = -2 + \frac{1}{3}\left(2d + 2 + \left(\Gamma_{S_0} \cdot \sigma_*(E_V + Z_V) \right) \right)$;
		
		\item[(2)] $K_{S_0}^2 = 4d + 2(K_{S_0} \cdot \Gamma_{S_0}) + \left(K_{S_0} \cdot \sigma_*(E_V+Z_V)\right)$;
		\item[(3)] $\left(\left(\pi^*K_X\right)|_S \cdot \sigma^*K_{S_0}\right) = 2d + (K_{S_0} \cdot \Gamma_{S_0}) + (K_{S_0} \cdot \sigma_*Z_V)$;
		\item[(4)] $K_X^3 \ge \frac{1}{2} \left(\left(\pi^*K_X\right)|_S \cdot \sigma^*K_{S_0}\right)$. 
	\end{enumerate}
\end{prop}

\begin{proof}
	The equality (1) follows is just \cite[(3.7) and (3.8)]{HZ}. The equality (2) follows from \eqref{eq: K_S0}. For (3), we have
	\begin{align*}
		\left(\left(\pi^*K_X\right)|_S \cdot \sigma^*K_{S_0}\right) & = \left((M|_S + Z|_S) \cdot \sigma^*K_{S_0}\right) \\
		& = \left((dC + \Gamma_S + Z_V) \cdot \sigma^*K_{S_0}\right)\\
		& = 2d + (K_{S_0} \cdot \Gamma_{S_0}) + (K_{S_0} \cdot \sigma_*Z_V).
	\end{align*}
	Thus the equality in (3) holds. To prove (4), note that we have
	$$
	K_X^3 = (\pi^*K_X)^3 \ge \left(\left(\pi^*K_X\right)|_S\right)^2 \ge \frac{1}{2} \left(\left(\pi^*K_X\right)|_S \cdot \sigma^*K_{S_0}\right),
	$$
	where the last inequality follows from \eqref{eq: extension}. The proof is completed.
\end{proof}

\subsection{Refined estimate} In this subsection, we prove two refined numerical results subject to the effective $\mathbb{Q}$-divisor $H$ in \eqref{eq: extension}.

Take a general linear pencil $\Lambda$ of $\mathrm{Mov}|K_X|$. Since $q(X) = 0$ (see \cite[Lemma 3.4]{HZ} for example), applying \cite[Proposition 3.1]{CHJ} to $\Lambda$, we get a birational morphism $\mu: W \to X$ with a fibration $g: W \to \PP^1$ such that $W$ is $\mathbb{Q}$-factorial terminal and 
\begin{equation}\label{eq: G}
	G := \mu^*(K_X+S_X)-K_W-S_W
\end{equation}
is an effective $\mu$-exceptional divisor, where $S_W$ is a general fibre of $g$ and $S_X = \mu_*S_W$. Note that $G$ is independent of $S_W$ by the negativity lemma \cite[Lemma 3.39]{KM}. We may write
\begin{align}\label{eq: fixed part}
	K_X = S_X + Z_X, \quad K_W = \mu^*K_X + E_\mu, 
\end{align}
where $E_\mu \ge 0$ is a $\mu$-exceptional $\mathbb{Q}$-divisor.
Since $|K_X|$ is not composed with a pencil and $\Lambda$ is general, we deduce that $Z_X$ is just the fixed part of $|K_X|$. Note that $S_X$ is a general member in $\mathrm{Mov}|K_X|$. We may assume that $S_X = \pi_*S$. Thus $S_W$ is birational to $S$. In particular, $S_0$ is the minimal model of $S_W$. Denote by $\sigma_W: S_W\to S_0$ the contraction. 

\subsubsection{The case when $H \neq 0$} We first consider the case when $H \ne 0$. We have the following refined Noether inequality.
\begin{prop}\label{prop: Noether H>0}
	Suppose that $H \neq 0$ for the effective $\QQ$-divisor $H$ in \eqref{eq: extension}. Then the following inequality holds:
	$$
	K_X^3\ge \frac{4}{3}p_g(X)-\frac{17}{6}.
	$$
\end{prop}

To prove Proposition \ref{prop: Noether H>0}, we assume that $H \ne 0$. Note that $\left(\pi^*K_X\right)|_S$ and $\sigma^*K_{S_0}$ are nef and big divisors. By \eqref{eq: extension} and the Hodge index theorem, we have 
\begin{equation} \label{eq: KK1}
	\left((\mu^*K_X)|_{S_W} \cdot \sigma_W^*K_{S_0}\right) = \left((\pi^*K_X)|_S \cdot \sigma^*K_{S_0}\right) > \frac{1}{2} K_{S_0}^2.
\end{equation}
On the other hand, by \eqref{eq: G} and \eqref{eq: fixed part}, we have
\begin{align}
	\left((\mu^*K_X)|_{S_W} \cdot \sigma_W^*K_{S_0}\right) & = \frac{1}{2} \left((\mu^*Z_X + K_W + S_W + G)|_{S_W} \cdot \sigma_W^*K_{S_0}\right) \nonumber \\ & 
	= \frac{1}{2} \left(K_{S_W} \cdot \sigma_W^*K_{S_0}\right) + \frac{1}{2}\left((\mu^*Z_X + G)|_{S_W} \cdot \sigma_W^*K_{S_0}\right) \label{eq: KK2}\\
	& = \frac{1}{2}K_{S_0}^2 + \frac{1}{2} \left(\sigma_W^*K_{S_0} \cdot (\mu^*Z_X + G)|_{S_W}\right). \nonumber
\end{align}
Combine the above two result together, and it follows that 
$$
\left(\sigma_W^*K_{S_0} \cdot (\mu^*Z_X + G)|_{S_W}\right) > 0.
$$
Thus there is an integral curve $A\subseteq \mathrm{Supp} \left((\mu^*Z_X + G)|_{S_W}\right)$ such that $(\sigma_W^*K_{S_0} \cdot A) \ge 1$. Let $\lambda$ be the coefficient of $A$ in the effective $\QQ$-divisor $(\mu^*Z_X + G)|_{S_W}$.

\begin{lem} \label{lem: coeff of lamda}
	We have $\lambda \ge \frac{1}{3}$. As a result, we have
	$$
	\left((\pi^*K_X)|_S \cdot \sigma^*K_{S_0}\right) \ge \frac{1}{2} K_{S_0}^2 + \frac{1}{6}.
	$$
\end{lem}

\begin{proof}
	Note that the inequality on $\left((\pi^*K_X)|_S \cdot \sigma^*K_{S_0}\right)$ is a consequence of \eqref{eq: KK1}, \eqref{eq: KK2} and the fact that $\lambda \ge \frac{1}{3}$. Thus we only need to prove that $\lambda \ge \frac{1}{3}$. In the following, we assume that $\lambda < 1$.
	
	If $A$ is not contained in a $\mu$-exceptional divisor, then $A \subset (\mu^{-1}_*Z_X)|_{S_W}$. In this case, $\lambda$ must be a positive integer. Thus we may assume $A \subset E_i|_{S_W}$ for some $\mu$-exceptional prime divisor $E_i$. If $\mu(A)$ is a curve on $X$, then $\mu(E_i) = \mu(A)$ is also a curve. Since the singularities of $X$ are isolated, $X$ is smooth at a general point of $\mu(E_i)$. In this case, $\lambda$ is again a positive integer. Thus we further reduce to the case when $\mu(A)$ is a point. Then we have $\left((\mu^*K_X)|_{S_W} \cdot A\right) = 0$. 
	
	On the other hand, similar to \eqref{eq: KK2}, we have
	\begin{align*}
		\left((\mu^*K_X)|_{S_W} \cdot A \right) & =  \frac{1}{2} \left((\mu^*Z_X + K_W + S_W + G)|_{S_W} \cdot A\right) \\
		& \ge \frac{1}{2} ((K_{S_W} + \lambda A) \cdot A) \\
		& = \frac{1}{2} (1-\lambda) (K_{S_W}\cdot A) + \lambda(p_a(A) - 1) \\
		& \ge \frac{1}{2} (1-\lambda) (K_{S_W}\cdot A) - \lambda.
	\end{align*}
	Since $(\sigma_W^*K_{S_0} \cdot A) > 0$, we see that $A$ is not $\sigma_W$-exceptional. Thus we have $(K_{S_W} \cdot A)\ge (\sigma_W^*K_{S_0} \cdot A)\ge 1$. Since $\lambda < 1$, the above inequality implies that $0 \ge 1 - 3\lambda$. Thus $\lambda\ge\frac{1}{3}$. The proof is completed.
\end{proof}

Now we prove Proposition \ref{prop: Noether H>0}.
\begin{proof} [Proof of Proposition \ref{prop: Noether H>0}]
	By Lemma \ref{lem: coeff of lamda}, we have
	$$
	\left((\pi^*K_X)|_S \cdot \sigma^*K_{S_0}\right) \ge \frac{1}{2} K_{S_0}^2 + \frac{1}{6}.
	$$
	Combine this with Proposition \ref{prop: estimate on S} (2) and (3), and we deduce that
	$$
	(K_{S_0} \cdot \sigma_*Z_V) \ge \frac{1}{2} \left(K_{S_0} \cdot \sigma_*(E_V + Z_V)\right) + \frac{1}{6}.
	$$
	By \eqref{eq: E|_SZ|_S}, $E_V + Z_V = K_{X'}|_S - S|_S - 2\Gamma_S$ . Thus $E_V + Z_V \ge 0$ is a Cartier divisor on $S$. Thus the above inequality implies that $\left(K_{S_0} \cdot \sigma_*(E_V+Z_V) \right) \ge 1$, which further implies that 
	$$
	(K_{S_0} \cdot \sigma_*Z_V) \ge \frac{2}{3}.
	$$
	Now $\sigma_*(E_V+Z_V) \ne 0$. Since $K_{S_0}$ is $2$-connected, we have $\left(\Gamma_{S_0} \cdot \sigma_*(E_V+Z_V) \right) \ge 1$. Together with Proposition \ref{prop: estimate on S} (1), we deduce that
	$$
	(K_{S_0} \cdot \Gamma_{S_0}) \ge \frac{2}{3} d - 1.
	$$
	Combine the above two inequalities with Proposition \ref{prop: estimate on S} (3) and (4), and it follows that 
	$$
	K_X^3 \ge \frac{1}{2} \left((\pi^*K_X)|_S \cdot \sigma^*K_{S_0}\right) \ge \frac{4}{3} d - \frac{1}{6} \ge \frac{4}{3} p_g(X) - \frac{17}{6},
	$$
	where the last inequality follows from the fact that $d \ge p_g(X) - 2$. Thus the proof is completed.
\end{proof}

\subsubsection{The case when $H=0$} We now treat the case when $H = 0$. We have the following very explicit description.

\begin{prop}\label{prop: Noether H=0}
	Suppose that $H=0$ for the effective $\QQ$-divisor $H$ in \eqref{eq: extension}. Then the following statements hold:
	\begin{enumerate}
		\item[(1)] the canonical linear system $|K_X|$ has no fixed part, i.e., $Z_X=0$;
		\item[(2)] a general member $S_X \in |K_X|$ has at worst Du Val singularities with $K_{S_X}$ nef;
		\item[(3)] the Cartier index of $K_X$ is at most two, and
		$$
		K_X^3 = \frac{4}{3}p_g(X)-\frac{10}{3}+\frac{N}{6}
		$$ 
		for some non-negative integer $N$.
	\end{enumerate}
\end{prop}

\begin{proof}
	Since $H = 0$, by \eqref{eq: extension}, we have $2(\pi^*K_X)|_S \sim_{\mathbb{Q}} \sigma^*K_{S_0}$, which implies that $2(\mu^*K_X)|_{S_W} \sim_{\mathbb{Q}} \sigma_W^*K_{S_0}$. Together with \eqref{eq: G} and \eqref{eq: fixed part}, we deduce that
	$$
	(G+\mu^*Z_X)|_{S_W} = 2(\mu^*K_X)|_{S_W} - K_{S_W} \sim_{\mathbb{Q}} -(K_{S_W} - \sigma_W^*K_{S_0}).
	$$
	Since $G$, $\mu^*Z_X$ and $K_{S_W} - \sigma_W^*K_{S_0}$ are all effective divisors, it follows that 
	$$
	2(\mu^*K_X)|_{S_W} - K_{S_W} = G|_{S_W} = (\mu^*Z_X)|_{S_W}=0
	$$ 
	and that $S_W$ is minimal. By \cite[Lemma 3.4 and 3.5]{CHJ}, we know that $S_X$ is klt and that $Z_X = 0$. Moreover, for any non-Gorenstein singularity $P\in X$, the Cartier index of $K_X$ at $P$ is the same as the Cartier index of $K_X|_{S_X}$ at $P$.
	
	Since $S_X$ is klt and $Z_X = 0$, we have $2K_X|_{S_X} = (K_X+S_X)|_{S_X} = K_{S_X}$. Pulling back by $\mu|_{S_W}$, we have $K_{S_W} = (\mu|_{S_W})^*K_{S_X}$. Thus $S_X$ has at worst Du Val singularities, and $K_{S_X} = 2K_X|_{S_X}$ is a nef Cartier divisor. It follows that the Cartier index of $K_X$ is at most $2$. Thus $2K_X^3$ is a positive integer, and it follows by Theorem \ref{thm: Noether1} that $N := 6K_X^3 - 8p_g(X) - 20 \ge 0$ is a non-negative integer. As a result, we have
	$$
	K_X^3 = \frac{4}{3}p_g(X) - \frac{10}{3} + \frac{N}{6}.
	$$ 
	The proof is completed. 
\end{proof}

\subsection{Main result}
We first recall the associated basket $B_X$ to $X$ according to Reid. There is a Riemann--Roch formula in \cite[Corollary 10.3]{YPG} for $P_2(X) = h^0(X, 2K_X)$:
\begin{equation} \label{eq: Riemann-Roch P2}
	P_2(X) = \frac{1}{2}K_X^3 + 3\chi(\omega_X) + l_2(X).
\end{equation}
Here the correction term
\begin{equation} \label{eq: l2}
	l_2(X) = \sum_{Q} \frac{b_Q(r_Q - b_Q)}{2r_Q},
\end{equation}
where the sum $\sum_{Q}$ runs over all singularities $Q\in B_X$ with the type $\frac{1}{r_Q}(1, -1, b_Q)$ ($b_Q$ and $r_Q$ are coprime, and $0 < b_Q \le \frac{1}{2}r_Q$).

The main result in this section is the following theorem.

\begin{thm} \label{thm: Horikawa surface}
	Let $X$ be a minimal threefold of general type with $p_g(X) \ge 5$ such that 
	\begin{itemize}
		\item [(i)] the canonical dimension is two;
		\item [(ii)] $X$ admits a fibration $f: X \to \PP^1$ with general fibre $F$ a $(1, 2)$-surface;
		\item [(iii)] $K_X^3 < \frac{4}{3}p_g(X) - \frac{17}{6}$.
	\end{itemize}
	Let $f_0: X_0 \to \PP^1$ be the relative canonical model of $X$ with respect to $f$. Then we have the Noether equality
	$$
	K_{X_0}^3 = \frac{4}{3}p_g(X_0) - \frac{10}{3} + \frac{N}{6}
	$$
	for an integer $N \in \{0, 1, 2\}$. Moreover, the following statements hold:
	\begin{enumerate}
		\item[(1)] the Cartier index of $X_0$ is at most two;  
		\item[(2)] the canonical linear system $|K_{X_0}|$ has no fixed part;
		\item[(3)] a general member $S_{X_0} \in |K_{X_0}|$ has at worst Du Val singularities,  $K_{S_{X_0}}$ is nef and $f_0|_{S_{X_0}}$-ample, and $K_{S_{X_0}}^2 = 2p_g(S_{X_0}) - 4 > 10$.
	\end{enumerate} 
\end{thm}

\begin{proof}
	By \cite[Lemma 3.4]{HZ}, $h^1(X, \mathcal{O}_X) = h^2(X, \mathcal{O}_X) = 0$. Thus the Riemann--Roch formula \eqref{eq: Riemann-Roch P2} for $X$ becomes
	$$
	P_2(X) = \frac{1}{2}K_X^3 + 3\left(p_g(X) - 1\right) + l_2(X).
	$$
	
	Since $K_X^3 < \frac{4}{3}p_g(X) - \frac{17}{6}$, by Proposition \ref{prop: Noether H>0}, we know that $H = 0$ for the effective $\mathbb{Q}$-divisor in \eqref{eq: extension}. Thus by Proposition \ref{prop: Noether H=0}, $X$ satisfies the equality
	$$
	K_X^3 = \frac{4}{3}p_g(X) - \frac{10}{3} + \frac{N}{6}
	$$
	for some integer $0 \le N \leq 2$, so does $X_0$. Let $\tau: X \to X_0$ be the contraction. Then we have $K_X = \tau^*K_{X_0}$. Thus the statement (1) and (2) also follow from Proposition \ref{prop: Noether H=0} (1) and (2), respectively.
	
	To prove the statement (3), let $S_{X_0} \in |K_{X_0}|$ be a general member, and let $S_X = \tau^*S_{X_0}$. By Proposition \ref{prop: Noether H=0} (3), $S_X$ has at worst Du Val singularities, and $K_{S_X}$ is nef. Since $\tau|_{S_X}: S_X \to S_{X_0}$ is the contraction onto the relative canonical model of $S_X$ with respect to $f|_{S_X}$, we deduce that $S_{X_0}$ also has at worst Du Val singularities and that $K_{S_{X_0}}$ is nef and $f_0|_{S_{X_0}}$-ample.
	
	To show that $K_{S_{X_0}}^2 = 2p_g(S_{X_0}) - 4$, we only need to show that $K_{S_X}^2 = 2p_g(S_X) - 4$. Since $S_X$ has at worst isolated singularities, by Proposition \ref{prop: Noether H=0} (2), we have $K_{S_X} = (K_X + S_X)|_{S_X} = 2K_X|_{S_X}$. Thus $K_{S_X}^2 = 4K_X^3$. Consider the following exact sequence:
	$$
	0 \to \mathcal{O}_X(K_X)\to \mathcal{O}_X(K_X + S_X) \to \mathcal{O}_{S_X}(K_{S_X}) \to 0.
	$$
	Since $h^1(X, K_X) = h^2(X, \mathcal{O}_X) = 0$, we have
	\begin{equation} \label{eq: p_g(S_X)}
		p_g(S_X) = P_2(X) - p_g(X) = \frac{1}{2}K_X^3 + 2p_g(X) - 3 + l_2(X).
	\end{equation}
	If $N = 0$, i.e., $K_X^3 = \frac{4}{3}p_g(X) - \frac{10}{3}$, then $l_2(X)=0$ by \cite[Proposition 4.3]{HZ}. Thus the equation \eqref{eq: p_g(S_X)} becomes
	$$
	p_g(S_X) = \frac{1}{2}K_X^3 + 2p_g(X) - 3 = 2K_X^3 + 2 = \frac{1}{2} K_{S_X}^2 + 2.
	$$
	If $N = 1$, i.e., $K_X^3 = \frac{4}{3}p_g(X) - \frac{19}{6}$, then $l_2(X)=\frac{1}{4}$ by \cite[Proposition 4.4]{HZ}. Thus the equation \eqref{eq: p_g(S_X)} becomes
	$$
	p_g(S_X) = \frac{1}{2}K_X^3 + 2p_g(X) - \frac{11}{4} = 2K_X^3 + 2 = \frac{1}{2} K_{S_X}^2 + 2.
	$$
	If $N = 2$, i.e., $K_X^3 = \frac{4}{3}p_g(X) - 3$, then $l_2(X)=\frac{1}{2}$ by \cite[Proposition 4.4]{HZ}. Thus the equation \eqref{eq: p_g(S_X)} becomes
	$$
	p_g(S_X) = \frac{1}{2}K_X^3 + 2p_g(X) - \frac{5}{2} = 2K_X^3 + 2 = \frac{1}{2} K_{S_X}^2 + 2.
	$$
	As a result, we have 
	$$
	K_{S_X}^2 = 2p_g(S_X) - 4 \ge K_X^3 + 4p_g(X) - 10 > 10
	$$
	in all three cases $N = 0, 1, 2$. Here the last inequality is from \eqref{eq: p_g(S_X)}. The proof is completed.
\end{proof}

\begin{prop} \label{prop: pg23}
	In Theorem \ref{thm: Horikawa surface}, if $p_g(X) \ge 23$, then the relative canonical model $X_0$ is just the canonical model of $X$.
\end{prop}

\begin{proof}
	Under the assumptions in Theorem \ref{thm: Horikawa surface}, if $p_g(X) \ge 23$, by \cite[Proposition 3.13 and Lemma 3.3]{HZ}, we may write 
	$$
	f_*\omega_{X} = \mathcal{O}_{\PP^1}(a) \oplus \mathcal{O}_{\PP^1}(b),
	$$
	where $a \ge b \ge 1$ are two positive integers.
	
	Consider the relative canonical map $\phi: X \dasharrow \PP(f_*\omega_{X})$ of $X$ over $\PP^1$. Since the base locus of $|K_F|$ is a single point, we see that the $f$-horizontal indeterminacies of $\phi$ form a section $\Gamma$ of $f$ whose intersection $\Gamma \cap F$ with $F$ is just the base point of $|K_F|$. Moreover, by \cite[Corollary 3.5]{Miyaoka}, $K_X - bF$ is nef away from $\Gamma$. In particular, $\left((K_X - F) \cdot C\right) \ge 0$ for any integral curve $C \ne \Gamma$. On the other hand, by \cite[Proposition 3.5]{HZ}, we have $\left((K_X - F) \cdot \Gamma\right) \ge \frac{1}{3}(p_g(X) - 4) - (F \cdot \Gamma) \ge 0$. Thus we conclude that $K_X - F$ is nef. 
	
	Denote by $F_0$ a general fibre of $f_0: X_0 \to \PP^1$. Then $K_{X_0} + tF_0$ is ample for a sufficiently large $t$. Note that the above argument implies that $K_{X_0} - F_0$ is nef. Thus $K_{X_0} = \frac{t}{t+1}(K_{X_0} - F_0) + \frac{1}{t+1}(K_{X_0} + tF_0)$ is ample. The proof is completed.
\end{proof}


\section{Existence of simple fibrations in \texorpdfstring{$(1, 2)$}{(1, 2)}-surfaces} \label{section: simple fibration}

In this section, we study the explicit structure of the fibration on the canonical model of the threefold in Theorem \ref{thm: Horikawa surface}.

We first recall the definition of a simple fibration in $(1,2)$-surfaces as in \cite[Definition 4.1]{CP}.
\begin{df}\label{def: simple fibration}
	A \emph{simple fibration in $(1,2)$-surfaces} is a surjective morphism
	$\pi \colon X \rightarrow B$
	such that
	\begin{itemize}
		\item [(i)] $B$ is a smooth curve;
		\item [(ii)] $X$ is a threefold with at worst canonical singularities;
		\item [(iii)] $K_X$ is $\pi$-ample;
		\item [(iv)] for all $p \in B$, the canonical ring $R(X_p, K_{X_p}):=\bigoplus_d H^0(X_p,dK_{X_p})$ of the surface $X_p: = \pi^{*} p$ is generated by four elements of respective degree $1$, $1$, $2$ and $5$ and related by a single equation of degree $10$, where $K_{X_p} = K_X|_{X_p}$.
	\end{itemize}
\end{df}

For simplicity, if a threefold $X$ admits a simple fibration in $(1,2)$-surfaces, we often write that $X$ itself \emph{is a simple fibration} as in \cite{CP}. Moreover, if $B \simeq \PP^1$, we say that the simple fibration is \emph{regular}.

\begin{thm} \label{thm: simple fibration}
	Suppose that $X$ is a canonical threefold with $p_g(X) \ge 5$ such that one of the following holds:
	\begin{itemize}
		\item [(1)] the canonical dimension of $X$ is two, $p_g(X) \ge 7$ and $K_X^3 < \frac{4}{3}p_g(X) - \frac{17}6$.
		\item [(2)] 
        $p_g(X) = 6$ and $K_X^3 < \frac{109}{30}$. 
        \item[(3)] $p_g(X) = 5$ and $K_X^3 < \frac{61}{12}$.
	\end{itemize}
	Then there is a crepant birational morphism $X_0 \to X$ such that $X_0$ admits a regular simple fibration in $(1,2)$-surfaces. Moreover, if $p_g(X) \ge 23$, then $X_0 \simeq X$.
\end{thm}

\begin{proof}
	By \cite[Proposition 2.1]{HZ} and Lemma \ref{lem: p_g5}, there is a minimal model $X_1$ of $X$ so that $X_1$ admits a fibration $\pi_1 \colon X_1\to \PP^1$ whose general fibre is a smooth $(1,2)$-surface. Let $X_0$ be the relative canonical model of $X_1$ with respect to $\pi_1$. Then we have the induced fibration $\pi_0: X_0 \to \PP^1$. Let $F_p$ denote the fibre of $\pi_0$ over any closed point $p \in \PP^1$. 
	
	We first prove the following two claims.
	
	\textbf{Claim 1}. A general element $C \in |K_{F_p}|$ is an integral curve of arithmetic genus two. In particular, $F_p$ is integral and $K_{F_p}^2 = 1$.
	
	In fact, take a general member $S_{X_0} \in |K_{X_0}|$. By Theorem \ref{thm: Horikawa surface} (3), $K_{S_{X_0}}$ is nef and $\pi_0|_{S_{X_0}}$-ample, and $K_{S_{X_0}}^2 = 2p_g(S_{X_0}) - 4 > 10$. In particular, $p_g(S_{X_0}) \ge 8$. By \cite[\S 1]{Horikawa1}, $S_{X_0}$ itself is a canonical surface on the Noether line. By the classification of singular fibres in \cite{Horikawagenus2}, every fibre of $\pi_0|_{S_{X_0}}: S_{X_0} \to \PP^1$ is an integral curve of arithmetic genus two. That is, $C_p := S_{X_0}|_{F_p}$ is integral for every $p$. So is $F_p$. Thus $K_{F_p}^2 = 1$.
	
	\textbf{Claim 2}. For any integer $n \ge 1$, we have $h^1(F_p, nK_{F_p}) = 0$. Moreover, $p_g(F_p) = 2$.
	
	In fact, for any integer $n \ge 1$, consider the exact sequence
	\begin{equation}\label{eq: exact-sequence}
		\begin{split}
			0 &\to H^0(X_0, nK_{X_0})\to H^0(X_0, nK_{X_0}+F_p)\to H^0(F_p, nK_{F_p})\\
			& \to H^1(X_0, nK_{X_0})\to H^1(X_0, nK_{X_0}+F_p)\to H^1(F_p, nK_{F_p})\\
			& \to H^2(X_0, nK_{X_0}).
		\end{split}
	\end{equation}
	Now $H^i(X_0, nK_{X_0})$ vanishes for $i = 1, 2$ when $n=1$ by \cite[Lemma 3.4]{HZ} and the Serre duality, and when $n \ge 2$ by the Kawamata-Viehweg vanishing theorem. Thus we have $h^1(X_0, nK_{X_0} + F_p) = h^1(F_p, nK_{F_p})$, and this does not depend on $p$. Therefore, since $h^1(F_p,nK_{F_p}) = 0$ for a general $F_p$ which is a canonical $(1, 2)$-surface, we have 
	$$
	h^1(X_0, nK_{X_0}+F_p) = h^1(F_p, nK_{F_p}) = 0
	$$ 
	for all $F_p$. Moreover, all plurigenera $h^0(F_p, nK_{F_p}) = h^0(X_0, nK_{X_0} + F_p) - h^0(X_0, nK_{X_0})$ do not depend on $p$. We conclude that $p_g(F_p) = 2$.

	With the above two claims, we now consider the half canonical ring $R(C, K_{F_p}|_{C}) := \bigoplus_d H^0(C, dK_{F_p}|_C)$ for a general element $C \in |K_{F_p}|$. By Theorem \ref{thm: Horikawa surface}, $2K_{X_0}$ is Cartier, so is $2K_{F_p}$. By the adjunction, $C$ is a Gorenstein curve. Note that $\omega_{F_p}|_C$ is a torsion free sheaf, not necessarily locally free. Nevertheless, by \textbf{Claim 2}, we have $h^0(C, K_{F_p}|_C) = p_g(F_p) - 1 = 1$, and $C$ is also integral by \textbf{Claim 1}. By \cite[Theorem 5.2]{CFPR}, $R(C, nK_{F_p}|_{C})$ is generated by three elements of respective degree $1$, $2$ and $5$ and related by a single equation of degree $10$. Then we further apply the proof of \cite[Theorem 3.3 (1)]{FPRb} verbatim to deduce that for any $F_p$, the canonical ring $R(F_p, K_{F_p})$ is generated by four elements of respective degree $1$, $1$, $2$ and $5$, and they are related by a single equation of degree $10$. As a result, $\pi_0$ is exactly a regular simple fibration in $(1, 2)$-surfaces. Moreover, if $p_g(X) \ge 23$, then we have $X \simeq X_0$ by Proposition \ref{prop: pg23}. The proof is completed.
\end{proof}

\begin{remark}\label{rem: pg=4}
	One cannot hope to completely remove these assumptions in Theorem \ref{thm: simple fibration}, because $X_{10} \subset \PP(1,1,1,1,5)$ is a threefold of general type with $p_g=4$, $K^3=2$ that is not birational to any simple fibration in $(1,2)$-surfaces. Its canonical map gives a double cover over $\PP^3$.
\end{remark}

Combining Theorem \ref{thm: fibration exsits} and \ref{thm: simple fibration} together, we immediately have the following corollary.

\begin{cor}\label{cor: simple fibration}
	Suppose that $X$ is a canonical threefold on the refined Noether line with $p_g(X) \ge 5$. Then there is a crepant birational morphism $X_0 \to X$ such that $X_0$ admits a regular simple fibration in $(1,2)$-surfaces. Moreover, if $p_g(X) \ge 23$, then $X_0 \simeq X$.
\end{cor}


\section{Simple fibrations as hypersurfaces in toric fourfolds} \label{section: toric 4fold}
Let $f: X \to \PP^1$ be a simple fibration in $(1,2)$-surfaces with $p_g(X) > 0$. Consider the relative canonical algebra 
$$
\mathcal{R} = \bigoplus_{m \geq 0}{\mathcal R}_m=\bigoplus_{m \geq 0} f_* \omega^{[m]}_{X/\PP^1}
$$
as a graded $\mathcal{O}_{\PP^1}$-algebra. By \cite[Theorem 4.6]{CP}, $X$ is isomorphic to a hypersurface of degree $10$ in the $\PP(1, 1, 2, 5)$-bundle $\mathbf{F}(X) : = \mathrm{\mathbf{Proj}}\, \mathcal{R}$ over $\PP^1$. Moreover, the fibration $\pi: \mathbf{F}(X) \to \PP^1$ admits two sections $\s_2$ and $\s_5$ such that for every point of $\s_2$ (resp. $\s_5$), there is an analytic neighborhood on which $\mathbf{F}(X)$ is isomorphic to the product of a disk and a quotient singularity of type $\frac{1}{2}(1, 1, 1)$ (resp. $\frac{1}{5}(1, 1, 2)$).

Note that $\mathcal{R}$ has a natural graded $\mathcal{O}_{\PP^1}$-subalgebra $\mathcal{Q}$ locally generated by $1$, $\mathcal{R}_1$ and $\mathcal{R}_2$ (see \cite[Definition 4.10]{CP}). Then $\mathbf{Q}(X) := \mathrm{\mathbf{Proj}}\, \mathcal{Q}$ is a ${\PP}(1,1,2)$-bundle over $\PP^1$. 

Since the fibres of $f$ are hypersurfaces of degree $10$ in $\PP (1,1,2,5)$, the multiplication in $\mathcal{R}$ yields an exact sequence
\begin{equation} \label{eq: multiplication exact sequence}
	0 \to \Sym^2 \mathcal{R}_1 \to \mathcal{R}_2 \to \mathcal{E}_2 \to 0,
\end{equation}
where ${\mathcal E}_2$ is a line bundle. Since every vector bundle over $\PP^1$ is a direct sum of line bundles, we may uniquely write
\begin{equation} \label{eqn: R1}
   	\mathcal{R}_1 = \mathcal{O}_{\PP^1} (d_0) x_0 \oplus \mathcal{O}_{\PP^1} (\delta - d_0) x_1,
\end{equation}
with $2d_0 \le \delta$, so that $\delta = \deg \mathcal{R}_1$. Set
\begin{gather} \label{eqn: dNd2}
\begin{split}
    d_2 := \deg \mathcal{E}_2, \quad  N := 3d_2 &- 2\delta, \quad
    d := \delta-d_2, \\ 
    e:= 3d-2d_0+N &= \delta-2d_0 \ge 0.
\end{split}
\end{gather}
By \cite[Theorem 2.7]{Fujita} and \cite[Proposition 4.6]{Viehweg2} ($p_g(X) > 0$ implies that $h^0 \left(\PP^1, \mathcal{R}_1 \otimes \mathcal{O}_{\PP^1}(-2)\right) > 0$ thus $\mathcal{R}_1$ contains an ample line bundle), we have
\begin{equation}\label{eqn: d0d2>0}
    d_0 \ge 0, \quad  d_2 \ge 1.  
\end{equation}
By \cite[Definition 4.18 and Proposition 4.21]{CP} and \cite[Lemma 3.4]{HZ}, $N$ is non-negative and 
\begin{equation} \label{eq: N}
    K^3_X = \frac43 \chi(\omega_X) - 2 \chi(\mathcal{O}_{\PP^1}) + \frac16 N = \frac{4}{3}p_g(X) - \frac{10}{3} + \frac16 N.
\end{equation}

\subsection{Simple fibrations with $N \le 4$}
We start from the following lemma.
\begin{lem} \label{lem: epsilon_2 splits}
	If $N \leq 4$, then the short exact sequence \eqref{eq: multiplication exact sequence} splits.
\end{lem}

\begin{proof}
    By \eqref{eqn: R1}, we have
	$$
	\Sym^2 \mathcal{R}_1 = \mathcal{O}_{\PP^1} (2d_0) x_0^2 \oplus \mathcal{O}_{\PP^1} (\delta) x_0x_1 \oplus \mathcal{O}_{\PP^1}(2\delta-2d_0) x_1^2,   
	$$
	and the exact sequence \eqref{eq: multiplication exact sequence} gives a class in
    \begin{multline} \label{eq: Ext with three summands}
        \mathrm{Ext}^1 \left(\mathcal{E}_2, \Sym^2 \mathcal{R}_1\right) \simeq H^1 \left(\PP^1, \Sym^2 \mathcal{R}_1 \otimes \mathcal{E}_2^\vee \right) \\
        \simeq H^1 \left(\PP^1, \mathcal{O}_{\PP^1} (2d_0-d_2) \right)
		\oplus H^1 \left(\PP^1, \mathcal{O}_{\PP^1} (d) \right) \oplus H^1 \left(\PP^1, \mathcal{O}_{\PP^1} (2\delta - 2d_0 - d_2) \right).
    \end{multline}
    Since $N \leq 4$ and $e = 3d-2d_0+N \ge 0$ by \eqref{eqn: dNd2}, we deduce from \eqref{eqn: d0d2>0} that $d \geq \left\lceil - \frac13 N \right\rceil \ge -1$. Thus the second and the third term in \eqref{eq: Ext with three summands} vanish, and
    \begin{equation}\label{eq: Ext with one summand}
		\Ext^1 \left(\mathcal{E}_2, \Sym^2 \mathcal{R}_1\right)
		\simeq H^1 \left(\PP^1, \mathcal{O}_{\PP^1} (2d_0 - d_2) \right).
	\end{equation}
	
	By \cite[Lemma 4.11, Corollary 4.15 and 4.16 (1)]{CP}, the inclusion $\mathcal{Q} \hookrightarrow \mathcal{R}$ induces a double cover $X \to \mathbf{Q}(X)$ whose branch divisor is given by a map
	\begin{equation} \label{eq: branching curve}
		\mathcal{O}_{{\mathbb P}^1}(2\delta + 2d_2) = \left((\det \mathcal{E}_1) \otimes \mathcal{E}_2 \right)^{\otimes 2} \hookrightarrow {\mathcal Q}_{10}.
	\end{equation}
	
	Let $\mathcal{I}$ be the graded ideal sheaf of $\mathcal{Q}$ locally generated by the direct summand $\mathcal{O}_{\PP^1} (\delta - d_0) x_1$ in $\mathcal{Q}_1$. Let $\mathcal{T} = \mathcal{Q}/\mathcal{I}$ be the graded quotient $\mathcal{O}_{\PP^1}$-algebra.
	Since the multiples of $x_1$ in $\mathcal{R}_2$ are in the image of the map $\Sym^2 \mathcal{R}_1 \to \mathcal{R}_2$, the exact sequence \eqref{eq: multiplication exact sequence} fits into the following commutative diagram
	$$
	\xymatrix{
		0 \ar[r] & \Sym^2 \mathcal{R}_1 \ar[r] \ar@{=}[d] & \ar[r]\ar@{=}[d] \mathcal{R}_2 & \ar[r]\ar@{=}[d] \mathcal{E}_2 & 0
		\\
		0\ar[r] & \Sym^2 \mathcal{Q}_1 \ar[r]\ar@{->>}[d] & \ar[r]\ar@{->>}[d] \mathcal{Q}_2 & \ar[r]\ar@{=}[d] \mathcal{E}_2 & 0
		\\
		0\ar[r] & \Sym^2 \mathcal{T}_1\ar[r] & \ar[r] \mathcal{T}_2 & \ar[r] \mathcal{E}_2 & 0
	}
	$$
	Since $\mathcal{T}_1 \simeq \mathcal{O}_{\PP^1}(d_0)$, the exact sequence at the bottom is given by a class in 
	\begin{equation}\label{eq: Ext T}
	\Ext^1(\mathcal {E}_2, \Sym^2 \mathcal {T}_1) \simeq H^1 \left(\PP^1, \mathcal{O}_{\PP^1} (2d_0-d_2) \right).
	\end{equation}
	Moreover, comparing with \eqref{eq: Ext with one summand}, the vertical maps connecting the top row with the bottom row induce an isomorphism on the $\Ext^1$-groups. It follows that  the exact sequence \eqref{eq: multiplication exact sequence} splits if and only if the exact sequence 
	\begin{equation}\label{eq: multiplication exact sequence on T}
		0 \to \Sym^2 \mathcal{T}_1 \to \mathcal{T}_2 \to \mathcal{E}_2 \to 0
	\end{equation}
	splits.

	To conclude the proof, we assume by contradiction that \eqref{eq: multiplication exact sequence on T} does not split. Then by \eqref{eq: Ext T}, we have $2d_0\le d_2-2$. Write $\mathcal{T}_2=\mathcal{O}_{\PP^1}(a)\oplus\mathcal{O}_{\PP^1}(b)$ with $a+b=2d_0+d_2$ and $2d_0<a\le b<d_2$. Thus the maximal degree of a direct summand of ${\mathcal T}_2$ is $d_2-1$. This implies that all direct summands of $\mathcal{T}_{10} \simeq \Sym^5 {\mathcal T}_2$ have degree at most $5d_2 - 5$. Since $(5d_2 - 5) - (2\delta + 2d_2) = N - 5 \leq -1$, it follows that $\Hom \left(\mathcal{O}_{\PP^1}(2\delta + 2d_2), \mathcal{O}_{\PP^1}(5d_2 - 5) \right)=0$, and therefore
	$$
	\Hom(\mathcal{O}_{\PP^1}(2\delta + 2d_2), \mathcal{T}_{10}) = 0.
	$$
	This implies that the image of the map \eqref{eq: branching curve} is in the ideal generated by $x_1$. As a result, the branch divisor of the double cover $X \to \mathbf{Q}(X)$ contains in particular the singular locus of $\mathbf{Q}(X)$, a section of $\mathbf{Q}(X) \to \PP^1$. Then $X$ contains the section $\mathfrak{s}_2$, contradicting \cite[Proposition 4.9 (2)]{CP}. The proof is completed.
\end{proof}

\begin{thm}\label{thm: X is hypersurface}
	If $N \leq 4$, then $\mathbf{F}(X) = \mathbb{C}^6 /\!/ (\mathbb{C}^*)^2$ is a toric fourfold with the weight matrix
	\begin{equation}\label{eq: matrix with good weights}
		\begin{pmatrix}
			t_0 & t_1 & x_0 & x_1 & y & z \\
			1 & 1 & d-d_0 & d_0-2d-N & -N & -2N \\
			0 & 0 & 1 & 1 & 2 & 5 \\
		\end{pmatrix}
	\end{equation}
	and the irrelevant ideal $(t_0,t_1) \cap (x_0,x_1,y,z)$, where $d_0 \ge 0$. 
    Moreover, $X$ is isomorphic to a divisor in $\mathbf{F}(X)$ of bidegree $(-4N, 10)$, and the defining equation of $X$ has the form 
	\begin{equation} \label{eq: defining equation}
		z^2 = \sum_{a_0 + a_1 + 2a_2 = 10} c_{a_0, a_1, a_2}(t_0, t_1) x_0^{a_0} x_1^{a_1} y^{a_2},
	\end{equation}
	where each $c_{a_0, a_1, a_2}(t_0, t_1)$ is a homogeneous polynomial of degree
	\begin{equation} \label{eq: degree c}
		\deg c_{a_0, a_1, a_2} = N + \frac12 (a_0 + a_1) d + \frac 12 (a_1-a_0) e.
	\end{equation}
\end{thm}

\begin{proof}
	Recall that $X$ is isomorphic to a hypersurface of degree $10$ in the $\PP(1, 1, 2, 5)$-bundle $\mathbf{F}(X)$ over $\PP^1$. By Lemma \ref{lem: epsilon_2 splits}, $\mathcal{R}_2 = (\Sym^2 \mathcal{R}_1) \oplus \mathcal{E}_2$. By \cite[Proposition 4.14 and Corollary 4.16]{CP}, we know that $\mathcal{R}_5 = \mathcal{Q}_5 \oplus \mathcal{E}_5$, where $\mathcal{E}_5 = (\det \mathcal{R}_1) \otimes \mathcal{E}_2$ is a line bundle. Then by \cite[Example 3.16]{CP}, $\mathbf{F}(X)$ is a toric variety with the weight matrix
	\begin{equation}\label{eq: matrix with bad weights}
		\begin{pmatrix}
			t_0 & t_1 & x_0 & x_1 & y & z \\
			1 & 1 & -d_0 & d_0 - \delta & -d_2 & -\delta - d_2 \\
			0 & 0 & 1 & 1 & 2 & 5\\
		\end{pmatrix}
	\end{equation}
	and the irrelevant ideal $(t_0,t_1) \cap (x_0,x_1,y,z)$. Moreover, by \cite[Corollary 4.15]{CP}, up to isomorphism, $X \in H^0 \left(\mathbf{F}(X), \mathcal{O}_{\mathbf{F}(X)}(10) \otimes \pi^*\mathcal{E}_5^{-2}\right)$, and it is defined by an equation
	$$
	z^2 = \sum_{a_0 + a_1 + 2a_2 = 10} b_{a_0, a_1, a_2}(t_0, t_1) x_0^{a_0} x_1^{a_1} y^{a_2},
	$$
	where 
	$$
	\deg b_{a_0, a_1, a_2} = -2d + a_0d_0 + a_1(\delta - d_0) + a_2d_2.
	$$
	
	It is easy to check that $d_0 - \delta + d = d_0 - 2d + N$, $-d_2 + 2d = -N$, and that $-\delta - d_2 + 5d = -2N$. We pass to the matrix \eqref{eq: matrix with good weights} by adding in the matrix \eqref{eq: matrix with bad weights} the second row multiplied by $d$ to the first row. In the meantime, the defining equation is also changed to the desired form as in \eqref{eq: defining equation}, where
	\begin{align*}
		\deg c_{a_0,a_1,a_2} & = -4N + a_0(d_0 - d) + a_1(N + 2d - d_0) + a_2 N \\
		& = N + \frac12 (a_0 + a_1) d + \frac 12 (a_1-a_0) (3d - 2d_0 + N).
	\end{align*}
	Here we use the fact that $a_2 = 5 - \frac{1}{2}(a_0 + a_1)$. The proof is completed.
\end{proof}

\begin{remark}
	The assumption that $N \ge 4$ in Theorem \ref{thm: X is hypersurface} is optimal, because there exists a regular simple fibration $X$ in $(1, 2)$-surfaces with $N = 5$ that is not a divisor in the toric fourfold with the weight matrix as in \eqref{eq: matrix with good weights}. See Section \ref{subsection: N=5}.
\end{remark}

\subsection{Existence of simple fibrations of type $(d, N; d_0)$} \label{subsection: existence}

In the following, we denote by $\mathbb{F}(d, N; d_0)$ the toric fourfold whose weight matrix and irrelevant ideal are identical to those in Theorem \ref{thm: X is hypersurface}.
We use $D_{\rho}$ for the torus-fixed divisor $\{\rho = 0\}$ on $\mathbb{F}(d, N; d_0)$, where $\rho \in \{t_0, t_1, x_0, x_1, y, z\}$. 

Let $F$ be the divisor $\{t_0 = 0\}$ and let $H = D_{x_0}+(d_0-d)F$. Then the classes of the torus invariant divisors in the class group of $\mathbb{F}(d, N; d_0)$ are 
$$
D_{t_0} = D_{t_1} = F, \quad D_{x_0} = H + (d-d_0)F, \quad  \ D_{x_1} = H + (d_0-2d-N)F,
$$
$$
D_{y} = 2H - NF, \quad D_{z} = 5H - 2NF.\footnote{The notation here will be used in \S \ref{section: moduli} as well.}
$$
With this notation, the section $\s_2$ (resp. $\s_5$) of the fibration $\mathbb{F}(d, N; d_0) \to \PP^1$ is just $D_{x_0} \cap D_{x_1} \cap D_z$ (resp. $D_{x_0} \cap D_{x_1} \cap D_y$). Moreover, we will use the further section $\s_0 = D_{x_1} \cap D_y \cap D_z$.

\begin{df} \label{def: type}
	We say that a regular simple fibration $X$ in $(1, 2)$-surfaces is \emph{of type $(d, N; d_0)$}, if it is isomorphic to a hypersurface of bidegree $(-4N, 10)$ in $\mathbb{F}(d, N; d_0)$. Such an $X$ will be denoted by $X(d, N; d_0)$ in the sequel.
\end{df}

By Theorem \ref{thm: X is hypersurface}, $X(d, N; d_0)$ exists if and only if a general member in the linear system $|10H - 4NF|$ has at worst canonical singularities. The following proposition gives necessary and sufficient conditions on the triple $(d, N; d_0)$ for the existence of $X(d, N; d_0)$.

\begin{prop} \label{prop: sing-X}
	Suppose that $d \ge 0$. Then $X(d, N; d_0)$ exists if and only if
	$$
	\frac14 (d + N) \le d_0 \le \frac12 (3d+N).
	$$
	A general $X(d, N; d_0)$ has $N \times \frac12 (1, 1, 1)$ singularities at isolated points on $\s_2$ and possibly has canonical singularities along $\s_0$.   
\end{prop}

\begin{proof}
    A more detailed version of this Proposition is proved in Appendix \ref{appendix: sing}.
\end{proof}

\subsection{Canonical divisor of $X(d, N; d_0)$} \label{subsection: KX}
For simplicity, we denote $\mathbb{F}(d, N; d_0)$ and $X(d, N; d_0)$ by $\mathbb{F}$ and $X$, respectively. 
\begin{lem} \label{lem: intersection}
	We have
	$$
	(H^3 \cdot F) = \frac{1}{10}, \quad H^4 = \frac{1}{10}d + \frac{19}{100}N.
	$$
\end{lem}

\begin{proof} 
	Since $D_{t_0} \cap D_{x_0} \cap D_y \cap D_z$ is a reduced smooth point, we have
	$$
	(D_{t_0} \cdot D_{x_0} \cdot D_y \cdot D_z) = 10(H^3 \cdot F) = 1.
	$$
	Thus $(H^3 \cdot F) = \frac{1}{10}$. On the other hand, since $D_{x_0} \cap D_{x_1} \cap D_y \cap D_z$ is empty, we have
	$$
	(D_{x_0} \cdot D_{x_1} \cdot D_y \cdot D_z) = 10H^4 - (10d + 19N)(H^3 \cdot F) = 0.
	$$
	Thus $H^4 = \frac{10d + 19N}{100}$. The proof is completed.
\end{proof}

Now we describe the nef cone and the ample cone of $\mathbb{F}$.
\begin{lem} \label{lem: nef and ample cone}
	The numerical divisor class $aH + bF$ is 
	\begin{enumerate}
		\item nef if and only if $a \ge 0$ and $b \ge a \cdot \max\left\{d-d_0, -\frac25N \right\}$
		\item ample if and only if $a > 0$ and $b > a \cdot \max\left\{d-d_0, -\frac25N\right\}$
	\end{enumerate}
\end{lem}

\begin{proof}
	By \cite[Theorem 6.3.12 and 6.3.13]{CLS}, $aH + bF$ is nef (resp. ample) if and only if its intersection number with all toric invariant curves is non-negative (resp. positive). Since toric invariant curves on $\mathbb{F}$ are intersections of three toric invariant divisors, we only need to check the positivity of all $\left((aH + bF) \cdot D_{\rho_1} \cdot D_{\rho_2} \cdot D_{\rho_3}\right)$, where $\rho_j \in \{t_0, t_1, x_0, x_1, y, z\}$. Moreover, it is sufficient to check the intersection numbers listed below, computed by Lemma \ref{lem: intersection}: 
	\begin{align*}
		\left((aH + bF) \cdot D_{t_0} \cdot D_{x_0} \cdot D_{x_1}\right) & = a (H^3 \cdot F) = \frac{1}{10} a, \\
		\left((aH + bF) \cdot D_{x_0} \cdot D_{x_1} \cdot D_y\right) & = 2a H^4 + (2b - 2ad - 3aN) (H^3 \cdot F) \\ 
		& = \frac15 b + \frac2{25} Na, \\
		\left((aH + bF) \cdot D_{x_0} \cdot D_{x_1} \cdot D_z\right) & = 5a H^4 + (5b - 5ad - 7aN) (H^3 \cdot F) \\
		& = \frac12 b + \frac 14 Na, \\
		\left((aH + bF) \cdot D_{x_0} \cdot D_y \cdot D_z\right) & = 10aH^4 + \left(10b - 9aN + 10a(d-d_0) \right) (H^3 \cdot F) \\
		& = b + a(N + d - d_0), \\
		\left((aH + bF) \cdot D_{x_1} \cdot D_y \cdot D_z\right) & = 10aH^4 + \left(10b - 19aN + 10a(d_0 - 2d) \right) (H^3 \cdot F) \\
		& = b + a(d_0 - d).
	\end{align*}
	Thus $aH + bF$ is nef if and only if $a \ge 0$ and $b \ge a \cdot \max \left\{d-d_0, -\frac25N \right\}$. The ampleness part follows similarly.
\end{proof}

By \cite[Proposition 1.1]{CP}, we have $\omega_{\FF} = \mathcal{O}_{\mathbb{F}}\left( (d+4N-2)F-9H \right)$. Since $X \in |10H - 4NF|$, the adjunction formula gives 
$$
K_X = (K_{\mathbb{F}}+X)|_X = \left((d - 2)F + H\right)|_X.
$$
We have the following propositions.

\begin{prop} \label{prop: canonical image}
	Let $\Sigma$ be the canonical image of $X$.
	\begin{enumerate}
		\item If $d_0 \ge 3$, then $\Sigma$ is isomorphic to the Hirzebruch surface $\mathbb{F}_e$;
		\item If $d_0 = 2$, then $\Sigma$ is the cone over a rational normal curve of degree $e = 3d - 4 + N$.
		\item If $d_0 = 1$, then $\Sigma$ is a rational normal curve of degree $e - 1 = 3d - 3 + N$.
	\end{enumerate}
	In each case, we have $p_g(X) = 3d - 2 + N$.
\end{prop}

\begin{proof}
	Consider the short exact sequence
	$$
	0 \to \mathcal{O}_\FF (K_\FF) \to \mathcal{O}_\FF (K_\FF + X) \to \mathcal{O}_X (K_X) \to 0.
	$$
	Since $H^0(\FF, K_\FF) = H^1(\FF, K_\FF) = 0$, it follows that 
	$$
	H^0(X, K_X) = H^0(\FF, K_\FF + X) = H^0 \left(\FF, (d-2)F + H\right).
	$$
	
	If $d_0 \ge 2$, then a basis for $H^0(X, K_X)$ is given by the monomials in 
	\begin{multline*}
		t_0^{d_0-2}x_0, t_0^{d_0-3}t_1x_0, \ldots, t_1^{d_0-2}x_0, \\
		t_0^{3d-2+N-d_0}x_1, t_0^{3d-3+N-d_0}t_1x_1, \ldots, t_1^{3d-2+N-d_0}x_1.
	\end{multline*}
	Thus $X$ is mapped to the Hirzebruch surface $\mathbb{F}_e$ by $|K_X|$ if $d_0 \ge 3$. If $d_0 = 2$, then $x_0$ is a basis element, and the negative curve on $\mathbb{F}_e$ is contracted to give a cone.
	
	If $d_0=1$, then there are no monomials involving $x_0$, and the basis becomes
	$$
	t_0^{3d-3+N}x_1, t_0^{3d-4+N}t_1x_1, \ldots, t_1^{3d-3+N}x_1.
	$$
	Clearly, now $\Sigma$ is a rational normal curve of the desired degree.
	
	From the above computation, we see that $p_g(X) = 3d - 2 + N$ in each case. The proof is completed.
\end{proof}

\begin{prop} \label{prop: KX-nef} 
	The canonical divisor $K_X$ is
	\begin{enumerate}
		\item nef if $\min \left\{d_0, d + \frac25N\right\} \ge 2$;
		\item ample if $\min\left\{d_0, d + \frac25N\right\} > 2$.
	\end{enumerate}
	Moreover, we have
	$$
	K_X^3 = 4d - 6 + \frac32N.
	$$
\end{prop}

\begin{proof}
	By Lemma \ref{lem: nef and ample cone}, $K_X$ is the restriction of a nef divisor on $\FF$ if $d-2\ge\max\{d-d_0,-\frac25N\}$. Separating the two inequalities, we get $d-2 \ge d-d_0$ which is equivalent to $d_0 \ge 2$, and $d-2 \ge - \frac25N$ which is equivalent to $d + \frac25N \ge 2$. The ampleness part follows similarly. 
	
	By Lemma \ref{lem: intersection}, we have
	\begin{align*}
		K_X^3 & = \left( \left((d-2)F + H\right)^3 \cdot (10H - 4NF)\right) \\
		& = 10H^4 + \left(30(d - 2) - 4N\right) (H^3 \cdot F) \\
		& = 4d - 6 + \frac{3}{2}N.
	\end{align*}
	The proof is completed.
\end{proof}

\subsection{Classification of $X(d, N; d_0)$ with nef but non-ample canonical classes}
We still denote $X(d, N; d_0)$ by $X$. Suppose that $d \ge 0$. By Proposition \ref{prop: sing-X}, when $\min\left\{d_0, d + \frac25N\right\} > 2$, $K_X$ is ample. Thus $X$ is canonical. 

When $\min\left\{d_0, d + \frac25N\right\} = 2$, $K_X$ is nef, but the canonical model of $X$ is a crepant contraction. In this case, $X$ can be explicitly classified. In fact, if $d_0 = 2$, by Proposition \ref{prop: sing-X}, we have $d + N \le 8$ and $3d + N \ge 4$. Thus $X$ is one of the following three cases:
\begin{enumerate}
    \item $X(0,N;2)$ for $5 \le N \le 8$;
	\item $X(1,N;2)$ for $3 \le N \le 7$;
	\item $X(d,N;2)$ for $2 \le d \le 8$, $0\le N \le 8-d$.
\end{enumerate}
If $d + \frac25N = 2$, then there is only one extra case:
\begin{enumerate}
    \item [(4)] $X(2,0;3)$.
\end{enumerate}
In each of the above cases, $K_X$ is big, because $K_X^3 = 4d - 6 + \frac{3}{2}N > 0$ by Proposition \ref{prop: KX-nef}. On the other hand, $K_X$ is never ample. Indeed, if $d + \frac25N > 2$ and $d_0 = 2$, then the proof of Proposition \ref{prop: sing-X} (see Appendix \ref{appendix: sing}) shows that the section $\s_0 = D_{x_1} \cap D_y \cap D_z$ is contained in $X$. By the proof of Lemma \ref{lem: nef and ample cone}, we know that $(K_X \cdot \s_0) = 0$. If $d + \frac25N = 2$, then $X = X(0,5;2)$, $X(2,0;2)$ or $X(2,0;3)$. By \cite[Example 1.12]{CP}, the canonical divisor of $X(2,0;3)$ is not ample. In the other two cases, by Theorem \ref{thm: X is hypersurface}, we always have $\deg c_{10,0,0} = 0$. We may assume the defining equation of $X$ is of the form $z^2 = x_0^{10} + \cdots$. In particular, the curve $\Gamma = D_{x_1} \cap D_{y} \cap D$ is contained in $X$, where $D$ is the divisor in $\FF$ defined by the equation $z = x_0^5$. Note that now $D_z \sim 5D_{x_0}$, which implies that $D \sim 5D_{x_0}$. By the proof of Lemma \ref{lem: nef and ample cone}, we know that $(K_X \cdot A) = 5(K_X \cdot \s_5) = 0$.

When $\min\left\{d_0, d + \frac25N\right\} < 2$, $K_X$ is no longer the restriction of a nef divisor on $\FF$. We will classify those explicitly in \S \ref{subsection: K non-nef}. Nevertheless, the following example shows that sometimes $K_X$ is still nef. 

\begin{example} \label{ex!X(1,2;2)}
	Consider the special hypersurface $X_{12}$ in $\PP(1,1,1,2,6)$ defined by the following equation
	$$
	c^2 = \sum_{k+2l \le 10} A^{(k,l)}_{12-k-2l}(a_1,a_2)a_0^kb^l
	$$
	of degree 12, where $a_0, a_1, a_2, b, c$ are the coordinates and $A_m^{(k,l)}$ are general homogeneous forms of degree indicated by the subscript. Then $\mathcal{O}_{X_{12}}(K_{X_{12}}) = \mathcal{O}_{X_{12}}(1)$. Thus $p_g(X_{12}) = 3$ and $K_{X_{12}}^3 = 1$. Since the equation has degree $12$, the right-hand side is contained in the ideal $(a_1,a_2)^2$. The hypersurface $X_{12}$ has a pencil over $\PP^1$ given by $(a_0: a_1: a_2: b: c) \mapsto (a_1: a_2)$, with the base locus $\Gamma : = \{a_1 = a_2 = c = 0\} \subset X_{12}$. It is easy to see that $X_{12}$ has $A_1$ singularities along the curve $\Gamma$, with a non-isolated $cA_1 \subset \frac12(1,1,1,0,0)$ hyperquotient singularity at $(0:0:0:1:0)$ on $\Gamma$. 
	
	Let $X \to X_{12}$ be the blow-up along $\Gamma$ which is a crepant partial resolution. Then $X$ is quasi-smooth with two $\frac12(1,1,1)$ singularities, $K_X$ is nef, and the induced fibration $f: X \to \PP^1$ is a regular simple fibration in $(1, 2)$-surfaces. Now by \eqref{eq: N}, $N = 6K_X^3 - 8p_g(X) + 20 = 2$. By Theorem \ref{thm: X is hypersurface}, $X$ is isomorphic to $X(d, 2; d_0)$ for some $d$ and $d_0$. By Proposition \ref{prop: KX-nef}, we deduce that $d = 1$. Since that canonical image of $X_{12}$ is $\PP^2$, by Proposition \ref{prop: canonical image}, we have $d_0 = 2$.
	
	Note that in this case, $\min\left\{d_0, d + \frac25N\right\} = d + \frac25N = \frac{9}{5}$, which is the largest possible value that is less than two.
\end{example}


\section{Moduli spaces of threefolds on the refined Noether line} \label{section: moduli}

In this section, we describe the moduli space of the canonical threefolds $X$ on the refined Noether line with $p_g(X) \geq 5$. 

Given such a threefold $X$, by Corollary \ref{cor: simple fibration}, up to a crepant birational morphism, we may assume that $X$ admits a regular simple fibration in $(1, 2)$-surfaces. Set 
$$
N := 6K_X^3 - 8p_g(X) + 20.
$$ 
Then $N \in \{0, 1, 2\}$. By Theorem \ref{thm: X is hypersurface}, $X$ is isomorphic to $X(d, N; d_0)$ as in Definition \ref{def: type} for some $d$ and $d_0 \ge 0$. By Proposition \ref{prop: KX-nef}, 
$$
p_g(X) = 3d - 2 + N.
$$ 
Thus $d \ge 3$ when $N = 0$, and $d \ge 2$ when $N = 1, 2$. By Theorem \ref{thm: fibration exsits}, the canonical dimension of $X$ is two. Thus $d_0 \ge 2$ by Proposition \ref{prop: canonical image}. 

For each $N \in \{0, 1, 2\}$, let $\M^N_d(d_0)$ denote the corresponding modular family of hypersurfaces $X(d, N; d_0)$ in $\FF(d, N; d_0)$. Then it is unirational. Let $\M_{K^3, p_g}$ be the moduli space of canonical threefolds with $p_g = 3d - 2 + N$ and $K^3 = 4d - 6 + \frac N6$. By Proposition \ref{prop: sing-X}, there is a non-trivial morphism 
$$
\Phi^N_{d, d_0}: \M^N_d(d_0) \to \M_{K^3, p_g}
$$
when $\frac{1}{4}(d + N) \leq d_0 \leq \frac{1}{2}(3d + N)$. By Proposition \ref{prop: KX-nef}, if $d_0 \ge 3$, then $X(d, N; d_0)$ is a canonical model, and $\Phi^N_{d, d_0}$ is an isomorphism onto its image. If $d_0 = 2$, then $X(d, 0; d_0)$ is not a canonical model in general. However, the morphism onto its canonical model is crepant. By \cite[Main Theorem]{KawamataMatsuki} on the finiteness of minimal models for threefolds, each canonical model admits only finitely many such morphisms. Thus $\Phi^N_{d, d_0}$, if not one-to-one, is at least finite-to-one onto its image.

\subsection{The dimension of $\M^N_d(d_0)$} 

From now on, we set $\Delta^N_d(d_0)$ for the dimension of $\M^N_d(d_0)$. In the following, we adopt the notation for divisors on $\FF(d, N; d_0)$ introduced in \S \ref{subsection: existence}. 

By Theorem \ref{thm: X is hypersurface}, every $X(d, N; d_0)$ admits a finite morphism of degree $2$ over $D_z$, whose branch locus $B$ is an element in $H^0(D_z, 10H_{D_z} - 4NF_z)$, where $H_{D_z} = H|_{D_z}$ and $F_z = F|_{D_z}$. The dimension of $\M^N_d(d_0)$ is therefore equal to the dimension of the family of pairs $(D_z, B)$, i.e.,
\begin{equation} \label{eq: dimension formula}
	\Delta^N_d(d_0) = h^0(D_z, 10H_{D_z} - 4NF_z) - \dim \Aut D_z-1.
\end{equation}

We first compute the dimension of the automorphism group of $D_z$.

\begin{lem} \label{lem: dimension AutDz}
	The dimension of the automorphism group of $D_z$ is
	$$
	\dim\Aut D_z=
	\begin{cases}
		3d + 10, & \text{ if } d_0 = \frac12(3d + N);\\ 
		6d - 2d_0 + 9 + N, & \text{ if } d + \frac12 N \le d_0 < \frac12(3d + N);\\
		8d - 4d_0 + 8 + 2N, & \text{ if } \frac 14(d + N) \le d_0 < d + \frac12 N.\\ 
	\end{cases}
	$$
\end{lem}

\begin{proof}
	By \cite[\S 4]{Cox} and the relations among $D_{\rho}$ and $H$ in \S \ref{subsection: existence}, we have the formula
	\begin{align}
		\dim\Aut D_z & = \sum_{\rho \in \{ t_0,t_1,x_0,x_1,y\}} h^0(D_z, D_{\rho}|_{D_z})-2 \nonumber \\
		& = 2h^0(D_z, F_z) + h^0(D_z, (d - d_0)F_z + H_{D_z}) \label{eq: dimension AutDz} \\
		& \quad + h^0(D_z, (d_0 - 2d - N)F_z + H_{D_z}) + h^0(D_z, 2H_{D_z} - NF_z)-2. \nonumber
	\end{align}
	It is easy to decompose these vector spaces in terms of monomials on $D_z$ using the weight matrix \eqref{eq: matrix with good weights} as follows:
	\begin{align*}
		H^0(D_z,F_z) & = S^1(t_0,t_1), \\
		H^0(D_z, (d-d_0)F_z+H_{D_z}) & = \CC x_0\oplus S^{3d - 2d_0 + N}(t_0,t_1) x_1, \\
		H^0(D_z, (d_0-2d-N)F_z+H_{D_z}) & = S^{2d_0 - 3d - N}(t_0,t_1) x_0 \oplus \CC x_1, \\
		H^0(D_z, 2H_{D_z} - NF_z) & = S^{2d_0 - 2d - N}(t_0,t_1) x_0^2 \oplus S^{d}(t_0,t_1) x_0 x_1 \\
		& \quad \oplus S^{4d - 2d_0 + N}(t_0,t_1) x_1^2 \oplus \CC y.
	\end{align*}
	
	It is clear that 
	$$
	h^0(D_z, F_z)=2.
	$$ 
	For the second term, we have
	$$
	h^0(D_z, (d-d_0)F_z + H_{D_z}) = 3d - 2d_0 + N + 2.
	$$
	For the third term, we have
	$$
	h^0(D_z, (d_0 - 2d - N)F_z + H_{D_z}) = 
	\begin{cases}
		2, & \text{if } d_0 = \frac12 (3d + N);\\
		1, & \text{otherwise}.\\
	\end{cases}
	$$
	Finally, we have
	$$
	h^0(D_z, 2H_{D_z} - NF_z)=
	\begin{cases}
		3d+4, &\text{if } d_0 \ge d + \frac12 N;\\
		5d - 2d_0 + 3 + N, &\text{otherwise}. \\
	\end{cases}
	$$
	In fact, note first that both $d$ and $4d - 2d_0 + N$ are positive. If $d_0 \ge d + \frac 12 N$, then
	\begin{align*}
		h^0(D_z, 2H_{D_z} - NF_z) & = (2d_0 - 2d - N + 1) + (d+1) + \left(4d - 2d_0 + N + 1\right) + 1 \\
		& = 3d + 4.
	\end{align*}
	If $d_0 < d + \frac 12 N$, then $x_0^2$ does not appear, and thus 
	\begin{align*}
		h^0(D_z, 2H_{D_z} - NF_z) & = (d+1) + \left(4d - 2d_0 + N + 1\right) + 1 \\
		& = 5d - 2d_0 + N + 3.
	\end{align*}
	Combining the above computations with \eqref{eq: dimension AutDz} together, we get the following three cases:
	\begin{itemize}
		\item [(1)] If $d_0 = \frac12(3d + N)$, then
		$$
		\dim\Aut D_z=2\cdot 2 + 2 + 2 + (3d+4) - 2 = 3d+10.
		$$
			
		\item [(2)] If $d + \frac{1}{2}N \le d_0 < \frac12(3d + N)$, then
		\begin{align*}
			\dim\Aut D_z & = 2 \cdot 2 + (3d - 2d_0 + N + 2) + 1 + (3d+4) - 2 \\
			& = 6d - 2d_0 + 9 + N.
		\end{align*}
			
		\item [(3)] If $\frac14(d + N) \le d_0 < d + \frac{1}{2}N$, then
		\begin{align*}
			\dim\Aut D_z & = 2 \cdot 2 + (3d - 2d_0 + N + 2) + 1 + (5d - 2d_0 + N + 3) - 2 \\
			& = 8d - 4d_0 + 8 + 2N.
		\end{align*}
	\end{itemize}
	The proof is completed.
\end{proof}
	
Next we count parameters for the branch divisor $B$ in $D_z$, which is an element of $H^0(D_z, 10H_{D_z} - 4NF_z)$ of the form
$$
\sum_{a_0+a_1+2a_2 = 10}
c_{a_0,a_1,a_2}(t_0,t_1) x_0^{a_0} x_1^{a_1} y^{a_2}.
$$
Each monomial $x_0^{a_0} x_1^{a_1} y^{a_2}$ contributes by adding $1 + \deg c_{a_0,a_1,a_2}$ to the dimension $h^0(D_z, 10H_{D_z} - 4NF_z)$, unless $\deg c_{a_0,a_1,a_2} < 0$, in which case the contribution is zero. The formula for the degree of each $c_{a_0,a_1,a_2}(t_0,t_1)$ is in \eqref{eq: degree c}.

In the proof of Proposition \ref{prop: sing-X} (see Appendix \ref{appendix: sing}), we see that the negativity of the degree of $c_{a_0,a_1,a_2}$ depends on some functions of $d_0,d,N$. For fixed $d$ and $N$, we let $d_0$ decrease. As $d_0$ decreases, more and more monomials disappear, because their coefficients have negative degree.
We summarize the results of this analysis in the following tables. 
\begin{table}[ht]\caption{Vanishing monomials when $N = 0$} \label{tab: N=0}
	\begin{tabular}{ccc}
        \hline
		$d_0$ & monomials with vanishing coefficient & stratum\\
		\hline
		$<d$ & $x_0^{10},\ x_0^8y,\ x_0^6y^2,\ x_0^4y^3,\ x_0^2y^4$& terminal  \\
		$<\frac78d$ & $x_0^9x_1$ &$cA_1$\\
		$<\frac56d$ & $x_0^7x_1y$ &$cA_3$\\
		$<\frac34d$& $x_0^5x_1y^2$ &$cA_4$ \\
		$<\frac23d$ & $x_0^8x_1^2$ &$cD_6$\\
		$<\frac12d$& $x_0^6x_1^2y$,\ $x_0^3x_1y^3$ &$cE_8$ \\
	\end{tabular}
\end{table}

\begin{table}[ht] \caption{Vanishing monomials when $N = 1$ and $d \ge 3$ or when $N = 2$ and $d \ge 6$} \label{tab: N=1}
	\begin{tabular}{ccc}
        \hline
		$d_0$ & monomials with vanishing coefficient & stratum\\
		\hline
		$= d$ & $x_0^{10},\ x_0^8y,\ x_0^6y^2,\ x_0^4y^3$ & terminal \\
		$< d$ & $x_0^2y^4$ & terminal  \\
		$<\frac78d + \frac38N$ & $x_0^9x_1$ &$cA_1$\\
		$<\frac56d + \frac13N$ & $x_0^7x_1y$ &$cA_3$\\
		$<\frac34d + \frac14N$ & $x_0^5x_1y^2$ &$cA_4$ \\
		$<\frac23d + \frac13N$ & $x_0^8x_1^2$ &$cD_6$\\
		$<\frac12d + \frac14N$ & $x_0^6x_1^2y$ & $cE_7$ \\
		$<\frac12d$ & $x_0^3x_1y^3$ & $cE_8$ \\
	\end{tabular}
\end{table}

The last column reflects the type of singularities that the general $X(d, N; d_0)$ has, when $d_0$ approaches the upper bound in the first column (see Appendix \ref{appendix: sing} for details). When $d_0 \ge \frac14(d + N)$, all the other coefficients have non-negative degrees. We treat the remaining cases that are not covered by Tables \ref{tab: N=0} and \ref{tab: N=1} separately.

\begin{lem} \label{lem: dimension 10H}
	Suppose that $N \le 1$, $d \ge 3$ or that $N = 2$, $d \ge 6$. Then the vector space $H^0(D_z, 10H_{D_z} - 4NF_z)$ has dimension
    $$
    \begin{cases}
        125d + 36 + 36N, & \text{ if } d < d_0 \le \frac32d + \frac12N \text{ or } d_0=d, N=0;\\ 
		125d + 32 + 46N, & \text{ if } d_0 = d, N>0; \\
		155d - 30d_0 + 31 + 46N, & \text{ if } \frac78d + \frac38 N \le d_0 < d;\\ 
		162d - 38d_0 + 30 + 49N, & \text{ if } \frac56d + \frac13 N \le d_0 < \frac78d + \frac38 N;\\ 
		167d - 44d_0 + 29 + 51N, & \text{ if } \frac34d + \frac14 N \le d_0 < \frac56d + \frac13 N;\\ 
		170d - 48d_0 + 28 + 52N, & \text{ if } \frac23d + \frac13 N \le d_0 < \frac34d + \frac14 N;\\ 
		174d - 54d_0 + 27 + 54N, & \text{ if } \frac12d + \frac14 N \le d_0 < \frac23d + \frac13 N;\\
		176d - 58d_0 + 26 + 55N, & \text{ if } \frac12d \le d_0 < \frac12d + \frac14 N;\\
		177d - 60d_0 + 25 + 55N, & \text{ if } \frac14d + \frac14 N \le d_0 < \frac12d. 	 
    \end{cases}
    $$
\end{lem}
    
\begin{proof}
	We first observe that
	\begin{equation} \label{eq: decompose-10H}
		H^0(D_z,10H_{D_z}-4NF_z)=\bigoplus_{a_0+a_1+2a_2=10} S^{\deg c_{a_0,a_1,a_2}}(t_0,t_1)x_0^{a_0}x_1^{a_1}y^{a_2}.
	\end{equation}
    
    If $d < d_0 \le \frac12(3d + N)$ or if $d=d_0$ and $N=0$, then by Tables \ref{tab: N=0} and \ref{tab: N=1}, all the coefficients $c_{a_0,a_1,a_2}$ have non-negative degree. The number of monomials is 
	$$
	\sum_{a_2=0}^5 h^0 \left(\PP^1,\hol_{\PP^1}(10-2a_2)\right)=11+9+7+5+3+1=36.
	$$
    Thus
    \begin{align*}
    h^0(D_z,10H_{D_z}-4NF_z) & =  \sum_{\substack{a_0 + a_1 + 2a_2 = 10}} (1 + \deg c_{a_0,a_1,a_2}) \\
	& = 36 + \sum_{\substack{a_0 + a_1 + 2a_2 = 10}} \deg c_{a_0,a_1,a_2}.
    \end{align*}
    Now we replace $\deg c_{a_0,a_1,a_2}$ with its expression in \eqref{eq: degree c}. By symmetry,
    \[
    \sum_{\substack{a_0 + a_1 + 2a_2 = 10}} (a_1-a_0)=0,
    \] 
    and then
    \begin{align*}
        & \sum_{\substack{a_0 + a_1 + 2a_2 = 10}} \deg c_{a_0,a_1,a_2} = \sum \left( \frac{1}{2} (a_0 + a_1) d +N\right) \\
        = & \ \frac{1}{2}  \sum (a_0 + a_1)d  +36N = d \sum a_1 +36N= d \sum_{a_2=0}^5 \sum_{a_1=0}^{10-2a_2} a_1 + 36N \\
        = & \ d \left[ \binom{11}2 + \binom92+ \binom72+ \binom52+ \binom32 \right]+36N=125d+36N.
    \end{align*}
	This concludes the proof of the case when $d < d_0$ or $d=d_0$, $N=0$.
    
    If $d_0 = d$, $N > 0$ then the monomials $x_0^{10}$, $x_0^8y$, $x_0^6y^2$, $x_0^4y^3$ no longer appear in the equation of the branch divisor. Thus
	\begin{align*}
		h^0(D_z, 10H_{D_z} - 4NF_z) & = 125d + 36 + 36N -\sum_{k=0}^3(1 + \deg c_{10-2k,0,k})\\
		& = 125d + 36 + 36N - \left(4  - 10N \right) \\
		& = 125d  + 32 + 46N.
	\end{align*}

	If $\frac78d + \frac38 N \le d_0 < d$, then we lose the monomials 
    $x_0^{10}$, $x_0^8y$, $x_0^6y^2$, $x_0^4y^3$ and also $x_0^2y^4$. Thus
	\begin{align*}
		h^0(D_z, 10H_{D_z} - 4NF_z) & = 125d + 36 + 36N -\sum_{k=0}^4(1 + \deg c_{10-2k,0,k})\\
		& = 125d + 36 + 36N - \left(5 + 30(d_0-d) - 10N \right) \\
		& = 155d - 30d_0 + 31 + 46N.
	\end{align*}

	If $\frac56d + \frac13 N \le d_0 < \frac78d + \frac38 N$, then we also lose $x_0^9x_1$. Thus
	\begin{align*}
		h^0(D_z, 10H_{D_z} - 4NF_z) & = 155d - 30d_0 + 31 + 46N - (1 + \deg c_{9, 1, 0}) \\
		& = 162d - 38d_0 + 30 + 49N.
	\end{align*}
	
	If $\frac34d + \frac14 N \le d_0 < \frac56d + \frac13 N$, then we also lose $x_0^7x_1y$. Thus
	\begin{align*}
		h^0(D_z, 10H_{D_z} - 4NF_z) & = 162d - 38d_0 + 30 + 49N - (1 + \deg c_{7, 1, 1}) \\
		& = 167d - 44d_0 + 29 + 51N.
	\end{align*}
	
	If $\frac23d + \frac13 N \le d_0 < \frac34d + \frac14 N$, then we also lose $x_0^5x_1y^2$. Thus
	\begin{align*}
		h^0(D_z, 10H_{D_z} - 4NF_z) & = 167d - 44d_0 + 29 + 51N - (1 + \deg c_{5, 1, 2}) \\
		& = 170d - 48d_0 + 28 + 52N.
	\end{align*}
	
	If $\frac12d + \frac14 N \le d_0 < \frac23d + \frac13 N$, then we also lose $x_0^8x_1^2$. Thus
	\begin{align*}
		h^0(D_z, 10H_{D_z} - 4NF_z) & = 170d - 48d_0 + 28 + 52N - (1 + \deg c_{8, 2, 0}) \\
		& = 174d - 54d_0 + 27 + 54N.
	\end{align*}
	
	If $\frac12d \le d_0 < \frac12d + \frac14 N$, then we also lose $x_0^6x_1^2y$. Thus
	\begin{align*}
		h^0(D_z, 10H_{D_z} - 4NF_z) & = 174d - 54d_0 + 27 + 54N - (1 + \deg c_{6, 2, 1}) \\
		& = 176d - 58d_0 + 26 + 55N.
	\end{align*}
	
	If $\frac14(d + N) \le d_0 < \frac12d$, then we also lose $x_0^3x_1y^3$. Thus
	\begin{align*}
		h^0(D_z, 10H_{D_z} - 4NF_z) & = 176d - 58d_0 + 26 + 55N - (1 + \deg c_{3, 1, 3}) \\
		& = 177d - 60d_0 + 25 + 55N.
	\end{align*}
	This concludes the proof.
\end{proof}

\begin{lem}\label{lem: dimension 10H N=1 1}
  Suppose that $N=1$ and $d=2$. Then the vector space $H^0(D_z, 10H_{D_z} - 4F_z)$ has dimension
	$$
	h^0(D_z, 10H_{D_z} - 4F_z)=
	\begin{cases}
		322, & \text{ if } d_0 = 3;\\ 
		328, & \text{ if } d_0 = 2.\\
 	\end{cases}
	$$
\end{lem}

\begin{proof}
    Since $d_0 \leq \frac{3}{2}d + \frac{1}{2}N = \frac{7}{2}$, we have $d_0 \leq 3$. When $d_0 = 3$, all the coefficients $c_{a_0, a_1, a_2}$ have non-negative degree. By the same argument as in the proof of Lemma \ref{lem: dimension 10H} for $N = 1$, we have
    $$
    h^0(D_z, 10H_{D_z} - 4F_z) = 125d + 36 + 36 = 322.
    $$
    When $d_0=2$, the monomials $x_0^{10}$, $x_0^8y$, $x_0^6y^2$, $x_0^4y^3$, $x_0^9x_1$ do not appear in the equation of the branch divisor. Hence we amend the result of Lemma \ref{lem: dimension 10H} ($N=1$, $d_0=d$) to compensate for the extra missing monomial $x_0^9x_1$:
    \begin{align*}
         h^0(D_z, 10H_{D_z} - 4F_z) & =125d + 32 + 46 - (1 + \deg c_{9,1,0}) \\ 
         & = 125d + 32 + 46 - (1 + 8d_0 - 7d - 3N) \\
         & = 328.
    \end{align*}
    The proof is completed.
\end{proof}

\begin{lem} \label{lem: dimension 10H N=2}
	Suppose that $N = 2$ and $d = 2, 3, 4$ or $5$. Then the dimension of the vector space $H^0(D_z,10H_{D_z}-8F_z)$, as a function of $d_0$ are those given in the following table. \begin{table}[ht]\caption{$h^0(D_z,10H_{D_z}-8F_z)$ for $N = 2$ and small $d$}\label{table: h^0(D_z,10H_{D_z}-8F_z}
		\begin{tabular}[t]{cc}
			\multicolumn{2}{c}{$d=2$} \\
			\hline
			$d_0$ & $h^0$ \\
			\hline
			$3,4$ &  $358$ \\
			$2$ &  $378$\\
			\phantom{$2$}
		\end{tabular}
		\quad\quad
		\begin{tabular}[t]{cc}
			\multicolumn{2}{c}{$d=3$} \\
			\hline
			$d_0$ & $h^0$ \\
			\hline
			$4,5$ & $483$ \\
			$3$ & $501$ \\
			$2$ & $549$ \\
		\end{tabular}
		\quad\quad
		\begin{tabular}[t]{cc}
			\multicolumn{2}{c}{$d=4$} \\
			\hline
			$d_0$ & $h^0$ \\
			\hline
			$5,6,7$ & $608$ \\
			$4$ & $625$ \\
			$3$ & $669$ \\
			$2$ & $724$ \\
			\phantom{$1$}
		\end{tabular}
		\quad\quad
		\begin{tabular}[t]{cc}
			\multicolumn{2}{c}{$d=5$} \\
			\hline
			$d_0$ & $h^0$ \\
			\hline
			$6,7,8$ & $733$ \\
			$5$ & $749$ \\
			$4$ & $790$ \\
			$3$ & $843$ \\
			$2$ & $900$
		\end{tabular}
	\end{table}
\end{lem}

\begin{proof}
	The proof is similar to that of Lemma \ref{lem: dimension 10H}. We only give a sketch here. If $d < d_0 \le \frac12(3d + N)$, then all the coefficients $c_{a_0,a_1,a_2}$ have non-negative degree. Thus we have
	$$
	h^0(D_z, 10H_{D_z} - 8F_z) = 125d + 36 + 36N = 125d + 108.
	$$

	If $d_0 = d \ge 4$, then the monomials $x_0^{10}$, $x_0^8y$, $x_0^6y^2$, $x_0^4y^3$, $x_0^9x_1$ no longer appear in the equation of the branch divisor. Thus
	$$
	h^0(D_z, 10H_{D_z} - 8F_z) = 160d - 36d_0 + 31 + 49N = 124d + 129.
	$$
    If $d_0 = d \le 3$, then the monomial $x_0^7x_1y$ also no longer appears in the equation of the branch divisor. Thus
    $$
    h^0(D_z, 10H_{D_z} - 8F_z) = 123d + 30 + 51N = 123d + 132.
    $$
    
    Now we are left with the cases with $d_0 < d$, for $d=3,4,5$. We treat each value of $d$ separately.
    
    \noindent\textbf{Case $d=5$}.
    If $d_0 = 4$, then we also lose $x_0^2y^4$, $x_0^7x_1y$ and $x_0^5x_1y^2$. Thus
	$$
	h^0(D_z, 10H_{D_z} - 8F_z) = 170d - 48d_0 + 28 + 52N = 170d - 60.
	$$
	If $d_0 = 3$, then we lose $x_0^8x_1^2$. Thus
	$$
	h^0(D_z, 10H_{D_z} - 8F_z) = 174d - 54d_0 + 27 + 54N = 174d - 27.
	$$
	If $d_0 = 2$, then we also lose $x_0^6x_1^2y$ and $x_0^3x_1y^3$. Thus
	$$
	h^0(D_z, 10H_{D_z} - 8F_z) = 177d - 60d_0 + 25 + 55N = 177d + 15.
	$$
    
    \noindent\textbf{Case $d=4$}. 
    If $d_0 = 3$, then we lose $x_0^2y^4$, $x_0^7x_1y$, $x_0^5x_1y^2$ and $x_0^8x_1^2$. Thus
	$$
	h^0(D_z, 10H_{D_z} - 8F_z) = 174d - 54d_0 + 27 + 54N = 174d - 27.
	$$
	If $d_0=2$, then we also lose $x_0^6x_1^2y$. Thus 
    $$
    h^0(D_z, 10H_{D_z} - 8F_z) = 176d - 58d_0 + 26 + 55N = 176d + 20.
    $$
    
    \noindent\textbf{Case $d=3$}. Then $d_0 = 2$, and we lose $x_0^2y^4$, $x_0^5x_1y^2$ and $x_0^8x_1^2$.
    Thus 
    $$
    h^0(D_z, 10H_{D_z} - 8F_z) = 174d - 54d_0 + 27 + 54N = 174d + 27.
    $$
    This concludes the proof.  
\end{proof}

Using the dimensions of $H^0(D_z, 10H_{D_z}-4NF_z)$ and $\Aut D_z$ computed by the preceding lemmas and the formula \eqref{eq: dimension formula}, we get

\begin{prop} \label{prop: dimension moduli} 
	Suppose that $N \le 1$, $d \ge 3$ or that $N = 2$, $d \ge 6$. Then the modular family $\M^N_d(d_0)$ is unirational and its dimension $\Delta^N_d(d_0)$ equals
    $$
    \begin{cases}
        122d + 25 + 36N, & \text{ if } d_0 = \frac32d + \frac12N;\\ 
		119d + 2d_0 + 26 + 35N, & \text{ if } d < d_0 < \frac32d + \frac12N \text{ or } d_0=d, N=0;\\ 
		121d + 23 + 44N, & \text{ if } d_0 = d, N>0; \\
		147d - 26d_0 + 22 + 44N, & \text{ if } \frac78d + \frac38 N \le d_0 < d;\\ 
		154d - 34d_0 + 21 + 47N, & \text{ if } \frac56d + \frac13 N \le d_0 < \frac78d + \frac38 N;\\ 
		159d - 40d_0 + 20 + 49N, & \text{ if } \frac34d + \frac14 N \le d_0 < \frac56d + \frac13 N;\\ 
		162d - 44d_0 + 19 + 50N, & \text{ if } \frac23d + \frac13 N \le d_0 < \frac34d + \frac14 N;\\ 
		166d - 50d_0 + 18 + 52N, & \text{ if } \frac12d + \frac14 N \le d_0 < \frac23d + \frac13 N;\\ 
		168d - 54d_0 + 17 + 53N, & \text{ if } \frac12d \le d_0 < \frac12d + \frac14 N;\\
		169d - 56d_0 + 16 + 53N, & \text{ if } \frac14d + \frac14 N \le d_0 < \frac12d. 
    \end{cases}
    $$
\end{prop}

\begin{prop} \label{prop: dimension moduli N=1 d=2} 
	Suppose that $N = 1$ and $d = 2$. Then the modular family $\M^N_d(d_0)$ is unirational and has dimension	
    \[
	\Delta^1_2(d_0) = 
	\begin{cases}
		305, & \text{ if } d_0 = 3;\\
		309, & \text{ if } d_0=2.\\ 
	\end{cases}
	\]
\end{prop}

\begin{prop} \label{prop: dimension moduli N=2} 
	Suppose that $N = 2$ and $d = 2, 3, 4$ or $5$. Then the modular family $\M^N_d(d_0)$ is unirational and has dimension listed in the following tables:
	
	\begin{tabular}[t]{cc}
		\multicolumn{2}{c}{$d=2$} \\
		\hline
		$d_0$ & $\Delta^N_d(d_0)$ \\
		\hline
		$4$ & $341$ \\
		$3$ & $340$\\
		$2$ & $357$\\
	\end{tabular}
	\quad\quad
	\begin{tabular}[t]{cc}
		\multicolumn{2}{c}{$d=3$} \\
		\hline
		$d_0$ & $\Delta^N_d(d_0)$ \\
		\hline
		$5$ & $463$ \\
		$4$ & $461$ \\
		$3$ & $476$ \\
		$2$ & $520$ \\
	\end{tabular}
	\quad\quad
	\begin{tabular}[t]{cc}
		\multicolumn{2}{c}{$d=4$} \\
		\hline
		$d_0$ & $\Delta^N_d(d_0)$ \\
		\hline
		$7$ & $585$ \\
		$6$ & $584$ \\
		$5$ & $582$ \\
		$4$ & $596$ \\
		$3$ & $636$ \\
		$2$ & $687$
	\end{tabular}
	\quad\quad
	\begin{tabular}[t]{cc}
		\multicolumn{2}{c}{$d=5$} \\
		\hline
		$d_0$ & $\Delta^N_d(d_0)$ \\
		\hline
		$8$ & $707$ \\
		$7$ & $705$ \\
		$6$ & $703$ \\
		$5$ & $716$ \\
		$4$ & $753$ \\
		$3$ & $802$ \\
		$2$ & $855$
	\end{tabular}
\end{prop}

In the above propositions, $d_0$ is assumed to be an integer, but it is natural to view $\Delta^N_d$ as a function in one real variable (see Figure \ref{figure!dimension-graph} on page \pageref{figure!dimension-graph} for an example). From this point of view, we have the following proposition.

\begin{prop} \label{prop: delta function}
	Suppose that $N \le 1$, $d \ge 3$ or that $N = 2$, $d \ge 6$. Then there exists a piecewise linear real-valued function 
	\[
	\Delta^N_d\colon \left[ \frac14d + \frac14N, \frac32d + \frac12N \right]  \rightarrow \RR
	\]
	whose component linear functions are given in Proposition \ref{prop: dimension moduli} such that
	\begin{itemize}
		\item [(i)] the set of discontinuities of $\Delta^N_d$ consists of the points $d_0 = \lambda_1 d + \lambda_2 N$, where the set of pairs $(\lambda_1, \lambda_2)$ is
		$$
		\left\{ \left(\frac12, 0\right), \left(\frac12, \frac14\right), \left(\frac23, \frac13\right), \left(\frac34, \frac14\right), \left(\frac56, \frac13\right), \left(\frac78, \frac38\right), \left(1, 0\right), \left(\frac32, \frac12\right) \right\};
		$$
		
		\item [(ii)] $\Delta^N_d$ is linear in each connected component of the domain of continuity;
		
		\item [(iii)] for each integer $d_0$ in the domain of $\Delta^N_d$, we have
		$$
		\dim \M^N_d(d_0) = \Delta^N_d(d_0).
		$$
	\end{itemize}
	Moreover,
	\begin{enumerate}
		\item the restriction of $\Delta^N_d$ to ${\left[\frac14d, d\right]} \cap \NN$ when $N = 0$, and to ${\left[\frac14d + \frac14 N, d + 1\right]} \cap \NN$ when $N > 0$, is strictly decreasing;
		
		\item the restriction of $\Delta^N_d$ to ${\left[d, \frac32d\right]} \cap \NN$ when $N = 0$, and to ${\left[d + 1, \frac32d + \frac12N\right]} \cap \NN$ when $N > 0$, is strictly increasing;
		\item we have 
		$$
		\Delta^N_d\left(\frac32d + \frac12 N\right) = \Delta^N_d \left( \frac{25d - 3 + 8N}{26} \right)
		$$
		when $N = 0$, or when $N = 1$ and $d > 5$, or when $N = 2$ and $d > 13$. 
	\end{enumerate}
\end{prop}

\begin{proof}
	Statements (ii) and (iii) just follow from the definition of the function $\Delta^N_d$. For (i), we only prove the case when $N = 0$ and leave the case $N > 0$ to the interested reader. When $N = 0$, we do not need to consider $\lambda_2$, and the discontinuity result just follows from the table below:
	\[
    \renewcommand{\arraystretch}{1.2}
	\begin{array}{|c|ccccccc|}
		\hline
		\lambda_1 & \frac12 &\frac23&\frac34&\frac56&\frac78&1&\frac32\\
		\hline 
		\Delta^0_d(\lambda_1 d) - \lim_{x \rightarrow {\lambda_1 d}^-}\Delta^0_d(x) & 2 &1&1&1&1&4&-1\\
		\hline
	\end{array}
	\]
	
	We emphasize here that both monotonicity statements (1) and (2) do not concern the function $\Delta^N_d$ as a whole, but only its restriction to the natural numbers. Indeed, such statements do not generalize to the whole function $\Delta^N_d$, exactly because of the points of discontinuity. Again, we only prove the case when $N = 0$ and leave the $N > 0$ case to the interested reader.
	
	To prove (2), we only need to check the case when $\frac32d$ is an integer, i.e., $d$ is even. In this case, $\Delta^0_d(\frac32 d) - \Delta^0_d(\frac32 d-1) = 1$. Thus the statement follows.
	
	To prove (1), note that from the definition, the slope of $\Delta^0_d$ is at least $-26$ for $d_0\le d$. Thus using the above discontinuity table, for any $x \in \RR$ with $\frac14 d \le x - 1 < x \le d$, we always have 
	$$
	\Delta^0_d(x - 1) - \Delta^0_d(x) \ge 26 - (2 + 1 + 1 + 1 + 1 + 4) = 16.
	$$
	Hence statement (1) follows.
	
	Finally, we prove (3). Indeed, when $N = 0$, we have $\frac78 d \le \frac{25d - 3}{26} < d$. Thus
	$$
	\Delta^N_d \left( \frac{25d - 3}{26} \right) = 147d - 26 \cdot \frac{25d - 3}{26} + 22 = 122d + 25.
	$$
	When $N = 1$ and $d > 5$ or when $N = 2$ and $d > 13$, we have $ \frac{25d - 3 + 8N}{26} < d$. Thus, as $\frac78d + \frac38 N \le \frac{25d - 3 + 8N}{26} $,
	\begin{align*}
		\Delta^N_d \left( \frac{25d - 3 + 8N}{26} \right) & = 147d - 26 \cdot \frac{25d - 3 + 8N}{26} + 22 + 44N \\
		& = 122d + 25 + 36N.
	\end{align*}
	The proof is completed.
\end{proof}

\subsection{The moduli space $\M_{K^3, p_g}$}
We can now prove the description of the moduli space of threefolds on the refined Noether line with $p_g \geq 5$.

Write $V^N_{d}(d_0) = \Phi^N_{d, d_0}(\M^N_{d}(d_0))$. Since $\Phi^N_{d, d_0}$ is always finite-to-one, we have $\dim V^N_{d}(d_0) = \Delta^N_d (d_0)$. Recall that $d$ is a deformation invariant, so if the closures of $V^N_d(d_0)$ and $V^N_{d'}(d'_0)$ intersect, then $d = d'$.

We have the following theorem when $N = 0$.

\begin{thm} \label{thm: moduli N=0}
	For each $d \ge 3$, the moduli space $\M_{K^3, p_g}$ of the canonical threefolds with $p_g = 3d-2$ and $K^3=4d-6$ stratifies as the disjoint union of the unirational strata $V^0_d(d_0)$, where $d_0 \in \NN$ and $\max \left\{ \frac14 d, 2 \right\} \le d_0 \le \frac32 d$. Moreover,
	\begin{enumerate}
		\item $V^0_d\left( \left\lfloor \frac32 d \right\rfloor \right)$ is dense in an irreducible component of $\M_{K^3, p_g}$.
		\item If $d_0 \ge d$, then $V^0_d(d_0)$ is contained in the closure of $V^0_d \left(\left\lfloor \frac32 d \right\rfloor \right)$.
		\item If $d_0 \le \frac{25d-3}{26}$, then $V^0_d(d_0)$ is dense in an irreducible component of $\M_{K^3, p_g}$.
	\end{enumerate}
\end{thm}

\begin{proof}
	Since $d_0 \ge 2$, the unirational subvarieties $V^0_1(d_0)$ stratify $\M_{K^3, p_g}$. Part (1) is \cite[Proposition 2.2]{CP}. Part (2) has been proved in \cite[Proposition 2.2 and 2.4]{CP} borrowing a technique from \cite{Pignatelli}.
	
	It remains to prove (3). Arguing by contradiction, we assume the existence of an integer $d_0 \le \frac{25d-3}{26}$  such that $V^0_d (d_0)$ is contained in the closure of $V^0_d (d'_0) $ for some $d_0' \neq d_0$. In other words, for each $X = X(d, 0; d_0)$ we have a flat family $\mathcal {X} \rightarrow \Lambda$ over a small open disc $\Lambda$ with central fibre $X$ and general fibre of type $(d, 0; d'_0)$.
	
	We claim that $d'_0 > d_0$. In fact, by Proposition \ref{prop: canonical image}, the canonical image of $X$ is birationally a Hirzebruch surface $\FF_{3d-2d_0}$. It follows that the relative canonical sheaf $\omega_{{\mathcal X}/\Lambda}$ induces a rational map ${\mathcal X}/\Lambda \dashrightarrow {\mathcal F}/\Lambda$, where ${\mathcal F}/\Lambda$ is a flat family of Hirzebruch surfaces, with central fibre isomorphic to ${\mathbb F}_{3d-2d_0}$ and general fibre isomorphic to ${\mathbb F}_{3d-2d'_0}$. This implies that $d_0' > d_0$.
	
	On the other hand, if $V^0_d(d_0)$ is contained in the closure of $V^0_d(d'_0)$, then $\Delta^0_d(d_0) < \Delta^0_d(d'_0)$, which by Proposition \ref{prop: delta function} implies $d_0' < d_0$, a contradiction. This completes the proof.
\end{proof}

We have the following theorem when $N = 1$.

\begin{thm} \label{thm: moduli N=1}
	For each $d \ge 2$, the moduli space $\M_{K^3, p_g}$ of the canonical threefolds with $p_g = 3d - 1$ and $K^3 = 4d - 6 + \frac16$ stratifies as the disjoint union of the unirational strata $V^1_d(d_0)$, where $d_0 \in \NN$ and $\max \left\{ \frac14 d + \frac14, 2 \right\} \le d_0 \le \frac32 d + \frac12$. Moreover,
	\begin{enumerate}
		\item $V^1_d\left( \left\lfloor \frac32 d + \frac12 \right\rfloor \right)$ is dense in an irreducible component of $\M_{K^3, p_g}$.
		\item If $d_0 \ge d + 1$, then $V^1_d(d_0)$ is contained in the closure of $V^1_d \left(\left\lfloor \frac32 d + \frac12 \right\rfloor \right)$.
		\item If $d > 6$ and $d_0 \le \frac{25d + 5}{26}$ or if $d_0 \le d \le 6$, then $V^1_d(d_0)$ is dense in an irreducible component of $\M_{K^3, p_g}$.
	\end{enumerate}
\end{thm}

\begin{proof}
	The proofs of Part (1) and (2) are identical to those of Theorem \ref{thm: moduli N=0}. For (3), if $d > 6$, using Proposition \ref{prop: delta function}, we may apply the same argument in the proof of Theorem \ref{thm: moduli N=0} here. If $d_0 \le d \le 6$, by Proposition \ref{prop: delta function} for $d \ge 3$ and Proposition \ref{prop: dimension moduli N=1 d=2} for $d = 2$, we always have $\Delta^1_d(d_0) \ge \Delta^1_d(d) > \Delta^1_d \left(\left\lfloor \frac32 d + \frac12 \right\rfloor \right)$. Thus the proof of Theorem \ref{thm: moduli N=0} also applies here.
\end{proof}

The same argument gives the following theorem when $N = 2$.

\begin{thm} \label{thm: moduli N=2}
	For each $d \ge 2$, the moduli space $\M_{K^3, p_g}$ of the canonical threefolds with $p_g = 3d$ and $K^3 = 4d - 3$ stratifies as the disjoint union of the unirational strata $V^2_d(d_0)$, where $d_0 \in \NN$ and $\max \left\{ \frac14 d + \frac12, 2 \right\} \le d_0 \le \frac32 d + 1$. Moreover,
	\begin{enumerate}
		\item $V^2_d\left( \left\lfloor \frac32 d \right\rfloor + 1 \right)$ is dense in an irreducible component of $\M_{K^3, p_g}$.
		\item If $d_0 \ge d + 1$, then $V^2_d(d_0)$ is contained in the closure of $V^2_d \left(\left\lfloor \frac32 d\right\rfloor + 1 \right)$.
		\item If $d > 14$ and $d_0 \le \frac{25d + 14}{26}$ or if $d_0 \le d \le 14$, then $V^1_d(d_0)$ is dense in an irreducible component of $\M_{K^3, p_g}$.
	\end{enumerate}
\end{thm}

Now we are ready to prove the main theorems of the paper.

\begin{proof}[Proof of Theorem \ref{thm: main1}]
    Let $\M_{K^3, p_g}$ be the coarse moduli space parameterizing all canonical threefolds on the refined Noether line with geometric genus $p_g \ge 13$. Write $N = 6K^3 - 8p_g + 20 \in \{0, 1, 2\}$. Then $d = \frac13 (p_g + 2 - N) \ge 5$.
    
    By Theorem \ref{thm: moduli N=0} for $N = 0$ as well as Theorem \ref{thm: moduli N=1} and \ref{thm: moduli N=2} for $N = 1, 2$, when $N = 0$ (resp.~$N = 1, 2$), all $X(d, N; d_0)$ with $d_0 \geq d$ (resp.~$d_0 \ge d + 1$) are in a single irreducible component, while the others may each be a different component. Note that all possible irreducible components are unirational. 
    
    Thus when $N = 0$ (resp.~$N = 1, 2$), an upper bound for the number $\nu_{p_g}$ of irreducible components of $\M_{K^3, p_g}$ is given by the number of integers between $\frac{d}4$ and $d$ (resp.~$\frac{d}4 + \frac N4$ and $d + 1$). This number is $\left \lfloor \frac34 d + 1\right \rfloor=\left \lfloor \frac{p_g+6}4\right \rfloor$ if $N=0$, $\left \lfloor \frac34 d + \frac74\right \rfloor = \left \lfloor \frac{p_g+8}4\right \rfloor$ if $N=1$, and $\left \lfloor \frac34 d + \frac32\right \rfloor = \left\lfloor \frac{p_g+6}4\right \rfloor$ if $N=2$.
    
    Similarly, a lower bound of $\nu_{p_g}$ is obtained by removing all integers lying in the interval $\left(\frac{25d - 3}{26}, d\right)$ when $N = 0$, (resp.~$\left(\frac{25d - 3 + 8N}{26}, d + 1\right)$ when $N = 1, 2$). To sum up, 
    \begin{itemize}
        \item if $N = 0$, then $\left \lfloor \frac34 d + 1\right \rfloor - \left \lfloor \frac{d + 2}{26}\right \rfloor \le \nu_{p_g} \le \left \lfloor \frac34 d + 1\right \rfloor$;

        \item if $N = 1$, then $\left \lfloor \frac34 d + \frac74\right \rfloor - \left \lfloor \frac{d + 20}{26}\right \rfloor \le \nu_{p_g} \le \left \lfloor \frac34 d + \frac74\right \rfloor$;

        \item if $N = 2$, then $\left \lfloor \frac34 d + \frac32\right \rfloor - \left \lfloor \frac{d + 12}{26}\right \rfloor \le \nu_{p_g} \le \left \lfloor \frac34 d + \frac32\right \rfloor$.
    \end{itemize}
    Thus Theorem \ref{thm: main1} (1) and (2) are proved.

    To prove the dimension formula, note that by Proposition \ref{prop: delta function} and Proposition \ref{prop: dimension moduli N=2} for $N = 2$ and $d = 5$, the stratum $V^N_d(d_0)$ with the maximal dimension is the one with $d_0 = \left \lceil \frac{d + N}4\right \rceil$. Thus 
    $$
    \dim \M_{K^3, p_g} = 169d - 56 \left \lceil \frac{d + N}4\right \rceil + 16 + 53N.
    $$
    The proof is completed.
\end{proof}

\begin{proof} [Proof of Theorem \ref{thm: main2}]
    Let $\M_{K^3, p_g}$ be the coarse moduli space parameterizing all canonical threefolds on the refined Noether line with geometric genus $5 \le p_g \le 12$. Let $N = 6K^3 - 8p_g + 20 \in \{0, 1, 2\}$. Then $d = \frac{1}{3} (p_g + 2 - N) \le 4$.

    Let $\nu_{p_g}$ denote the number of irreducible components of $\M_{K^3, p_g}$. If $N = 0$, by Theorem \ref{thm: moduli N=0}, $\nu_{p_g} = d - 2$. If $N = 1$ or $2$, by Theorem \ref{thm: moduli N=1}  or \ref{thm: moduli N=2}, we always have $\nu_{p_g} = d - 1$. Thus Theorem \ref{thm: main2} (1) is proved.

    For the dimension of each irreducible component, if $N = 0$, then $d = 3, 4$, and the irreducible components of the corresponding moduli space $\M_{K^3, p_g}$ are $V_d^0(2)$, $\cdots$, $V_d^0(d - 1)$ and $V^1_d\left( \left\lfloor \frac32 d\right\rfloor \right)$, whose dimensions are computed in Proposition \ref{prop: dimension moduli}.
    If $N = 1$, then $d =2, 3, 4$, and the irreducible components of the corresponding moduli space $\M_{K^3, p_g}$ are $V_d^1(2),\dots,V_d^1(d)$ and $V^1_d\left( \left\lfloor \frac32 d + \frac12 \right\rfloor \right)$, whose dimensions are computed in Proposition \ref{prop: dimension moduli N=1 d=2} and \ref{prop: dimension moduli}. If $N = 2$, then $d = 2, 3, 4$, and the irreducible components of the corresponding moduli space $\M_{K^3, p_g}$ are $V_d^2(2),\dots,V_d^2(d)$ and $V^2_d\left( \left\lfloor \frac32 d \right\rfloor + 1 \right)$, whose dimensions are computed in Proposition \ref{prop: dimension moduli N=2}. Thus Theorem \ref{thm: main2} (2) is proved.
\end{proof}

\begin{remark}
	Though the moduli space of canonical surfaces on the Noether line (i.e., $K^2 = 2p_g - 4$) has at most two irreducible components, recently Rana and Rollenske \cite{RR} studied the moduli space of stable surfaces of general type on the Noether line, also obtaining several components.
\end{remark}

\begin{figure}[ht]
\begin{tikzpicture}[scale=0.7,xscale=0.4,yscale=0.1,domain=36:390] 
    \node at (36,270-5) {$\frac32$};
    \node at (24,270-5) {$1$};
    \node at (21,270-5) {$\frac78$};
    \node at (20,270-5) {$\frac56$};
    \node at (18,270-5) {$\frac34$};
    \node at (16,270-5) {$\frac23$};
    \node at (12,270-5) {$\tfrac12$};
    \node at (6,270-5) {$\tfrac14$};
    \draw [decorate,decoration={brace,amplitude=5pt}]
     (6,270) -- (22.3,270) node[midway,yshift=1em]{\tiny not smoothable};
    \draw [decorate,decoration={brace,amplitude=5pt}]
     (24,270) -- (36,270) node[midway,yshift=1em]{\tiny nonsingular};
    \draw [decorate,decoration={brace,amplitude=5pt}]
     (22.3,270) -- (24,270) node[midway,yshift=1em]{\tiny ?};

    \draw[->] (6,270) -- (38,270) node[right] {$\tfrac{d_0}{d}$}; 
    \draw[->] (6,270) -- (6,390) node[above] {$\Delta^0_{24}$};
    \draw[color=black] (6,380) -- (12,340) node[pos=0.6,above=10] {$\scriptstyle cE_8$};
    \draw[color=black] (12,342) -- (16,318) node[pos=0.7,above=10] {$\scriptstyle cD_6$};
    \draw[color=black] (16,320) -- (18,310) node[pos=0.8,above=5] {$\scriptstyle cA_4$};
    \draw[color=black] (18,311.5) -- (20,302.5) node[pos=0.8,above=5] {$\scriptstyle cA_3$}; 
    \draw[color=black] (20,303.25) -- (21,299.25);
    \draw[color=black] (21,300) -- (24,292);
    \draw[color=black] (24,293) -- (36,297);
    \draw[->] (18,298) to[bend right] (20,301);
    \node at (17,298) {$\scriptstyle cA_1$};
    \draw[->] (24,303) to[bend right] (22,300);
    \node at (26,303) {\tiny terminal};


    \node at (36,297) {$\scriptstyle *$};
    \node at (35,296.66) {$\scriptstyle *$};
    \node at (34,296.33) {$\scriptstyle *$};
    \node at (33,296) {$\scriptstyle *$};
    \node at (32,295.66) {$\scriptstyle *$};
    \node at (31,295.33) {$\scriptstyle *$};
    \node at (30,295) {$\scriptstyle *$};
    \node at (29,294.66) {$\scriptstyle *$};
    \node at (28,294.33) {$\scriptstyle *$};
    \node at (27,294) {$\scriptstyle *$};
    \node at (26,293.66) {$\scriptstyle *$};
    \node at (25,293.33) {$\scriptstyle *$};
    \node at (24,293.0) {$\scriptstyle *$};
    \node at (23,294.66) {$\scriptstyle *$};
    \node at (22,297.33) {$\scriptstyle *$};
    \node at (21,300) {$\scriptstyle *$};
    \node at (20,303.25) {$\scriptstyle *$};
    \node at (19,307) {$\scriptstyle *$};
    \node at (18,311.5) {$\scriptstyle *$};
    \node at (17,315) {$\scriptstyle *$};
    \node at (16,320) {$\scriptstyle *$};
    \node at (15,324) {$\scriptstyle *$};
    \node at (14,330) {$\scriptstyle *$};
    \node at (13,336) {$\scriptstyle *$};
    \node at (12,342) {$\scriptstyle *$};
    \node at (11,346.66) {$\scriptstyle *$};
    \node at (10,353.33) {$\scriptstyle *$};
    \node at (9,360) {$\scriptstyle *$};
    \node at (8,366.66) {$\scriptstyle *$};
    \node at (7,373.33) {$\scriptstyle *$};
    \node at (6,380) {$\scriptstyle *$};
    \draw[dashed,color=black] (36,297) -- (22.4,297) node[below,pos=0] {};
    \draw[dashed,color=black] (22.3,270) -- (22.3,297) node[below,pos=0] {};
    \draw[->] (22.3,270-9) to (22.3,270-1);
    \node at (22.3,270-13) {$\scriptstyle \frac{597}{26}$};
  \end{tikzpicture}
\caption{Dimension of modular families for $d=24$, $N=0$}\label{figure!dimension-graph}
\end{figure}

\subsection{Final remark about the strata} \label{subsection: final}
The statement of Theorem \ref{thm: moduli N=0} does not say anything about the strata $V^0_d(d_0)$ with $\frac{25d-3}{26} < d_0 < d$, and there are $\left\lfloor \frac{d + 2}{26} \right\rfloor = \left\lfloor \frac{p_g+8}{78} \right\rfloor$ of them. For these strata, the argument in the proof of Theorem \ref{thm: moduli N=0} leaves two possibilities: either $V^0_d(d_0)$ is dense in an irreducible component of $\M_{K^3, p_g}$ or $V^0_d(d_0)$ is contained in the closure of $V^0_d\left( \left\lfloor \frac32 d \right\rfloor\right)$.

For numerical reasons, there is no such stratum when $p_g \le 69$. The case when $p_g = 70$ (thus $K^3 = 90$ and $d = 24$) is the first case in which we cannot decide if a certain stratum is dense in an irreducible component or not. As an illustration, the dimensions $\Delta^0_{24}(d_0)$ of the relevant strata $V^0_{24}(d_0)$ of the moduli space $\M_{90, 70}$ are given in Figure \ref{figure!dimension-graph}.

In this case, we do not know whether $V^0_{24}(23)$, which has dimension $2952$, is dense in an irreducible component of $\M_{90, 70}$, or lies in the boundary of $V^0_{24}(36)$ whose dimension is $2953$.

Note that similar phenomena occur for the strata $V_d^1(d_0)$ with $\frac{25d + 5}{26}< d_0 \le d$ and the strata $V_d^2(d_0)$ with $\frac{25d + 13}{26}< d_0 \le d$.

\section{(Non-) Simple fibrations in \texorpdfstring{$(1,2)$}{(1, 2)}-surfaces: more examples} \label{section: example}

In this section, we give more examples of simple and non-simple fibrations in $(1,2)$-surfaces. For simple fibrations, we will adopt the notation in \S \ref{section: toric 4fold}.

\subsection{Simple fibrations with $K_X$ not nef} \label{subsection: K non-nef}
Here we give a complete list of regular simple fibrations $X = X(d, N; d_0)$ with $d \ge 0$ whose canonical class is not nef. 

Given such an $X$, by Proposition \ref{prop: KX-nef}, we may assume that $\min\{d_0,d+\frac25N\} < 2$. By Proposition \ref{prop: sing-X}, $\frac{1}{4}(d + N) \le d_0 \le \frac{1}{2}(3d + N)$. Thus $d \le 4$, and we list all possibilities below:
\begin{gather*}
    X(0,0;0)^*,\ X(0,2;1)^*,\ X(0,3;1)^*,\ X(0,4;1),\ X(0,4;2)^*; \\
    X(1,0;1)^*,\ X(1,1;1),\ X(1,1;2)^*,\ X(1,2;1),\ X(1,2;2)^\dagger,\ X(1,3;1); \\
    X(2,0;1),\ X(2,1;1),\ X(2,2;1); \\
    X(3,0;1),\ X(3,1;1); \\
    X(4,0;1).
\end{gather*}
The ones which are not of general type are marked with an asterisk$^*$, and $X(1,2;2)$ is marked with a dagger$^\dagger$ because it has nef $K_X$ (see Example \ref{ex!X(1,2;2)}). The other ten are all $X(d,N;d_0)$ of general type with $K_X$ not nef, because they all contain the curve $\s_0$ and $(K_X \cdot \s_0) < 0$. We list the properties of the canonical model of each example in Table \ref{tab: K not nef}.

\renewcommand{\arraystretch}{1.2}
\begin{table}[ht] \caption{$X(d,N;d_0)$ of general type with $K_X$ not nef}\label{tab: K not nef}
    $$
    \begin{array}{ccccc}
    \hline
    X(d,N;d_0) & p_g(X) & P_2(X) & K_{X_{\textrm{can}}}^3 & \text{Singularities of }X_{\textrm{can}} \\
    \hline
    X(0,4;1) & 2 & 5 & \frac12 & 7 \times \frac12(1,1,1) \\
    X(1,1;1) & 2 & 4 & \frac13 & 2 \times \frac12(1,1,1),\ \frac13(1,2,2) \\
    X(1,2;1) & 3 & 8 & \frac43 & 4 \times \frac12(1,1,1),\ \frac13(1,2,2) \\
    X(1,3;1) & 4 & 12 & \frac83 & 4\times\frac12(1,1,1),\ 2\times\frac13(1,2,2) \\
    X(2,0;1) & 4 & 11 & \frac94 & 2\times\frac12(1,1,1),\ \frac14(1,3,3) \\
    X(2,1;1) & 5 & 15 & \frac{109}{30} & \frac12(1,1,1),\ \frac13(1,2,2),\ \frac15(2,3,4) \\
    X(2,2;1) & 6 & 19 & \frac{61}{12} & 3\times\frac12(1,1,1),\ \frac13(1,2,2),\ \frac14(1,3,3) \\
    X(3,0;1) & 7 & 22 & \frac{85}{14} & \frac12(1,1,1),\ \frac17(3,4,6) \\
    X(3,1;1) & 8 & 26 & \frac{151}{20} & \frac12(1,1,1),\ \frac14(1,3,3),\ \frac15(2,3,4) \\
    X(4,0;1) & 10 & 33 & \frac{301}{30} & \frac12(1,1,1),\ \frac13(1,2,2),\ \frac15(1,4,4) \\
    \hline
    \end{array}
    $$
\end{table}

Among the ten examples, the three with $N=0$ have appeared in \cite[Proposition 6.1]{CP}. We briefly explain the strategy of the calculation via the example $X(3,1;1)$. After a crepant blow-up, $X(3,1;1)$ has a curve of $cE_8$ singularities along $\s_0$. This curve is $K_X$-negative, and there is a non-terminal flip $X \dashrightarrow X^+$ which contracts $\s_0$. The extracted curve $s_0^+$ is a cuspidal rational curve, and $X^+$ has a $\frac15(2,3,4)$ singularity at the cusp, as well as a $\frac14(1,3,3)$ singularity at another point of $\s_0^+$. Then $K_{X^+}$ is ample so $X^+$ is the canonical model. To compute $K_{X^+}^3$, we combine the Riemann--Roch formula of \eqref{eq: Riemann-Roch P2} for $P_2(X)$ with $P_2(X)=h^0(X, 2H+2F)=26$ to get:
$$
K_{X^+}^3 = 2\left(26+3(1-8) - \frac{1\cdot1}{4} - \frac{3\cdot1}{8} - \frac{2 \cdot 3}{10} \right) = \frac{151}{20}.
$$

The case $X(1,1;1)$ is extra-special, because $K_{X^+}$ is not ample after the flip $X\dashrightarrow X^+$. The canonical model is obtained from a divisorial contraction $X^+\to X_{\frac13,2}$, where $X_{\frac13,2}$ is a general hypersurface $X_{16}\subset\PP(1,1,2,3,8)$ with $K^3=\frac13$, $p_g=2$ and singularities $2\times\frac12(1,1,1)$, $\frac13(1,2,2)$. Conversely, given a general $X_{\frac13,2}$, the  $(1,2,2)$-weighted blowup of the $\frac13(1,2,2)$ point gives $X^+$. Thus every $X_{\frac13,2}$ birationally admits a fibration in $(1,2)$-surfaces induced by the canonical pencil. Indeed, a computation shows that $\Delta_1^1(1)=204-14-1=189$. This agrees with the dimension of the moduli space $\M_{\frac13,2}$ computed in \cite{CHJ}.

Note that \cite[Theorem 4.7]{CHJ} shows that the volume of a threefold of general type with $p_g = 6$ and the canonical dimension one is at least $\frac{61}{12}$. The example $X(2,2;1)$ here shows that this bound is optimal (see \cite[Remark 4.8]{CHJ}).

\subsection{Simple fibrations whose associated fourfold is not toric} \label{subsection: N=5}
Here we show that the assumption on $N$ in Theorem \ref{thm: X is hypersurface} is optimal, by exhibiting a regular simple fibration in $(1,2)$-surfaces with $N = 5$ which is not isomorphic to a divisor in a toric fourfold.

Let $\FF$ be the toric fivefold with the weight matrix
$$
\begin{pmatrix}
	t_0 & t_1 & x_0 & x_1 & y_0 & y_1 & z \\
	1 & 1 & -1 & -a & -3 & -4 & -10 \\
	0 & 0 & 1 & 1 & 2 & 2 & 5
\end{pmatrix}
$$
where $a\ge7$ and the irrelevant ideal $(t_0, t_1) \cap (x_0, x_1, y_0, y_1, z)$. Then $\mathbb{F}$ is a $\PP(1, 1, 2, 2, 5)$-bundle over $\PP^1$. 

We are interested in a general complete intersection $X$ in $\mathbb{F}$ defined by two equations of bidegree $(-2, 2)$ and $(-20, 10)$, respectively. For simplicity, suppose that $X$ is general among those with equations of the form
\begin{equation} \label{eq!non-toric-example}
	\begin{split}
		t_0^2y_1 - t_1y_0 & = ex_0^2+e_{1,1}(t_0,t_1)x_0x_1+e_{0,2}(t_0,t_1)x_1^2, \\
		z^2 & = f_{a-7}(t_0,t_1)x_0x_1y_0^4 + y_1^5 + g_{10a-20}(t_0,t_1)x_1^{10},
	\end{split}
\end{equation}
where $e$ is a constant (later we assume that $e \ne 0$), and $e_{i,j}(t_0,t_1)$ are homogeneous of degrees 
$$
\deg e_{1,1}=a-1, \quad \deg e_{0,2}=2a-2.
$$

We claim that $f\colon X \to \PP^1$ is a regular simple fibration in $(1,2)$-surfaces. Indeed, over the open chart $U_0 = \{t_0 \ne 0\}$ of $\PP^1$, the corresponding chart of $\mathbb{F}$ is isomorphic to $U_0\times\mathbb{P}(1,1,2,2,5)$ with coordinates 
\[t'=t_1/t_0,\quad x_0'=t_0x_0,\quad x_1'=t_0^ax_1,\quad y_0'=t_0^3y_0,\quad y_1'=t_0^4y_1,\quad z'=t_0^{10}z.\] 
On this chart, the equations reduce to
\begin{align*}
    y_1' & = t'y_0' + ex_0'^2+e'_{1,1}(t')x'_0x'_1+e_{0,2}(t')x_1'^2 \\
    z'^2 &= f(t')x_0'x_1'y_0'^4+y_1'^5+g(t')x_1'^{10}
\end{align*}
Thus we can use the first equation to eliminate $y_1'$ and get a hypersurface in $U_0\times\PP(1,1,2,5)$ with equation
\[
z'^2=f(t')x_0'x_1'y_0'^4+(t'y_0'+ ex_0'^2+e'_{1,1}(t')x'_0x'_1+e_{0,2}(t')x_1'^2)^5 + g(t')x_1'^{10}.
\]
This is a simple fibration over $U_0$. Indeed, one can check that the only singularity is at the point $(0;0,0,1,0)$ on the fibre over $t'=0$. This is a terminal hyperquotient singularity of type $cA_4/(\mathbb{Z}/2)$ which has local analytic equation
\[(z'^2=x'_0x'_1+t'^5)\subset\frac12(1,1,1,0).\]
This singularity has a local $\mathbb{Q}$-smoothing to five quotient singularities of type $\frac12(1,1,1)$. A similar computation shows that the other chart is also a simple fibration, and that $X$ has no further singularities there.

We will now see that this example shows that the inequality $N\le 4$ in Theorem \ref{thm: X is hypersurface} is sharp, as in this case $N=5$ and the associated fourfold $\mathbf{F}(X)$ is not of the form $\FF(d, N; d_0)$.

Let $D_{x_0}$ be the torus invariant divisor $\{x_0 = 0\}$ on $\FF$. Let $F$ be a fibre of $f$. Set $H = D_{x_0} + F$. By a similar calculation as in \S \ref{subsection: KX}, we know that $K_X = \left((a - 6)F + H\right)|_X$. Moreover, we have
$$
p_g(X) = 3a-9, \quad K_X^3=4a - \frac{29}2.
$$
Thus by \eqref{eq: N}, $N = 6K_X^3 - 8p_g(X) + 20 = 5$.
From another point of view, we have explained above that $X$ has the equivalent of $5\times\frac12(1,1,1)$ singularities.

To show that $\mathbf{F}(X)$ is not of the form $\FF(d, N; d_0)$, reversing the argument at the beginning of the proof of Theorem \ref{thm: X is hypersurface}, by \cite[Example 3.16]{CP}, it is enough to show that the exact sequence \eqref{eq: multiplication exact sequence} does not split.

We then compute
\begin{align*}
	\mathcal{R}_1 & = f_* \omega_{X/\PP^1} = \mathcal{O}_{\PP^1}(a-3)\cdot x_0\oplus\mathcal{O}_{\PP^1}(2a-4)\cdot x_1, \\
	\mathcal{R}_2 & = f_* \omega_{X/\PP^1}^{[2]} = {\left(\Sym^2\mathcal{R}_1 \oplus \mathcal{O}_{\PP^1}(2a-5)\cdot y_0\oplus \mathcal{O}_{\PP^1}(2a-4)\cdot y_1\right)} / \mathcal{J}
\end{align*}
where $\mathcal{J}=\mathcal{O}_{\PP^1}(2a-6)\cdot(t_0^2y_1 - t_1y_0 - \left( ex_0^2+e_{1,1}(t_0,t_1)x_0x_1+e_{0,2}(t_0,t_1)x_1^2) \right)$. Therefore  
\[
\mathcal{E}_2=\mathcal{R}_2/\Sym^2\mathcal{R}_1  \cong {(\mathcal{O}_{\PP^1}(2a-5)\oplus\mathcal{O}_{\PP^1}(2a-4))} / {(t_1,-t_0^2)} \cong \mathcal{O}_{\PP^1}(2a-3).
\]

The exact sequence \eqref{eq: multiplication exact sequence} for $f$ becomes
\begin{equation*}
	0 \to \Sym^2\mathcal{R}_1 \to \mathcal{R}_2 \to \mathcal{E}_2={\mathcal{O}_{\PP^1}(2a-3)} \to 0
\end{equation*}
where the cokernel $\mathcal{E}_2=\mathcal{O}_{\PP^1}(2a-3)$ is generated by $y:=y_0/t_0^2=y_1/t_1$.
Dividing out by the graded ideal generated by $\mathcal{O}_{\mathbb{P}^1}(2a-4)x_1$ as in the proof of Lemma \ref{lem: epsilon_2 splits} gives the simplified exact sequence
\[0\to\Sym^2\mathcal{T}_1\to\mathcal{T}_2\to\E_2\to0\]
where $\Sym^2\mathcal{T}_1=\mathcal{O}_{\mathbb{P}^1}(2a-6)\cdot x_0^2$ and $\mathcal{T}_2$ has the following presentation as the cokernel of the map 
$\varphi=\left(e,\ t_1,\ -t_0^2\right)$:
\[0\to\mathcal{O}_{\mathbb{P}^1}(2a-6) \xrightarrow{\varphi} 
\mathcal{O}_{\mathbb{P}^1}(2a-6)\oplus
\mathcal{O}_{\mathbb{P}^1}(2a-5)\oplus
\mathcal{O}_{\mathbb{P}^1}(2a-4) \to \mathcal{T}_2.\]
Hence if $e\ne0$ then $\mathcal{T}_2=\mathcal{O}_{\mathbb{P}^1}(2a-5)\oplus
\mathcal{O}_{\mathbb{P}^1}(2a-4)$ and then the map $\mathcal{T}_2\to\E_2 =\mathcal{O}_{\PP^1}(2a-3)$ cannot have a right inverse (if $e=0$ the right inverse do in fact exist).

Therefore, if $e\ne0$ then also the exact sequence \eqref{eq: multiplication exact sequence} does not split, which implies that $X$ is not a divisor in a toric variety of the form $\FF(d, N; d_0)$. 

\begin{remark}
    There are families of non-toric simple fibrations for every $N\ge6$. In these cases, the threefolds can be constructed with $N$ distinct $\frac12(1,1,1)$ singularities instead of a $cA/(\mathbb{Z}/2)$ singularity. We do not know if it is possible to avoid the hyperquotient singularity when $N=5$.
\end{remark}

\subsection{Fibrations in \texorpdfstring{$(1,2)$}{(1,2)}-surfaces of index three}\label{subsection: index 3}

Here we use a similar toric method to produce a sequence of canonical threefolds of index three close to the Noether line that are fibred in $(1,2)$-surfaces. This answers a question posed to the third author by Jungkai Chen. 

Choose an integer $a \ge 1$ and define $\FF = \FF(a)$ to be the toric fivefold with weight matrix
\begin{equation}\label{formula!index-3-weightmatrix}
\begin{pmatrix}
    t_0 & t_1 & x_0 & x_1 & y &u & z \\
    1 & 1 & -a & -a & 0 &-1 & 0 \\
    0 & 0 &  1 & 1 & 2 & 3 & 5
\end{pmatrix}
\end{equation}
and irrelevant ideal $I=(t_0,t_1) \cap (x_0,x_1,y,u,z)$. Then $\FF$ admits a natural fibration $f: \FF \to \PP^1$ by the projection to the first two coordinates.

Let $D_{x_0}$ be the torus invariant divisor $\{x_0=0\}$ on $\FF$. Let $F$ be a fibre of $f$. Set $H = D_{x_0} + aF$. Then each of the {\it coordinates} $\rho \in \{t_0,t_1,x_0,x_1,y,u,z\}$ corresponds to a torus invariant irreducible Weil divisor $D_\rho$ in $\FF$ whose class is as follows:
\begin{gather*}
	D_{t_0}=D_{t_1}=F,\quad 
	D_{x_0}=D_{x_1}=H-aF,\\
	D_y=2H,\quad 
	D_u=3H-F,\quad 
	D_z=5H.
\end{gather*}
Note that $ D_y\cap D_u \cap D_z$ is a Hirzebruch surface $\FF_0$.

\begin{prop}\label{prop!omega_F}
We have $\omega_{\FF} \cong \hol_{\FF} (-12H+(2a-1)F)$.
\end{prop}

\begin{proof}
    We have $[K_{\FF}]=-[D_{t_0} + D_{t_1}+D_{x_0} + D_{x_1} +  D_y + D_z]$ by \cite[Thm 8.2.3]{CLS}.
\end{proof}

\begin{lem} \label{lem: intersection numbers for index three}
    The intersection numbers on $\FF(a)$ are
    $$
    (H^4 \cdot F) = \frac1{30}, \quad H^5=\frac{6a+1}{90}.
    $$
\end{lem}

\begin{proof}
    Since the intersection $D_{t_0} \cap D_{x_0} \cap D_y \cap D_u \cap D_z$ is a reduced smooth point, we have 
    $$
    (D_{t_0} \cdot D_{x_0} \cdot D_y \cdot D_u \cdot D_z)= 30 (H^4 \cdot F) = 1.
    $$
    Similarly, since $D_{x_0} \cap D_{x_1} \cap  D_y \cap D_u \cap D_z$ is empty, we have
    $$
    (D_{x_0} \cdot D_{x_1} \cdot D_y \cdot D_u \cdot D_z) =30 H^5 - 10 (6a+1) (H^4 \cdot F) =0.
    $$
    Rearranging and substituting $H^4F=\frac1{30}$ gives $H^5=\frac{6a+1}{90}$.
\end{proof}

Let $X \subset \FF(a)$ be a general complete intersection of two divisors whose respective classes are $3H$ and $10H$. By Bertini's theorem, $X$ is quasi-smooth. 

\begin{prop} \label{prop: index three}
    The threefold $X$ has a unique singular point $p$, a cyclic quotient singularity of type $\frac13(1,2,2)$. Moreover, $X$ is a canonical threefold of index three with
    $$
    p_g(X)=6a, \quad K^3_X=\frac43 p_g(X)-\frac83.
    $$
\end{prop}

\begin{proof}
    Since the two divisors are general, we may assume that their respective equations are, up to a coordinate change,
    $$
    t_0u +\ldots, \qquad z^2+y^5+\alpha_3(t_0,t_1)x_0u^3 + \ldots.
    $$
    Then it is clear that $X$ intersects the singular locus of $\FF$ just at the point $t_0=x_0=x_1=y=z=0$. Thus $X$ has a unique singular point. Using the above few monomials from the equations, since $\alpha_3$ is general, the singularity is of type $\frac13(1,2,2)$. It follows that the index of $X$ is three.

    By the adjunction, $K_X \sim (H + (2a-1)F)|_X$, which implies $p_g(X) = 6a$. In fact, a basis of $H^0(X, K_X)$ is given by the monomials $t_0^dt_1^{3a-1-d}x_j$. By Lemma \ref{lem: intersection numbers for index three}, we have
    \begin{align*}
        K_X^3 & =\left((3H) \cdot (10H) \cdot \left( H + (2a-1)F \right)^3 \right) =30 \left( H^5 + 3(2a-1)(H^4 \cdot F) \right) \\
        & =30 \left(\frac{6a+1}{90} + \frac{3(2a-1)}{30}\right) = 8a-\frac83.
    \end{align*}
    It is easy to check the ampleness of $K_X$ for $a\ge 1$ by the same method used in Lemma \ref{lem: nef and ample cone}.
\end{proof}

Consider the induced fibration $f_0 := f|_X\colon X \to \PP^1$. Then the general fibre $F_0$ is a $(3, 10)$-complete intersection in the weighted projective space $\PP(1,1,2,3,5)$, where the equation of degree $3$ may be used to eliminate the variable of degree $3$. Thus $F_0$ is a canonical $(1,2)$-surface. However, $f_0$ is not a simple fibration. Otherwise, by Proposition \ref{prop: index three}, it would be a simple fibration with $N = 4$, and the residue class of $p_g(X)$ would be $2$ modulo $3$. A contradiction. 

In fact, we have the following more general result about the uniqueness of fibrations in $(1,2)$-surfaces:

\begin{prop} \label{prop: uniqueness of fibration}
    Let $X$ be a threefold with canonical singularities and $p_g(X) \ge 5$. 
    Suppose further that $f_0 \colon X \to \PP^1$ is a fibration in $(1,2)$-surfaces in its relative canonical model. If $\pi \colon X_1 \to X$ is any birational morphism such that $X_1$ is smooth and admits a fibration $f_1 \colon X_1 \to \PP^1$ in $(1, 2)$-surfaces, then $f_1 = f_0 \circ \pi$.
\end{prop}

\begin{proof}
    Let $F_1$ be a general fibre of $f_1$. Given any integer $n$, by tensoring the adjunction short exact sequence with $\mathcal{O}_{X_1}((1-n)F_1)$, we get
    \[0\to \mathcal{O}_{X_1}(K_{X_1}-nF_1) \to 
    \mathcal{O}_{X_1}(K_{X_1}+(1-n)F_1) \to 
    \mathcal{O}_{F_1}(K_{F_1}) \to 0.\]
    Taking the long exact sequence, we obtain 
    $$
    h^0(X, K_{X_1} - nF_1) \ge h^0(X, K_{X_1} + (1-n)F_1) - p_g(F_1).
    $$
    Combining the two inequalities for $n = 1$ and $2$ gives
    $$
    h^0(X_1, K_{X_1} - 2F_1) \ge p_g(X_1) - 2p_g(F_1) \ge 1.
    $$
    Thus $h^0(X, K_X - 2\pi_*F_1) \ge 1$ and $K_X - 2\pi_*F_1$ is effective. 
    
    Since the general fibre $F_0$ of $f_0$ is Gorenstein, we know that $K_{F_0}$ is an ample Cartier divisor. Restricting to $F_0$ and intersecting with $K_{F_0}$ gives $K_{F_0} \cdot (K_X|_{F_0} - 2\pi_*F_1|_{F_0}) \ge0$. 
    It follows that $1 = K_{F_0}^2 \ge 2 \left(K_{F_0} \cdot (\pi_*F_1)|_{F_0} \right)$. We deduce that $(F_0 \cdot \pi_*F_1) = 0$ as $1$-cycles, which implies that $f_1 = f_0 \circ \pi$.
\end{proof}

By Proposition \ref{prop: uniqueness of fibration}, any fibration in $(1, 2)$-surfaces over $\PP^1$ from a birational model of $X$, has $X$ as the relative canonical model. Thus the threefolds constructed in Proposition \ref{prop: index three} do not admit simple fibrations in $(1,2)$-surfaces, even birationally. In particular, this shows that Conjecture \ref{conj} fails for any $\varepsilon > \frac{2}{3}$.

\appendix

\section{On the existence of fibrations in \texorpdfstring{$(1, 2)$}{(1,2)}-surfaces} \label{section: A1}

The main purpose in this appendix is to show that threefolds on the refined Noether line with $p_g \ge 5$ birationally admit a fibration in $(1, 2)$-surfaces, which extends \cite[Proposition 2.1 and 4.6]{HZ} to the general case.

\begin{lem} \label{lem: p_g5}
   Let $X$ be a minimal threefold of general type satisfying one of the following conditions:
   \begin{enumerate}
       \item $p_g(X)=5$ and $K_X^3 < \frac{109}{30}$;
       \item $p_g(X)=6$ and $K_X^3 < \frac{61}{12}$.
   \end{enumerate}
   Then the canonical image $\Sigma\subseteq\PP^{p_g(X)-1}$ of $X$ is a non-degenerate surface of degree $p_g(X)-2$.  Moreover, there exists a minimal threefold $X_1$ birational to $X$ such that $X_1$ admits a fibration $f: X_1 \to \PP^1$ with general fibre $F_1$ a $(1,2)$-surface. 
\end{lem}

\begin{proof}
    Suppose that $X$ satisfies one of the above conditions. By \cite[Theorem 2.4]{Kobayashi} and \cite[Theorem 4.6]{CHJ}, the canonical image $\Sigma$ of $X$ is a surface. Since $\Sigma \subseteq \PP^{p_g(X)-1}$ is non-degenerate, we have $\deg \Sigma \geq p_g(X)-2$.

    The proof in the following is very similar to that of \cite[Proposition 2.1]{HZ}. However, for the reader's convenience, we present the proof in detail. Take a birational modification $\pi: X' \to X$ such that $X'$ is smooth projective and $|M| = \mathrm{Mov}|\rounddown{\pi^*K_X}|$ is base point free. Denote by $\phi_M: X' \to \Sigma$ the morphism induced by $|M|$. Let $X' \stackrel {\psi}\rightarrow \Sigma' \stackrel{\tau} \rightarrow \Sigma$ be the Stein factorization of $\phi_M$. Denote by $C$ a general fibre of $\psi$. By \cite[Theorem 4.1]{CCJ}, $C$ is a smooth curve of genus $2$.

    Let $S \in |M|$ be a general member. By Bertini's theorem, $S$ is a smooth surface of general type and we have
    $$
    M|_S \equiv dC,
    $$
    where $d = (\deg\tau) \cdot (\deg\Sigma)$. Denote by $\sigma: S \to S_0$ the contraction onto its minimal model. 

    \textbf{Step 1.} In this step, we prove that $(\pi^*K_X \cdot C) \geq 1$, $\deg \Sigma = p_g(X) - 2$ and $\deg\tau=1$. 
    
    Note that we have $K_S \geq 2M|_S \equiv 2d C$. In particular, $S$ cannot be a $(1,2)$-surface. By \cite[Lemma 2.4]{CC}, we have $(\sigma^*K_{S_0} \cdot C)\geq 2$. By \cite[Corollary 2.3]{CCJ}, $2\pi^*K_X|_S-\sigma^*K_{S_0}$ is $\QQ$-effective. Thus $(\pi^*K_X \cdot C) \geq 1$.
    
    By the same argument as in the proof of \cite[Theorem 4.2]{CCJ}, we have
    $$
    K_X^3 \geq (\pi^*K_X|_S)^2 \geq \frac{2(2d-1)}{3}.
    $$
    If $p_g(X)=5$, then the assumption (1) implies that $d \le 3$. On the other hand, $d \ge \deg \Sigma \ge 3$. It follows that $\deg \tau = 1$ and $\deg \Sigma = d = 3$. If $p_g(X) = 6$, by the assumption (2), we deduce similarly that $\deg \tau = 1$ and $\deg \Sigma = d = 4$.

    \textbf{Step 2.} In this step, we prove that $\Sigma$ cannot be a Veronese surface. 

    Suppose that $\Sigma \subseteq \PP^{p_g(X)-1}$ is a Veronese surface. Then the only possibility is that $p_g(X)=6$, $\Sigma \cong \PP^2$, and the embedding $\Sigma \subseteq \PP^5$ is induced by the linear system $|2H|$, where $H$ is a line on $\PP^2$. 

    Let $S_H \in \psi^*|H|$ be a general member. By Bertini's theorem again, $S_H$ is a smooth surface of general type. Denote by $\sigma_H: S_H \to S_{H,0}$ the contraction onto its minimal model. Note that we have $\pi^*K_X \geq 2S_H$. By the adjunction formula, we have
    $$
    K_{S_H} = (K_{X'}+S_H)|_{S_H} \geq 3 S_H|_{S_H} \equiv 3C.
    $$
    Thus $S_H$ cannot be a $(1, 2)$-surface. By \cite[Lemma 2.4]{CC}, again we have $(\sigma_H^* K_{S_{H,0}} \cdot C) \geq 2$. By \cite[Corollary 2.3]{CCJ} (take $\lambda=\frac{1}{2}$, $D=K_X$ and $S=S_H$), we have $\pi^*K_X|_{S_H}\geq \frac{2}{3}\sigma_H^*K_{S_{H,0}}$. It follows that $(\pi^*K_X\cdot C)\geq \frac{4}{3}$. We deduce that
    $$
    K_X^3 \geq \left((\pi^*K_X)|_S \right)^2 \geq ((\pi^*K_X)|_S \cdot S|_S) = d(\pi^*K_X \cdot C) \geq \frac{16}{3}>\frac{61}{12},
    $$
    which is a contradiction.

    \textbf{Step 3.} In this step, we construct a relatively minimal fibration from a birational model of $X$ to $\PP^1$.

    By \cite[\S 10]{Nagata}, \textbf{Step 1} and \textbf{Step 2}, there is a Hirzebruch surface $\mathbb{F}_e$ for some $e \ge 0$ and a morphism
    $$
	r: \mathbb{F}_e \to  \PP^{p_g(X) - 1}
	$$
    induced by the linear system $|\mathbf{s} + (e + k)\mathbf{l}|$ such that $\Sigma = r(\FF_e)$. Here $\mathbf{l}$ is a ruling of the natural fibration $p: \FF_e \to \PP^1$, $\mathbf{s}$ is a section of $p$ with $\mathbf{s}^2 = -e$, and $k \in \ZZ_{\ge 0}$ such that $\deg \Sigma = e + 2k$. In particular, $\Sigma$ is normal. Thus $\tau$ is an isomorphism.

    Replacing $X'$ by its birational modification, we may assume that there is a surjective morphism $\varphi: X' \to \FF_e$ such that $\psi = r \circ \varphi$. Thus we obtain a fibration
	$$
	f':= p \circ \varphi: X' \to \FF_e \to \PP^1
	$$
	with a general fibre $F' = \varphi^*\mathbf{l}$. Let $\zeta: X' \dashrightarrow X_1$ be the contraction of $X'$ onto its relative minimal model $X_1$ over $\PP^1$. Up to a birational modification, we may assume that $\zeta$ is a morphism. Then we obtain a relatively minimal fibration
	$$
	f_1: X_1 \to \PP^1
	$$
	with a general fibre $F_1$. Here $\mu:= \zeta|_{F'}: F' \to F_1$ is just the contraction onto the minimal model of $F'$.

    \textbf{Step 4}. In this step, we prove that $F_1$ is a $(1, 2)$-surface.
	
	By \textbf{Step 1} and the assumption that $p_g(X) \ge 5$, we deduce that $e+k \ge \frac{1}{2} \deg \Sigma = \frac{1}{2} p_g(X) - 1 \ge \frac{3}{2}$, i.e., $e + k \ge 2$. Also recall that $M = \varphi^*(\mathbf{s} + (e + k) \mathbf{l})$. Thus $\pi^*K_X - 2F' \ge 0$. By \cite[Corollary 2.3]{CCJ}, $\frac{3}{2}(\pi^*K_X)|_{F'} - \mu^*K_{F_1}$ is $\QQ$-effective.  On the other hand, by the assumption and \textbf{Step 1}, we always have $K_X^3 < \frac{4}{3}d$. Note that
	$$
    K_X^3 \ge d \left((\pi^*K_X) \cdot C \right).
	$$
	It follows  that
	$$
	\left((\mu^*K_{F_1}) \cdot C \right) \le \frac{3}{2} \left((\pi^*K_X)|_{F'} \cdot C \right) <  2.
	$$
	By \cite[Lemma 2.4]{CC}, we conclude that $F_1$ is a $(1,2)$-surface.
	
	\textbf{Step 5}. In this step, we show that $X_1$ is minimal. By \cite[Lemma 3.2, (2) $\Leftrightarrow$ (3)]{CCJ}, it suffices to show that
	$$
	(\pi^*K_X)|_{F'} = (\zeta^*K_{X_1})|_{F'} = \mu^*K_{F_1}.
	$$
	The second equality holds by the adjunction. Thus it reduces to show that
	$$
	(\pi^*K_X)|_{F'} = \mu^*K_{F_1}.
	$$
    By considering the Zariski decomposition of $K_{F'}$, we deduce that $\mu^*(K_{F_1})-\pi^*K_X|_{F'}$ is an effective $\QQ$-divisor. Thus we have
    $$
    1 = K_{F_1}^2 \geq \left(\mu^*K_{F_1} \cdot (\pi^*K_X)|_{F'}\right) \geq (\pi^*K_X|_{F'})^2.
    $$
    By \textbf{Step 1}, we have
    $$
    \left( (\pi^*K_X)|_{F'}\cdot S|_{F'} \right) = (\pi^*K_X \cdot C) \geq 1. 
    $$
    Thus all the above inequalities become equalities. By the Hodge index theorem, we have 
    $$
    (\pi^*K_X)|_{F'} = \mu^*K_{F_1}.
    $$
    Thus the proof is completed.
\end{proof}

\begin{thm} \label{thm: existence of (1,2)-surface fibration}
   Let $X$ be a minimal threefold of general type with $p_g(X)\geq 5$ and on the refined Noether line. Then the canonical image $\Sigma$ of $X$ is a surface. Moreover, there exists a minimal threefold $X_1$ birational to $X$ such that $X_1$ admits a fibration $f: X_1 \to \PP^1$ with general fibre $F_1$ a $(1,2)$-surface. 
\end{thm}

\begin{proof}
    First, by \cite[Theorem 2.4]{Kobayashi}, we have $\dim \Sigma \leq 2$. If $p_g(X) \geq 11$, then $\dim \Sigma = 2$ by \cite[Proposition 4.6]{HZ}. If $5 \leq p_g(X) \leq 10$, then $\dim \Sigma=2$ by \cite[Theorem 4.4 and Theorem 4.5]{CCJ} and \cite[Theorem 4.6]{CHJ}. Therefore, $\Sigma$ is a surface.

    The existence of the fibration structure is guaranteed by \cite[Proposition 2.1]{HZ} when $p_g(X) \ge 7$, and by Lemma \ref{lem: p_g5} when $p_g(X) = 5, 6$. The proof is completed.    
\end{proof}

\section{Singularities on simple fibrations in \texorpdfstring{$(1, 2)$}{(1,2)}-surfaces}\label{appendix: sing}

In this appendix, we classify the singularities of simple fibrations in $(1,2)$-surfaces, by proving a more detailed version of Proposition \ref{prop: sing-X}. This is both a refinement and a generalization of \cite[Proposition 1.6]{CP}, which only treats the case when $N = 0$. We adopt the same notation as in \S \ref{section: toric 4fold}.
\begin{prop} \label{prop: sing-X-appendix} 
	Suppose that $d \ge 0$. Then $X(d, N; d_0)$ exists if and only if
	$$
	\frac14 (d + N) \le d_0 \le \frac12 (3d+N).
	$$
	A general $X(d, N; d_0)$ has $N \times \frac12 (1, 1, 1)$ singularities at isolated points on $\s_2$ and possibly has canonical singularities along $\s_0$. More precisely, 
	\begin{enumerate}
		\item $X(d, N; d_0)$ is quasi-smooth if and only if $d + \frac38 N \le d_0 \le 
		\frac12 (3d+N)$ or $d_0 = \frac78d + \frac38 N$;
		\item $X(d, N; d_0)$ has $8d_0-7d-3N$ terminal singularities (counted with multiplicities) if and only if $\frac78 d + \frac38N \le d_0 < d + \frac38 N$;
		\item $X(d, N; d_0)$ has canonical singularities along $\s_0$, at the general point of $\s_0$ of the type 
		\begin{itemize}
			\item $cA_1$ if and only if $\frac56d + \frac13 N \le d_0 < \frac78d + \frac38 N$;
            
            \item $cA_2$ if and only if one of the following holds:
            \begin{enumerate}
                \item $\frac{1}{4}N\leq d<\frac{1}{2}N$ and $d+\frac{1}{4}N\leq d_0<\frac{5}{6}d+\frac{1}{3}N$;
                \item $d<\frac{1}{4}N$ and $\frac{2}{3}d+\frac{1}{3}N\leq d_0<\frac{5}{6}d+\frac{1}{3}N$;
            \end{enumerate}
            
			\item $cA_3$ if and only if one of the following holds:
            \begin{enumerate}
                \item $d \geq N$ and $\frac34d + \frac14 N \le d_0 < \frac56d + \frac13 N$;
                \item $\frac{1}{2}N\leq d<N$ and $\frac{2}{3}d+\frac{1}{3}N\leq d<\frac{5}{6}d+\frac{1}{3}N$;
                \item $\frac{1}{4}N\leq d<\frac{1}{2}N$ and $\frac{2}{3}d+\frac{1}{3}N\leq d_0<d+\frac{1}{4}N$;
            \end{enumerate}
            
			\item $cA_4$ if and only if $d \geq N$ and $\frac23d + \frac13 N \le d_0 < \frac34d + \frac14 N$;
            
            \item $cD_4$ if and only if $d<N$ and $\frac{3}{4}d+\frac{1}{4}N\leq d_0<\frac{2}{3}d+\frac{1}{3}N$;
            
            \item $cD_5$ if and only if one of the following holds:
            \begin{enumerate}
                \item $\frac{1}{2}N\leq d<N$ and $d\leq d_0<\frac{3}{4}d+\frac{1}{4}N$;
                \item $d<\frac{1}{2}N$ and $\frac{1}{2}d+\frac{1}{4}N\leq d_0<\frac{3}{4}d+\frac{1}{4}N$; 
            \end{enumerate}
            
			\item $cD_6$ if and only if one of the following holds:
            \begin{enumerate}
                \item $d\geq N$ and $\frac12d + \frac14 N \le d_0 < \frac23d + \frac13 N$;
                \item $\frac{1}{2}N\leq d<N$ and $\frac{1}{2}d+\frac{1}{4}N\leq d_0<d$;
            \end{enumerate}
			
            \item $cE_6$ if and only if one of the following holds:
            \begin{enumerate}
                \item $\frac{1}{3}N\leq d<\frac{1}{2}N$ and $d\leq d_0<\frac{1}{2}d+\frac{1}{4}N$;
                \item $d<\frac{1}{3}N$ and $\frac{1}{4}d+\frac{1}{4}N\leq d_0<\frac{1}{2}d+\frac{1}{4}N$;
            \end{enumerate}
            
			\item $cE_7$ if and only if one of the following holds:
            \begin{enumerate}
                \item $d\geq N$ and $\frac12d \le d_0 < \frac12d + \frac14 N$;
                \item $\frac{1}{2}N\leq d<N$ and $\frac{1}{4}d+\frac{1}{4}N\leq d_0<\frac{1}{2}d+\frac{1}{4}N$;
                \item $\frac{1}{3}N\leq d<\frac{1}{2}N$ and $\frac{1}{4}d+\frac{1}{4}N\leq d_0<d$.
            \end{enumerate}
            
			\item $cE_8$ if and only if $d\geq N$ and $\frac14d + \frac14 N \le d_0 < \frac12d$.
		\end{itemize}
	\end{enumerate}
\end{prop}

\begin{proof}
	For simplicity, we denote by $X$ a general member in $|10H - 4NF|$.
	
	We first assume that $d_0 \ge d + \frac25N$, i.e., $e \le d + \frac15 N$. Since $N \ge 0$, $d \ge 0$ and $|a_1-a_0| \le 10$, it follows from \eqref{eq: degree c} that all $c_{a_0,a_1,a_2}$ have non-negative degrees. Thus $|10H - 4NF|$ is base point free, and $X$ is quasi-smooth with $N$ singularities of type $\frac12(1,1,1)$ at isolated points of $\s_2$, corresponding to the $N$ zeros of $c_{0,0,5}$. In particular, $X$ is a regular simple fibration in $(1, 2)$-surfaces.
	
	From now on, we assume that $d_0 < d + \tfrac25N$. By \eqref{eq: degree c}, we have $\deg c_{10,0,0} < 0$ and $\deg c_{0,10,0} \ge 0$. Thus the linear system $|10H - 4NF|$ has base locus $\s_0$, and $X$ has the defining equation:
	$$
	z^2 = c_{0,0,5}y^5 + y(c_{8,0,1}x_0^8 + c_{6,0,2}x_0^6y + c_{4,0,3}x_0^4y^2+c_{2,0,4}x_0^2y^3) + x_1 (c_{9,1,0}x_0^9 + g),
	$$
	where $g = g(t_0, t_1, x_0, x_1, y)$ vanishes along $\s_0$. Now by \eqref{eq: degree c}, we have
	$$
	\deg c_{9,1,0} = N + 5d - 4e, \quad \deg c_{8,0,1} = N + 4d - 4e.
	$$
	If $d_0 \ge d + \frac38 N$, i.e., $e \le d + \frac14 N$, then both $\deg c_{9,1,0} \ge 0$ and $\deg c_{8,0,1} \ge 0$. Since $X$ is general, we may assume that $c_{9,1,0}$ and $c_{8,0,1}$ have distinct roots, so that they do not vanish simultaneously. Thus $X$ has no singularities along $\s_0$ and is therefore quasi-smooth. In particular, $X$ is a regular simple fibration in $(1, 2)$-surfaces.
	
	If $\frac78d + \frac38 N \le d_0 < d + \frac38N$, i.e., $d + \frac14 N < e \le \frac{5}{4}d + \frac14 N$, then $\deg c_{8,0,1} < 0$ and $\deg c_{9,1,0} \ge 0$. Thus $X$ has $\deg c_{9,1,0}=8d_0-7d-3N$ terminal singularities at the points of $\s_0$ where $c_{9,1,0}$ vanishes. These singularities are locally of the form
	\[z^2+y^k+tx_1=0,\]
	where the exponent $k$ is the minimum $2\le k\le 5$ for which $\deg c_{10-2k,0,k} \ge 0$. As a result, $X$ is a regular simple fibration in $(1, 2)$-surfaces. Note that if $d_0 = \tfrac78d + \tfrac38N$, then $\deg c_{9,1,0} = 0$. We may assume that $c_{9,1,0}$ is a non-zero constant. Then $X$ is quasi-smooth in this case.
	
	If $d_0 < \frac78d + \frac38N$, i.e., $e > \frac{5}{4}d + \frac14 N$, then $c_{10,0,0}$, $c_{8,0,1}$ and $c_{9,1,0}$ all have negative degrees. Now $X$ has the defining equation:
    \begin{align*} 
        z^2 & =  c_{0, 0, 5}y^5 + c_{6,0,2}x_0^6y^2+c_{2,0,4}x_0^2y^4 +c_{4,0,3}x_0^4y^3 \\
        & \qquad x_1(c_{8,2,0}x_0^8x_1 + c_{7,1,1}x_0^7y + c_{7,3,0}x_0^7x_1^2 + c_{6,2,1}x_0^6x_1y + c_{5,1,2}x_0^5y^2 + g),
    \end{align*}
	where $g$ vanishes at $\s_0$ with multiplicity at least $3$. Thus $X$ is singular along $\s_0$. Here we list the critical coefficients with their degrees according to \eqref{eq: degree c}:
	$$
	\deg c_{7, 1, 1} = N + 4d - 3e, \quad \deg c_{5,1,2} = N + 3d - 2e, \quad \deg c_{8, 2, 0} = N + 5d - 3e,
	$$
	$$
	\deg c_{6,2,1} = N + 4d - 2e, \quad \deg c_{7,3,0} = N + 5d - 2e, \quad \deg c_{3,1,3} = N + 2d - e,
	$$
    $$
    \deg c_{6,0,2} = N + 3d - 3e, \quad \deg c_{4,0,3} = N + 2d - 2e, \quad \deg c_{2, 0, 4} = N + d - e.
    $$
    
    If $\frac56d + \frac13 N \le d_0 < \frac78d + \frac38 N$, i.e., $\frac54d + \frac14 N < e \le \frac43d + \frac13N$, then the first six critical coefficients are nonzero for $X$. When $c_{6,0,2}$ is nonzero, the local analytic equation is $z^2= c_{8,2,0} x_1^2 + c_{7,1,1}x_1y+c_{6,0,2}y^2$. It is clear that $X$ has $cA_1$ singularities along $\s_0$. When $c_{6,0,2}$ has negative degree, the local analytic equation is $z^2= c_{8,2,0} x_1^2 + c_{7,1,1}x_1y$.  It is then easy to see that $X$ has $cDV$ singularities and has $cA_1$ singularity at the general point of $\s_0$. 
    
    If $d_0 < \frac{5}{6}d + \frac{1}{3}N$, i.e., $e > \frac43d + \frac13N$, then $c_{6,0,2}$ has negative degree. We divide the proof into four cases.
    
    \textbf{Case 1}: $d \ge N$. We first consider the case when $d \ge N$. Note that if $X$ is on the refined Noether line with $p_g \ge 5$, then we always have $d \ge N$. It remains to determine the type of singularities.
	\begin{itemize}
		\item[(a1)] If $\frac34d + \frac14 N \le d_0 < \frac56d + \frac13 N$, i.e., $\frac43d + \frac13N < e \le \frac32d + \frac12N$, then both $c_{7,1,1}$ and $c_{4,0,3}$ have negative degrees. The local analytic equation of $X$ along $\s_0$ is $z^2 = c_{8,2,0}x_1^2 +  c_{5,1,2}x_1y^2 $.  Thus $X$ has $cDV$ singularities and has $cA_3$ singularity at the general point of $\s_0$. 
		
		\item[(b1)] If $\frac23d + \frac13 N \le d_0 < \frac34d + \frac14 N$, i.e., $\frac32d + \frac12N < e \le \frac53d + \frac13N$, then $c_{5,1,2}$ and $c_{2,0,4}$ have negative degrees. The local analytic equation of $X$ along $\s_0$ is $z^2 = c_{8,2,0}x_1^2 + c_{0,0,5} y^5$. Thus $X$ has $cDV$ singularities and has $cA_4$ singularity at the general point of $\s_0$.
		
		\item[(c1)] If $\frac12d + \frac14 N \le d_0 < \frac23d + \frac13 N$, i.e., $\frac53d + \frac13N < e \le 2d + \frac12N$, then $c_{8,2,0}$ has negative degree. The local analytic equation of $X$ along $\s_0$ is $z^2 = c_{6,2,1}x_1^2y + c_{7,3,0}x_1^3 + c_{0,0,5}y^5$.  Thus $X$ has $cDV$ singularities and has $cD_6$ singularity at the general point of $\s_0$.
		
		
		
		
		\item[(d1)] If $\frac12d \le d_0 < \frac12d + \frac14 N $, i.e., $2d + \frac12N < e \le 2d + N$, then $c_{6,2,1}$ has negative degree. The local analytic equation of $X$ along $\s_0$ is $z^2 = c_{7,3,0}x_1^3 + c_{3, 1, 3}x_1y^3 + c_{0,0,5}y^5$. It is easy to check that $X$ has $cE_7$ singularity at the points of $\s_0$ where $c_{7,3,0}$ does not vanish. Locally at the points where $c_{7,3,0}$ vanishes, $X$ is given by the equation $z^2 = tx_1^3 + x_1y^3$. It is not $cDV$, but the relevant affine chart of the crepant blowup is given by 
        $$
        z=t^2z', \quad x_1=tx'_1, \quad y=ty'.
        $$
        The blow-up variety $X'$ is defined locally by $z'^2 = x_1^{'3} + x'_1y'^3$, which is $cDV$. Thus $X$ has at worst canonical singularities along $\s_0$.
        
		\item[(e1)] If $\frac14d + \frac14 N \le d_0 < \frac12d$, i.e., $2d + N < e \le \frac52d + \frac12N$, then $c_{3,1,3}$ has negative degree. The local analytic equation of $X$ along $\s_0$ is $z^2=c_{7,3,0}x_1^3 + c_{0,0,5}y^5$. Thus $X$  has $cE_8$ singularity at the points of $\s_0$ where both $c_{7,3,0}$ and $c_{0,0,5}$ do not vanish. At the points of $\s_0$ where $c_{7,3,0}$ vanishes, $X$ is locally given by the equation $z^2 = tx_1^3 + y^5$. It was proved in \cite[Lemma 1.14]{CP} that this singularity is canonical. At the point of $\s_0$ where $c_{0,0,5}$ vanishes, $X$ is locally given by the equation $z^2 = x_1^3 + ty^5$. We may assign weights $\mathrm{wt}(t,y,x_1,z)=(1,1,2,3)$. The corresponding weighted blow-up of $\pi: X' \to X$ is crepant and $X'$ has at worst $cDV$ singularities (This can be checked by the same method as in (d). We refer to \cite[Theorem (2.11) and Corollary (2.12)]{C3f} or \cite[\S 5.6]{KM} for details).     
	\end{itemize}
    
	\textbf{Case 2}: $\frac{1}{2}N\leq d<N$. Now we consider the case when $\frac{1}{2}N\leq d<N$. 
	\begin{itemize}
		\item[(a2)] If $\frac23d + \frac13 N \le d_0 < \frac56d + \frac13 N$, i.e., $\frac43d + \frac13N < e \le \frac53d + \frac13N$, then both $c_{7,1,1}$ and $c_{4,0,3}$ have negative degrees. Thus the singularities on $X$ are just the same as (a1).
        
		
		\item[(b2)] If $\frac34d + \frac14 N \le d_0 <\frac23d + \frac13 N $, i.e., $\frac53d + \frac13N < e \le\frac32d + \frac12N $, then $c_{8,2,0}$ has negative degree. The local analytic equation of $X$ along $\s_0$ is $z^2= c_{6,2,1}x_1^2y + c_{5,1,2}x_1y^2 + c_{7,3,0}x_1^3$. Thus $X$ has $cDV$ singularities and has $cD_4$ singularity at the general point of $\s_0$.
        
		\item[(c2)] If $d \le d_0 < \frac34d + \frac14 N$, i.e., $\frac32d + \frac12N < e \le d + N$, then $c_{5,1,2}$ has negative degree. The local analytic equation of $X$ along $\s_0$ is $z^2 = c_{6,2,1}x_1^2y + c_{7,3,0}x_1^3 + c_{2,0,4}y^4$. Thus $X$ has $cDV$ singularities and has $cD_5$ singularity at the general point of $\s_0$.
        
		\item[(d2)] If $\frac12d + \frac14 N \le d_0 < d $, i.e., $d+N < e \le 2d + \frac{1}{2}N$, then $c_{2,0,4}$ has negative degree. Thus the singularities on $X$ are the same as (c1).
        
		
		\item[(e2)] If $\frac14d+\frac{1}{4}N \le d_0 < \frac12d + \frac14 N$, i.e., $2d + \frac{1}{2}N < e \le 2d + N$, then $c_{6,2,1}$ has negative degree. The local analytic equation of $X$ along $\s_0$ is $z^2=c_{7,3,0}x_1^3 +c_{3,1,3}x_1y^3+c_{0,0,5}y^5$. Thus $X$  has  $cE_7$ singularity at the general point of $\s_0$, and $X$ has $cDV$ singularity at the point where $c_{7,3,0}$ does not vanish. At the point of $\s_0$ where $c_{7,3,0}$ vanishes, $X$ is locally given by the equation $z^2=tx_1^3+x_1y^3+y^5$. We may assign weights $\mathrm{wt}(t,y,x_1,z)=(1,1,1,2)$. The corresponding weighted blow-up of $\pi: X'\to X$ is crepant and $X'$ has at worst $cDV$ singularities. Thus $X$ has canonical singularities.
    \end{itemize}
      
    \textbf{Case 3}: $\frac{1}{4}N\leq d<\frac{N}{2}$. Now we treat the case when $\frac{1}{4}N\leq d<\frac{N}{2}$.
    \begin{itemize}
        \item [(a3)] If $d+\frac{1}{4}N\leq d_0<\frac{5}{6}d+\frac{1}{3}N$, i.e., $\frac{4}{3}d+\frac{1}{3}N<e\leq d+\frac{1}{2}N$, then $c_{7,1,1}$ has negative degree. The local analytic equation of $X$ along $\s_0$ is $z^2 = c_{8,2,0}x_1^2 +c_{4,0,3}y^3+ c_{5,1,2}x_1y^2$. Thus $X$ has $cDV$ singularities and has $cA_2$ singularity at the general point of $\s_0$. 

        \item [(b3)] If $\frac{2}{3}d+\frac{1}{3}N\leq d_0<d+\frac{1}{4}N$, i.e., $d+\frac{1}{2}N<e\leq \frac{5}{3}d+\frac{1}{3}N$, then $c_{4,0,3}$ has negative degree. Thus the singularities on $X$ are the same as (a1).
        

        \item [(c3)] If $\frac{3}{4}d+\frac{1}{4}N\leq d_0<\frac{2}{3}d+\frac{1}{3}N$, i.e., $\frac53d + \frac13N < e \le\frac32d + \frac12N$, then $c_{8,2,0}$ has negative degree. Thus the singularities on $X$ are the same as (b2). 

        \item [(d3)] If $\frac{1}{2}d+\frac{1}{4}N\leq d_0<\frac{3}{4}d+\frac{1}{4}N$, i.e., $\frac{3}{2}d+\frac{1}{2}N<e\leq 2d+\frac{1}{2}N$, then $c_{5,1,2}$ has negative degree. Thus the singularities on $X$ are the same as (c1). 
    \end{itemize}
    
    \textbf{Subcase 3.1}. $\frac{1}{3}N\leq d<\frac{1}{2}N$.

    \begin{itemize}
        \item [(e3)] If $d\leq d_0<\frac{1}{2}d+\frac{1}{4}N$, i.e., $2d+\frac{1}{2}N<e\leq d+N$, then $c_{6,2,1}$ has negative degree. The local analytic equation of $X$ along $\s_0$ is $z^2=c_{7,3,0}x_1^3 + c_{3,1,3}x_1y^3 + c_{2,0,4}y^4$. Thus $X$ has $cE_6$ singularity at the general point of $\s_0$, and $X$ has $cDV$ singularity at the point where $c_{7,3,0}$ does not vanish. At the point of $\s_0$ where $c_{7,3,0}$ vanishes, $X$ is locally given by the equation $z^2=tx_1^3+x_1y^3+y^4$. We may assign weights $\mathrm{wt}(t,y,x_1,z)=(1,1,1,2)$. The corresponding weighted blow-up of $\pi: X'\to X$ is crepant and $X'$ has at worst $cDV$ singularities. Thus $X$ has canonical singularities.

        \item [(f3)] If $\frac{1}{4}d+\frac{1}{4}N\leq d_0<d$, i.e., $d+N<e\leq \frac{5}{2}d+\frac{1}{2}N$, then $c_{2,0,4}$ has negative degree. 
        Thus the singularities on $X$ are the same as (e2). 
    \end{itemize}

    \textbf{Subcase 3.2}. $\frac{1}{4}N \leq d < \frac{1}{3}N$.

    \begin{itemize}
          \item[(e3')] If $\frac{1}{4}d+\frac{1}{4}N\leq d_0<\frac{1}{2}d+\frac{1}{4}N$, i.e., $2d+\frac{1}{2}N<e\leq \frac{5}{2}d+\frac{1}{2}N$, then $c_{6,2,1}$ has negative degree. 
          The singularities on $X$ are the same as (e3). 
    \end{itemize}
    
    \textbf{Case 4}. $d < \frac{1}{4}N$. We now treat the case when $d < \frac{1}{4}N$.

    \begin{itemize}
        \item [(a4)] If $\frac{2}{3}d+\frac{1}{3}N\leq d_0<\frac{5}{6}d+\frac{1}{3}N$, i.e., $\frac{4}{3}d + \frac{1}{3}N <e\leq \frac{5}{3}d+\frac{1}{3}N$, then $c_{7,1,1}$ has negative degree. The singularities on $X$ are the same as (a3).

        \item [(b4)] If $\frac{3}{4}d+\frac{1}{4}N\leq d_0<\frac{2}{3}d+\frac{1}{3}N$, i.e., $\frac{5}{3}d+\frac{1}{3}N<e\leq d+\frac{1}{2}N$, then both $c_{8,2,0}$ and $c_{4,0,3}$ has negative degree. The singularities on $X$ are the same as (b2).

        \item [(c4)] If $\frac{1}{2}d+\frac{1}{4}N\leq d_0<\frac{3}{4}d+\frac{1}{4}N$, i.e., $\frac{3}{2}d+\frac{1}{2}N<e\leq 2d+\frac{1}{2}N$, then $c_{5,1,2}$ has negative degree. The singularities on $X$ are the same as (c2).

        \item[(d4)] If $\frac{1}{4}d+\frac{1}{4}N\leq d_0<\frac{1}{2}d+\frac{1}{4}N$, i.e., $2d+\frac{1}{2}N<e\leq \frac{5}{2}d+\frac{1}{2}N$, then $c_{6,2,1}$ has negative degree. 
        The singularities on $X$ are the same as (e3). 
    \end{itemize}
	In each case, $X$ is a regular simple fibration in $(1, 2)$-surfaces. 
    
    Finally, we prove that $X$ is a regular simple fibration in $(1, 2)$-surfaces only when $d_0 \ge \frac{1}{4}(d + N)$. The proof is very similar to that of \cite[Proposition 1.6]{CP}, and we just sketch it here. Let $\mathbf{x} = x_1/x_0$, $\mathbf{y} = y/x_0^2$, $\mathbf{z} = z/x_0^5$ denote local fibre coordinates near $X_t\cap \s_0$ for a general fibre $X_t$ of the fibration $X \to \PP^1$. Using a lemma of Reid \cite[\S 4.6 and \S 4.9]{YPG}, if $X$ has at worst canonical singularities, the equation of $X$ must have monomials of weight $<1$ with respect to each of the weights $\frac12 (1,1,0)$, $\frac13 (1,1,1)$, $\frac14(2,1,1)$ and $\frac16(3,2,1)$.
	With coordinates $(\mathbf{x}, \mathbf{y}, \mathbf{z})$ and weights $\frac14(1,1,2)$, we see that there are $a_1$ and $a_2$ with $a_1 + a_2 < 4$ such that $\deg c_{a_0, a_1, a_2} \ge 0$. Since $a_1 + a_2 < 4$ is equivalent to $a_0 - a_1 \ge 4$, combining this with the fact that $a_0 + a_1 \le 10$, it follows from \eqref{eq: degree c} that 
    $$
	N + 5d - 4e \ge \deg c_{a_0,a_1,a_2} \ge 0,
	$$
	which is equivalent to $d_0 \ge \frac{1}{4}(d + N)$. The proof is completed.
\end{proof}

\bibliography{Ref_Noetherline}
\bibliographystyle{amsalpha}

\end{document}